%% file: main.tex
\documentclass[11pt]{article}

\input{./packages}

\input{./new_commands}

\title{Wrapped sutured Legendrian homology and unit conormal of local 2-braids}
\author{Côme Dattin}

\begin{document}

\maketitle

\begin{center}
\begin{abstract}
\noindent We extend the sutured framework to the case of Legendrians with boundary. Using ideas from Lagrangian Floer theory, we define the cylindrical and the wrapped sutured Legendrian homologies of a pair of sutured Legendrians. They fit together into an exact sequence, and the exact triangle is invariant along an Legendrian isotopy fixed at the boundary. For a single Legendrian, we also define a wrapped version of its Chekanov-Eliashberg dga. Our main example of sutured Legendrian is obtained via the unit conormal construction : a submanifold $N \subset M$, such that $\dd N \subset \dd M$, induces a sutured Legendrian $\La_N \subset ST^*M$, thus we get smooth invariants of manifolds with boundary. As a simple application, we show that if the conormals of two local 2-braids are isotopic (as Legendrians with fixed boundary), then the braids are equivalent.
\end{abstract}
\end{center}
 
In the literature, contact manifolds  are usually assumed to be either closed \cite{Par15_CH_&_virtual}, or the contactisation of some Liouville domain \cite{EES05_LCH_PxR}. A few years ago the notion of {\it  sutured contact manifold}, which generalise both of those settings, was introduced by Colin-Ghiggini-Honda-Hutchings in \cite{CGHH10_Sutures}, where they define the contact homology of such manifolds. In this paper we introduce the notion of {\it sutured Legendrian} and, adapting ideas from Lagrangian Floer theory, we define the {\it cylindrical} and the {\it wrapped  sutured Legendrian homology} of a pair. They fit together into an exact sequence, and the induced exact triangle yields an invariant of sutured Legendrians with fixed boundary. 

We illustrate this construction with an example obtained from the smooth geometry. Starting with a submanifold $N \subset M$, such that $\dd N \subset \dd M$, we construct a sutured contact manifold from $ST^*M$, as in \cite{Dat22_Conormal_stops_Hyperbolic_knots}. Then the unit conormal of $N$ induces a sutured Legendrian, thus the previous exact triangle yields an invariant of $N$, seen as a submanifold with fixed boundary.

As a simple application, we show that if two local 2-braids have unit conormals which are isotopic as Legendrians with fixed boundary, then the two braids are equivalent. This can be seen as a relative version of the analogous result for knots from \cite{She16_Conormal_torus_knot_inrt}\cite{ENS16_Complete_knot_invrt} (see also \cite{Dat22_Conormal_stops_Hyperbolic_knots} for a different point of view), in a much simpler case. The general result, which requires more technology, is only sketched and will be the subject of subsequent work.

\subsection*{Sutured setting}

\paragraph{Sutured contact manifold.} Originally introduced in the topological world by Gabai \cite{Gab83_Folia} as a tool to study 3-manifolds, the definition was adapted to the contact setting in \cite{CGHH10_Sutures}, as a way to handle contact manifolds with convex boundary. Let us briefly present those objects.

Loosely speaking, a (balanced, completed) sutured contact manifold $(V, \xi = \ker \l)$ is a non-compact contact manifold which, at infinity, looks like the contactisation of some Liouville domain $(W, \b)$ :
\[ (V, \l) \simeq (\RR_t \times W, dt+\b) \qquad \qquad \text{ away from a compact.}  \]
More precisely, a sutured contact manifold $(V, \xi)$ is a compact manifold with boundary and corners, such that
\[ \dd V = W_+ \cup [-1,1] \times \Ga \cup W_-, \]
where $W_\pm$ are two Liouville domains with common contact boundary $\Ga$. Moreover the contact structure must satisfy some properties on a neighbourhood of the boundary, see \cref{ssec.Sut_leg} for a detailed definition.
The contact manifold $(\Ga, \l_\Ga) = (\dd W_\pm, \b_\pm)$ is called the {\it suture}, and can be seen as a codimension 2 contact submanifold in $(V, \xi)$.
Following \cite{CGHH10_Sutures}, we can complete $V$ into a non-compact contact manifold, and if $(V, \xi)$ is {\it balanced}, ie $(W_+, \b_+) \simeq (W_-, \b_-)$, we recover the previous loose description.

Note that, contrary to the compact case, we specify a contact form, and not only the contact structure, because we want some control at infinity. As a consequence of those constraints, it is proven in \cite{CGHH10_Sutures} that the contact homology of such a manifold is well-defined (after picking a well-behaved almost complex structure).

Let us give some examples : first, while the trivialisation of a Liouville domain is a trivial case, applying any operation which only affects a compact (for example surgeries) yields a sutured contact manifold.
As a second source of examples, it is shown in \cite{CGHH10_Sutures} that given a contact manifold with  convex boundary, as defined in \cite{Gir91_Convexity}, one can construct a sutured contact manifold, and vice-versa. In particular, the complement of a standard neighbourhood of a Legendrian (in a closed contact manifold) yields a sutured contact manifold.
Finally, and as alluded to above, the unit bundle of a manifold $M$ with boundary has convex boundary. Thus, applying the previous construction, we get a a sutured contact manifold which is, at infinity, the contactisation of $T^*(\dd M)$. In particular, the suture is $ST^*(\dd M)$, see \cref{ssec.ExVarSut} for more details (as well as \cite{Dat22_Conormal_stops_Hyperbolic_knots}).

\paragraph{Sutured Legendrian.} We now add a Legendrian submanifold to the picture, which is also non-compact but will behave nicely at infinity. Roughly speaking, a sutured Legendrian $\La \subset (V, \l)$ is a Legendrian which, at infinity, is {\it cylindrical} and {\it centered} :
\[ (\La, V, \l) \simeq (\RR \times W, \{0\} \times \La_\Ga, dt+\b) \qquad \qquad \text{ away from a compact,} \]
where $\La_\Ga \subset \Ga$ is a Legendrian submanifold, which will also be denoted $\dd \La$.  One should think about this definition as generalising the notion of exact cylindrical Lagrangian in a Liouville domain. Indeed, any such submanifold lift to a sutured Legendrian in the contactisation.
As mentioned above, the main source of example will be applying the unit conormal construction to a submanifold $N \subset M$ with boundary $\dd N \subset \dd M$, such that $N$ intersect $\dd M$ transversely. We get a sutured Legendrian $\La_N$, so that
\[ (\La_N, ST^*M) \simeq (\RR \times \La_{\dd N}, \RR \times T^*(\dd M))  \qquad \qquad \text{ away from a compact,}  \]
where $\La_N \subset ST^*(\dd M) \simeq  \dd T^*(\dd M)$ is the unit conormal of $\dd N$ in $\dd M$.

\subsection*{Main results}

We first adapt the convex-sutured construction from \cite{CGHH10_Sutures} to the Legendrian case : given a Legendrian submanifold $\La \subset (V, \xi)$, such that 
\begin{itemize}
    \item $\La$ is transverse to the boundary of $V$,
    \item $\dd V$ is W-convex (meaning that both its positive and negative region admit a Weinstein structure),
    \item there is a contact vector field $X$, transverse to the boundary, whose associated dividing set contains $\dd \La$,
\end{itemize}
we construct a sutured contact manifold containing a sutured Legendrian. Thus all invariants of sutured Legendrian defined below will be invariants of Legendrians inside contact manifolds with convex boundary (such that $\dd \La$ is in the dividing set).

\subsubsection*{Sutured Legendrian homologies}
We define several invariants of sutured Legendrian, mimicking constructions from Lagrangian Floer theory. The main difference with \cite{CGHH10_Sutures} is that the Legendrian is non-compact, and that Reeb chords will not be contained in a compact. To constrain the relevant holomorphic curves, we rely on several maximum principles for curves with Lagrangian boundary conditions, where both the contact manifolds and the Legendrians inducing the boundary condition are non-compact. Nevertheless, for a well-behaved almost complex structure, we can still control the curves so that Gromov's compactness theorem holds \cite{Gro85_Hol_curves}.

\paragraph{Cylindrical sutured homology.} Given a sutured Legendrian $\La \subset (V,\l)$, consider the dga $\LC(\La, V, \l)$ generated by Reeb chords, or the complex $LC(\La, V, \l)$ if some additional assumptions are satisfied (we assume that the contact form is hypertight so we can ignore Reeb orbits). The differential counts holomorphic disks in the symplectisation of $V$, and with boundary on $\RR \times \La$. Although both $\La$ and $V$ are non-compact, this indeed defines a differential, and the homology is an invariant of the sutured Legendrian.

\begin{MainThm}\label{MThm.Cylindrical_homology}
The homology $\LH(\La, V, \l)$ is well-defined and independent of the choices. Moreover it is invariant under sutured Legendrian isotopy.
\end{MainThm}
We emphasise that, although the boundary $\dd \La$ can move during the isotopy, the sutured Legendrian must stay cylindrical (and in a $t$-level) at infinity.

\paragraph{Wrapped sutured homology.} This version is only defined for a pair $\La_0, \La_1 \subset (V, \l)$ of disjoint sutured Legendrians (here we assume that they are both without contractible chord). Pushing the first one by a contact vector field induced by a function quadratic at infinity (and vanishing on the compact part), we get a Legendrian $\La_0^\W$ which is not sutured anymore. Note that when $V$ is a contactisation, and $\La_0$ is the lift of an exact Lagrangian $L_0$, the wrapped Legendrian $\La_0^\W$ can be seen as the lift of the Lagrangian obtained by pushing $L$ via an Hamiltonian quadratic at infinity.

With this procedure, we have created Reeb chords going from $\La_0$ to $\La_1$, which correspond to chords between the boundaries in the suture :
\[ \C(\La_0^\W, \La_1; V, \l) = \C(\La_0, \La_1; V, \l) \cup \C(\dd \La_0, \dd \La_1; \Ga, \l_\Ga). \]
Denote $\W LC(\La_0, \La_1; V, \l)$ the complex generated by those chords, where the differential counts  holomorphic strips in the symplectisation of $V$, with boundary on $\RR\times (\La_0^\W \cup \La_1)$. While chords can now form a non-compact set, this still defines a differential, and the homology is invariant in the following sense

\begin{MainThm}\label{MThm.Sutured_homology}
The homology $\W LH(\La_0, \La_1; V, \l)$ is well-defined and independent of the choices. Moreover it is invariant under a sutured Legendrian isotopy $(\La_0^u, \La_1^u)_{u \in [0,1]}$ such that
\begin{itemize}
    \item for all $u$, $\La_0^u$ and $\La_1^u$ are disjoints ;
    \item $\La_0^u$ is cylindrical, included in the level $\{t=0\}$ ;
    \item $\La_1^u$ is cylindrical, included in the level $\{t = t_u \leq 0\}$, with $t_u =0$ for $u \in \{0,1\}$.
\end{itemize}
\end{MainThm}
We emphasise that, since $\La_0$ can now pass below $\La_1$ during the isotopy, a $\RR_+$-family of Reeb chords can appear in the non-compact part (for some $u \neq 0,1$).

\paragraph{Exact triangle.} Similarly to the case of Lagrangian Floer theory, the cylindrical complex $LC(\La_0, \La_1; V, \l)$ can be seen as a sub-complex of the wrapped complex. Hence we get an exact triangle 
\[ \ra LH(\La_0, \La_1; V, \l) \lra \W LH (\La_0, \La_1; V, \l) \lra LH^\ext (\La_0, \La_1; V, \l) \ra\]
where $LH^\ext (\La_0, \La_1; V, \l)$ is the homology of the quotient, generated by the chords that we created by wrapping. We expect the following result to holds, generalising \cite{Ekh09_Q-SFT_lin-CH_lagr-FH} to the sutured setting

\begin{MainConj} \label{MConj.Invariance_triangle} 
The homology $LH^\ext (\La_0, \La_1; V, \l)$ is isomorphic to the (shifted) bilinearised homology $LH_{\e_0, \e_1}(\dd \La_0, \dd \La_1; \Ga, \l_\Ga)[1]$, with $H_1(\La)$ coefficients induced by $\dd\La \subset \La$, and where the augmentations are induced by $\La_0, \La_1$.

Moreover, given an isotopy of disjoint sutured Legendrians $(\La_0^u, \La_1^u)_{u \in [0,1]}$, such that the boundaries stay in the level $\{t=0\}$, 
    we get a commutative diagram
\[ \begin{xymatrix} {
     LH(\La_0^0, \La_1^0; V) \ar[r] \ar[d]^\wr & \W LH(\La^0_0, \La_1^0; V) \ar[r]^-{[-1]} \ar[d]^\wr & LH_\e(\dd \La_0, \dd \La_1; \Ga) \ar[r]^-{\d^0} \ar[d]^\wr & \\
      LH(\La_0^1, \La_0^1; V) \ar[r] & \W LH(\La_0^1, \La_1^1; V) \ar[r]^-{[-1]} & LH_\e(\dd \La_0^1, \dd \La_1^1; \Ga) \ar[r]^-{\d^1}  & }
    \end{xymatrix} \]    
    where the first two vertical maps are the isomorphisms of the previous theorems, and the third map is the isomorphism coming from the cobordism induced by the isotopy $\dd \La_0^u \cup \dd \La_1^u \subset (\Ga, \l_\Ga)$. In particular, if the boundary is fixed during the isotopy, this map is induced by the trivial cobordism. Thus, if the augmentations vanish, it is the identity. 
\end{MainConj}

From that point of view, we can think of a sutured Legendrian $\La$ as generalising an (immersed) filling of $(\dd \La, \Ga)$, and the fact that it induces an augmentation will be a particular case of the next result. Moreover we will prove this conjecture for our main example, however this relies heavily on topological restriction specific to the case at hand. For the general case, one would need to adapt the tools developed in \cite{EkOa15_Sympl_Cont_dga} to the sutured setting.

\paragraph{Wrapping the Chekanov-Eliashberg dga.} For a single sutured Legendrian $\La \subset (V, \l)$, we can also define a wrapped theory by adopting a different strategy. By modifying the contact form into $\tilde \l$, we create two families of Reeb chords, in cancelling positions, and once again each family is in bijection with the chords of the boundary $\C(\dd \La, \Ga, \l_\Ga)$. Picking an appropriate almost complex structure, we construct a holomorphic foliation in the symplectisation (near those families), which implies that one of those families form a sub-dga. 

\begin{MainThm}\label{MThm.Wrapped_dga}
Consider a sutured Legendrian $\La \subset (V, \l)$, and assume that $(V, \l)$ is S-sutured and $(\Ga, \l_\Ga)$ is hypertight. Then there is an inclusion of dga
\[\LC(\dd \La; \Ga, \l_\Ga) \ra \LC(\La; V, \tilde \l),\] 
where the homology of the boundary has $H_1(\La)$ coefficients, induced by $\dd \La \subset \La$.
In particular, an augmentation on $\LC(\La;V, \l)$ induces an augmentation on the dga of the boundary.
Moreover, an isotopy $\La^u$ of sutured Legendrians fixed on the boundary yields a commutative diagram
\[\begin{xymatrix}{
    \L C(\dd \La, \Ga, \l_\Ga) \ar[r] \ar[d]_{\Id} & \L C(\La^0, V, \tilde \l) \ar[d]^\wr\\
    \L C(\dd \La, \Ga, \l_\Ga)\ar[r] & \L C(\La^1, V, \tilde \l). }
\end{xymatrix}\]
\end{MainThm}

The wrapped Chekanov-Eliashberg dga is obtained via quotienting by the generators of this sub-dga (when the boundary $\dd \La \subset (\Ga, \l_\Ga)$ is hypertight). Although no actual wrapping happens in that construction, this quotient corresponds, at least morally, to adding chords at infinity which cancel out with the generators of the sub-dga.

\subsubsection*{Invariant of local 2-braids}

As mentioned previously, the unit conormal of a submanifold $N \subset M$, with boundary $\dd N \subset \dd M$, is a rich source of sutured Legendrians. For example, the conormal of a braid $B \subset [-1,1] \times \Si$, where $\Si$ is a fixed surface, is a collection of cylinders, whose boundaries are fibers in $\Ga = ST^*(\Si \sqcup \Si)$.
For very simple braids (namely, the braid has only 2 strands and projects to a small disk in $\Si$), the previous exact triangle is a complete invariant :
\begin{MainThm}\label{MThm.Invariant_complete_2-braids}
If $B, B' \subset [-1,1] \times \Si$ are two local 2-braids, such that $\La_B$ and $\La_{B'}$ are isotopic as sutured Legendrians with fixed boundary, then the braids are equivalent :
\[ \La_B \sim \La_{B'} \text{ rel } \dd \Ra B \sim B'.\]
\end{MainThm}

Note that, if one allows more general isotopies, the unit conormals become equivalents :
\begin{itemize}
    \item as smooth manifolds with fixed boundary, two strands can cross each other, so  we have $\La_B \over{C^\ity}\sim \La_{B'} \text{ rel } \dd$ ;
    \item as sutured Legendrians with moving boundary, we can unbraid everything so the conormals are again isotopic.
\end{itemize}

To prove this result, we compute the exact sequence of the pair of sutured Legendrians, and use homotopical restrictions to show that this is invariant in the sense of \cref{MConj.Invariance_triangle}. The computation relies on the fact that, because our braids are local, we can compute holomorphic curves in the neighbourhood of a constant strand, which is contactomorphic to $J^1(\RR \times S^1)$. In that 1-jet space, the sutured Legendrians become conical at infinity, and wrapping corresponds to a modification creating critical points, which is reminiscent of the construction in \cite{PaRu20_Aug_&_Immers_lagr_fill}.

\subsection*{Relation to other works}

\paragraph{Lagrangian Floer homology.}
The sutured setting is a straight forward generalisation of Floer theory. Indeed, any exact immersed Lagrangian in a Liouville domain $(W, \b)$ lifts to a sutured Legendrian, and intersection points lifts to Reeb chords. Moreover, by \cite{Diri16_Lift} the holomorphic curves in $W$ and in the symplectisation of $V = \RR \times W$ are also in bijection, so the (wrapped) sutured Legendrian homology of a pair of hypertight Legendrians $\La_0, \La_1 \subset V$ is isomorphic (on the chain level) to the (wrapped) positive Lagrangian Floer homology of their Lagrangians projections $L_0, L_1 \subset W$ :
\[ LC(\La_0, \La_1; V) \simeq FC^+(L_0, L_1; W) \qquad \W LC(\La_0, \La_1; V) \simeq \W FC^+(L_0, L_1; W). \]
Note that an holomorphic curve in $W$ going from a intersection point of positive action, to one of negative action, lifts to a banana curve in the symplectisation (ie a curves with two positive asymptotics). Thus to recover the full Lagrangian Floer homology from the sutured setting, one should consider a Rabinowitz-type complex, generated by chords going from $\La_0$ to $\La_1$ as well as from $\La_1$ to $\La_0$, and where the differential counts those bananas.

\paragraph{Conical Legendrians.} In \cite{PaRu20_Aug_&_Immers_lagr_fill}, sutured Legendrians in the contactisation of a standard ball $(D^{2n}, \b)$, where the Liouville vector field is outgoing, are studied using  the contactomorphism $(\RR \times D^{2n}, dt+\b) \simeq (\RR \times T^*\RR^n, dt - p\cdot dq)$. This turns a sutured Legendrian into the 1-jet of a conical function (or, more generally, a mutisection). By adding critical points in cancelling position, and studying Morse flow trees, they obtain the map from \cref{MThm.Wrapped_dga}.

\paragraph{The unit conormal of a knot.} As stated previously, our result for braids is a relative version of an analogous theorem for knots, proved in \cite{She16_Conormal_torus_knot_inrt}\cite{ENS16_Complete_knot_invrt} : given two knots $K, K' \subset \RR^3$, if $\La_K \simeq \La_{K'}$ as Legendrians in $ST^*\RR^3$, then $K \simeq K'$ (up to mirror). The strategy of those papers is to use Legendrian invariants to recover a classical complete invariant (namely, the group knot with its peripheral subgroup). We also mention another point of view due to the author \cite{Dat22_Conormal_stops_Hyperbolic_knots}, based on an idea of \cite{CGHH10_Sutures}. For $K \subset M$ an hyperbolic knot, consider the sutured contact manifold obtained by removing a standard neighbourhood of $\La_K \subset ST^*M$, and applying the convex-sutured construction. Then the (sutured) Legendrian homology of a unit fiber, with its product structure, recover the knot group $\pi_1(M \sms K)$.

\paragraph{Sutured manifolds and stops.} If we restrict to balanced sutured manifolds, there is a correspondence with the stopped perspective introduced in \cite{Syl16_Stops}\cite{GPS17_Fuk_sectors}\cite{AsEk21_Chekanov-Eliashberg_dga_singular}. More precisely, one can go from one of the following object to the other
\begin{itemize}
    \item a balanced sutured contact manifold, whose horizontal boundaries are both $(W, \b)$ ;
    \item a contact manifold with (smooth) convex boundary, whose positive and negative regions are both $(W, \b)$, and with dividing set contactomorphic to the suture ;
    \item a closed contact manifold containing a stop $W$, ie a submanifold of codimension $1$ which is a Liouville domain.
\end{itemize}
The equivalence between the first two objects is due to \cite{CGHH10_Sutures}. For a more detailed discussion regarding the relation with the stopped point of view, we refer to \cite[§2.4]{Dat22_Conormal_stops_Hyperbolic_knots}.

\subsection*{Organisation}

In the first section we recall some usual definitions of symplectic and contact geometry. In \cref{sec.Sutured_Leg_cvx_hsurf}, we recall the sutured framework introduced by \cite{CGHH10_Sutures}, which we adapt to the Legendrian case in \cref{ssec.Sut_leg}. We present the relation with the convex point of view in \cref{ssec.cvx_hsurf} and \cref{ssec.cvx_2_sut}, and in \cref{ssec.ExVarSut} we construct some examples of sutured Legendrians. In \cref{sec.Max_princ}, we describe the almost complex structures that we will be used, as well as the relevant holomorphic curves, and prove the necessary maximum principles. Then in \cref{sec.Sutured_homologies}, we define our sutured Legendrian invariants. The cylindrical and wrapped complexes are described in \cref{ssec.DfnInvariants}, and their invariance is proved in \cref{ssec.Proof_invariance}. The sutured exact sequence is the object of \cref{ssec.Exact_seq}, and the wrapping for the Chekanov-Eliashberg dga is presented in \cref{ssec.dga_wrapped}. Finally in \cref{sec.Invariants_2-braids} we prove our result about braids, starting by recalling standard vocabulary. We then show in \cref{ssec.Braid_1-jet} that we can compute the sutured exact sequence in a 1-jet space, so we can use Morse theory to explicit the differential. The proof of the theorem is finalised in \cref{ssec.proof}, and in \cref{ssec.Sketch_general_case} we sketch the general case.

\paragraph{Acknowledgement :} This paper was part of a PhD thesis at the University of Nantes,
and the author would like to thank his advisor Vincent Colin who introduced him to this
subject, suggested several fascinating projects and was always ready for numerous questions.
 The author is currently supported by the Wallenberg foundation through the
grant KAW 2019.0507, and is grateful for their help.

\tableofcontents 
~\\~\\

\section{Contact manifolds and Liouville cobordisms}

We recall some classical definitions from contact and symplectic geometry. All objects are assumed smooths.

\subsection{Symplectic and contact manifolds}
A {\it symplectic manifold} $(W,\omega)$ is an even-dimensional manifold $W^{2n}$, oriented, with a closed 2-form $\omega$ such that $\omega^n > 0$. In other words $\w$ is non-degenerate : for any $v \in T_x W, (\i_v \w)_x \neq 0$. \\

A {\it contact manifold} $(V, \xi)$ an odd-dimensional manifold $V^{2n+1}$, oriented, endowed with a {\it contact structure} $\xi$, ie a cooriented hyperplane field maximally non-integrable. In other words there exists a 1-form $\lambda$ such that $\ker \lambda = \xi$, and satisfying the contact condition $\lambda \wedge d\lambda^n > 0$. Such a form will be called a contact form.
Note that if $\lambda_0$ is a contact form for $\xi$, all other contact forms can be written $\lambda = f\lambda_0$, with $f : V \ra \RR^*$.

Its {\it Reeb vector field} $R_\lambda$ is defined by $i_{R_\lambda}d\lambda=0$ and $\lambda (R_\lambda)=1$. Note that if $f\lambda$ is another contact form, the Reeb vector fields are related by 
\[R_{f\lambda} = \frac{1}{f^2}\left(f R_{\lambda} + v\right)\] 
where $v \in \xi$ is such that $\i_v d\l_{\rest \xi} = df_{\rest \xi}$.

If $(V, \lambda)$ is contact, its {\it symplectisation} is 
\[\Symp(V) = (\RR \times V, \w = d(e^s\lambda)).\] 
We also define the positive symplectisation $\Symp_+ = \{s \geq 0\}$ and the negative symplectisation $\Symp_-=\{s \leq 0\}$. 
If $(W,\omega = d\beta)$ is an exact symplectic manifold, ie $\w = d\b$ for some $1$-form $\b$, its {\it contactisation} is 
\[\Cont(W) = (\RR_t \times W, \l = dt + \b).\] 
Similarly, the $S^1$-contactisation of $W$ is $\Cont_{S^1}(W) = (S^1_\theta \times W, d\theta + \b)$.

\begin{exs} Here are some classic examples of contact and symplectic manifolds, constructed from a smooth manifold $M$ :
\begin{itemize}
\item the cotangent bundle $(T^*M, \w = d(-p\cdot dq))$ is a symplectic manifold ;
\item  $(J^1(M) = \RR_z \times T^* M, \lambda = dz - p\cdot dq)$ is contact, and $R_\lambda = \dd_z$ (it can also be seen as the contactisation on $T^*M$) ;
\item  the unit cotangent bundle $ST^*M$ endowed with 
\[\xi_{(q, \a)} = \{(Q, A) \in T_{(q, \a)} ST^*M, \a(Q)=0\}\] 
is contact. If we pick a metric $g$ on $M$, there is contactomorphism between $ST^*M$ and $(U_g M, \lambda_g)$, where $(\l_g)_{(q,v)}(Q, *)=g(v,Q)$. The Reeb vector field is $R_{(q,v)} = (v, w)$, where $w$ is determined by the metric, and the Reeb flow lifts the geodesic flow of $g$.
\end{itemize}
\end{exs}

\begin{prop}[Gray's theorem] \label{prop.Moser} Let $(\xi_t)_{t\in [0,1]}$ be a path of contact structures on $V$. Them $\xi_0$ are $\xi_1$ are contactomorphic. More precisely, there exists a (time-dependent) vector field  $X_t$ such that $(\phi_X^t)^*(\xi_t)=\xi_0$.
\end{prop}

A {\it Legendrian} is a $n$-dimensional submanifold $\Lambda^n \subset (V^{2n+1},\xi)$, satisfying $T\Lambda \subset \xi$. In other words, if $\l$ is a contact form for $\xi$, $\l_{\rest T\La}=0$.
A {\it Lagrangian} is a  of $n$-dimensional submanifold $L^n \subset (W^{2n}, \w)$ and such that $\w_{\rest TL} = 0$.

\begin{prop}[Standards neighbourhood] \label{Vois_st_leg} Consider a point $p \in (V, \xi)$. There exists a neighbourhood of $p$ and coordinates $(z, x_i, y_i)$  such that $\xi = \ker(dz - x_i dy_i)$. \\
Let $\Lambda \subset (V, \xi)$ be a Legendrian. Then a neighbourhood of $\Lambda$ is contactomorphic to a neighbourhood of the zero section $0_\Lambda  = \{0\} \times \Lambda  \subset (J^1(\Lambda), \xi_\st)$. \\
Let $L \subset (W, \w)$ be a Lagrangian. Then a neighbourhood of $L$ is symplectomorphic to a neighbourhood of the zero section $0_L = L \times \{0\} \subset (T^*M, \w_\st)$.
\end{prop}

\begin{prop}[Moser for Legendrians] \cite[Theorem 2.41]{Gei06_Contact_Geometry}  \label{prop.Moser_Leg}
Let  $\Lambda_t \subset (V, \xi_t)$ be a path of Legendrians. Then there exists a path of diffeomorphisms $\phi_t$ such that
\[\phi_t^* \xi_t = \xi_0 \text{  and  } \phi_t^*(\Lambda_t)=\Lambda_0 \ (\text{ie } \phi_t(\Lambda_0) = \Lambda_t).\]
\end{prop}

\subsection{Liouville domains}
A {\it Liouville domain} $(W,\beta)$ is a manifold endowed with a 1-form such that :
\begin{itemize}
\item $W$ is a compact manifold (with boundary) ;
\item $(W, d\beta)$ is symplectic ;
\item the Liouville vector field $Y$, defined by $\iota_Y d\beta = \beta$, is positively transverse to the boundary.
\end{itemize}

A {\it Liouville cobordism} $(W,\beta)$ is a manifold endowed with a 1-form such that 
\begin{itemize}
\item $(W, d\beta)$ is symplectic. 
\item The boundary of $W$ splits into $\dd W = \dd_+ W \sqcup \dd_- W$, such that the Liouville vector field is positively (resp. negatively) transverse  to $\dd_+ W$ (resp. $\dd_- W$).
\end{itemize}
We will then say that $W$ is a Liouville cobordism from $(\dd_+ W, \beta_{ \rest\dd_+ W})$ to $(\dd_- W, \beta_{ \rest\dd_- W})$. In particular a Liouville domain is a cobordism from $(\dd W, \beta_{ \rest\dd W})$ to the emptyset. 

\begin{prop} \label{prop.Vois_bord_Liou} Let $(W,\beta)$ be a Liouville cobordism. Then the boundary $(\dd_+ W, \beta_{ \rest\dd_+ W})$ is contact, and there exists a neighbourhood 
\[(\N(\dd_+ W), \b) \simeq \big((-\e, 0]_r \times \dd W, e^r  \beta_{\rest\dd W}\big) \]
In other words, a neighbourhood of the positive boundary is symplectomorphic to (a neighbourhood of the boundary of) the negative symplectisation of $\dd_+ W$. Similarly, a neighbourhood of the  negative boundary is symplectomorphic to a positive  symplectisation.
\end{prop}

The {\it completion} of a Liouville cobordism $(W,\beta)$, denoted $(\hat{W}, \hat{\beta})$, is obtained by gluing the positive (resp. negative) symplectisation of $\dd_+ W$ (resp. $\dd_- W$) via the previous coordinates : $\hat{W} = W \cup (\RR^\pm_r\times \dd_\pm W, e^r \beta_{\rest\dd_\pm W})$.

\begin{rmk}
\cite{Cou12_Symplectom_not_contactom} The symplectisations of discincts contact manifolds can sometimes be symplectomorphic, thus it might be more natural the directly work with completed manifolds.
\end{rmk}

\begin{exs} \label{ex.Liouville_Cobord} 
The disk $(D^{2n}, -x_i dy_i)$ is a Liouville domain, its Liouville vector field is radial and its completion is $\RR^2n$.
For a smooth manifold $M$, its cotangent $(D^*M, -p\cdot dq)$ is a Liouville domain, its Liouville vector field is $p.\dd_p$, and its completion $T^*M$.
\end{exs}

\begin{dfn} \label{dfn.LagrExact} 
Let $L \subset (W, \b)$ be a Lagrangian immersed in a Liouville cobordism. It is exact if it satisfies one of the following equivalent  conditions :
\begin{itemize}
\item $\b_{\rest TL}$ is an  exact 1-form, ie $\b_{\rest TL} = df$, and $f$ is constant at near the boundary
\item $L$ lifts to a Legendrian in the contactisation of $W$.
\end{itemize}
\end{dfn}
The boundary of such a Lagrangian are Legendrians submanifolds $\dd_\pm \subset \dd_\pm W$, and $L$ will be called an exact Lagrangian cobordism from $\dd_+ L$ to $\dd_- L$.

\subsection{From a Legendrian isotopy to an exact cobordism} \label{ssec.ConstrCobordLagr}

Let $(\l_t)_{t \in [0,1]}$ be a path of contact forms on a manifold $V$. Then there exists a Liouville cobordism from $(V, C \l_1)$ to $(V, \l_0)$, for $C \in \RR$ big enough. To construct it, choose a function $\psi : \RR_s \ra [0, 1]$ increasing, vanishing for $s << 0$ and evaluating to $1$ for $s >> 0$. Then if $\psi'$ is small enough, the manifold 
\[(W, \b) = (\RR_s \times V, e^s \l_{\psi(s)})\]
is a Liouville domain. Indeed, 
\[d\b  = e^s (d\l_{\psi(s)} + ds \wedge \l_{\psi(s)} + \psi'(s)ds \wedge \dd_s \l_{\psi(s)} )\]
and the symplectic condition is open.

More generally, if $\Lambda_s \subset (V,\lambda_s), s \in [0,1]$ is a path of Legendrians, we can construct an exact Lagrangian cobordism : 
\begin{lem}
There exists a Liouville cobordism $(W, \b)$, as well as an exact Lagrangian cobordism $L \subset (W, \b)$, going from $\Lambda_S \subset (V, C \l_S)$ to $\Lambda_0 \subset (V, \l_0)$ (for $C$ big enough).
\end{lem}

\begin{proof}
Contrary to the case of contact forms, the graph of a Legendrian isotopy does {\it not} directly define a Lagrangian. 
However we can modify the Liouville form to make it work, as in \cite[Appendix A]{Ekh06_Rational_SFT_Z2_cobord} : if $\La_s \subset (V, \l)$ is induced by the flow of $X_s$, associated to $f_s : V \ra \RR$, then the graph of this path
\[L = \{(s, x), x \in \La_s\} \subset \RR_s \times V \]
in an exact Lagrangian for the Liouville form
\[ \b = e^s(\l_s - df )\]
where $f(s, x)= f_s(x)$. Once again one might have to slow down the isotopy so that this is indeed a Liouville Form. Note we have 
$\b_{\rest TL} = d(e^s f)$. 
\end{proof}

\begin{rmk} Another construction is also possible, using Hamiltonian vector fields, see \cite{Cha06_Lagr_concord_knots} \cite{CCDr16_Posit_Leg_isot_&_Floer_th} for more details. Consider the cobordism $(\RR \times V, e^s \mu_s \l_0)$, where $\mu_s$ is given by \cref{prop.Moser} (as previously, one might have to slow down the isotopy to get a Liouville form). Choose a time-dependent Hamiltonian $H_u$, vanishing at $- \ity$ and given by $e^s f_u(x)$ for $s \gg 0$. Then $\phi_H(\RR \times \La_0)$ is the desired exact cobordism.
\end{rmk}

\subsection{Stein and Weinstein manifolds}
A {\it Weinstein manifold } $(W,\beta,f)$ is a Liouville domain, with a Morse function $f$ such that :
\begin{itemize}
\item  the Liouville vector field $Y$ is a pseudo-gradient of $f$ : $\forall x \notin \Crit(f), df_x(Z(x)) > 0$ ; 
\item the boundary of $W$ is a regular level of $f$.
\end{itemize}
Note that $f$ can be extended to the completion $\hat W$ without adding critical points.

The skeleton of the manifold, defined by  $\Core (W) = \under{\cap}{t> 0} \phi_Y^{-t} (W)$,
is then the union of the  $Y$-stables manifolds of the critical points of $f$, which are isotropic and consequently of dimension smaller than $n$ (see \cite{CiEl12_Stein_Weinstein} and \cite{ElGr91_Convex}). 

\begin{rmk} 
Although the skeleton is defined for any Liouville domain, it can a priori be of codimension 1 as noticed in \cite{McD91_Sympl}.
\end{rmk}

The pair $(\hat W, J_0)$ is a {\it Stein manifold} if $J_0$ is an integrable complex structure, ie if $(W, J_0)$ can be properly embedded in $(\CC^N, i)$. It is equivalent to the existence of a function $\phi : \hat W \ra \RR$ which is proper, bounded below, and strictly pluri-sub-harmonic, ie such that $\Delta_J \phi := - d(d\phi \circ J)$ is a symplectic form.

In particular $\b = - d\phi \circ J$ is a Liouville form, however $J$ might not be adapted in the sens of \cref{ssec.Struct_p-complexe}, because the level $\phi$ are a priori different from the levels of the  Liouville vector field, see \cite[§3.2]{CGHH10_Sutures}.

\begin{rmk}
A Stein manifold is automatically Weinstein : $\b = -d\phi \circ J_0$ is a Liouville form, and the corresponding Liouville vector field is the negative gradient of $\phi$. Conversely, Weinstein manifolds can be deformed into Stein manifolds, see \cite[Theorème 13.9]{CiEl12_Stein_Weinstein}. \end{rmk}

\begin{ex} 
\cite{CiEl13_Flexible_Weinst} The cotangent  $(T^*M, \b = -p \cdot dq + dh, \L = (p - \nabla g,0),f=  h(q) + |p|^2)$ is a Weinstein manifold, where $h$ is a small Morse function on $M$. 
 \end{ex}

\section{Sutured Legendrians and convex hypersurfaces}  \label{sec.Sutured_Leg_cvx_hsurf}

Sutured manifolds were first introduced by Gabai \cite{Gab83_Folia} as a tool to study 3-manifolds, and later extended to the contact setting by Colin-Ghiggini-Hutchings-Honda \cite{CGHH10_Sutures}. In this section we introduce the notion of sutured Legendrian, which further generalise this construction.

\subsection{Sutured contact manifolds} \label{ssec.Sut_leg}

\begin{dfn} \cite{CGHH10_Sutures}
A {\it sutured contact manifold} is a collection $(V, \Ga, \l, \N_0, \psi)$ where ;
\begin{itemize}
    \item the pair  $(V, \ker \l)$ is a  compact, oriented, contact $(2n+1)$-manifold with corners ;
    \item $\Ga$ is $(2n-1)$-submanifold included in $\dd V$ ;
    \item $\N_0$ is a neighbourhood of $\Ga$, and $\psi$ is a diffeomorphism 
    \[\N_0 \over{\psi}{\simeq} (-\e, 0]_\t \times [-1,1]_t \times \Ga\]
    providing coordinates\footnote{As in \cite{CGHH10_Sutures} this product is actually oriented as $[-1,1]_t \times (-\e, 0]_\t \times \Ga$, but we want to think about $t$ (resp. $\t$) as the vertical (resp. horizontal) direction.} ;
\end{itemize}
satisfying the following conditions :
\begin{enumerate}
    \item the boundary of $V$ splits into
        \[\dd V \simeq R_+ \under{\cup}{\{1\} \times \Ga} [-1, 1] \times \Ga \under{\cup}{\{1\} \times \Ga} R_-,\]
    where the corners are exactly the gluing loci ;
    \item in coordinates, $\Ga = \{\t = t =0\}$ and 
        $ \dd V \cap \N_0 = \{t = 1\} \cup \{\t = 0\} \cup \{t=-1\}$,
    where once again the corners of $V$ are exactly the gluing loci ;
    \item \label{item.dfn_sut_R=Liouv} the pair $(R_+, \l_{\rest R_+})$ (resp. $(R_-, \l_{\rest R_-})$), oriented as the boundary of $V$ (resp. with reversed orientation), is a Liouville domain ;
    \item \label{item.dfn_sut_contact_f} on $\N_0$, we have $\l = C.dt + e^\t \l_\Ga$, where $\l_\Ga$ is a contact form on $\Ga$ (independent of $t$ and $\t$, and without term in $dt$ and $d\t$).
\end{enumerate}
 \end{dfn}
The Liouville forms $\l_{\rest R_\pm}$ will be denoted $\b_\pm$, 
and the Reeb vector field of $(V, \l)$ (resp. $(\Ga, \l_\Ga)$ will be $R$ (resp. $R_\Ga$).
When the collection is unambiguous, a sutured manifold will only be noted $(V, \Ga, \l)$.

We list some consequences from this definition which will be used afterwards, for more precise proofs we refer to \cite[§2]{CGHH10_Sutures}.

First of all, on the neighbourhood $\N_0$, the Reeb vector field is $C^{-1} \dd_t$ and the contact structure splits into 
\[   \ker \l = \< \dd_\t, C R_\Ga -  e^\t \dd_t \> \oplus  \ker \l_\Ga.\]

The conditions \cref{item.dfn_sut_R=Liouv} and \cref{item.dfn_sut_contact_f} imply that $(R_\pm, \b_\pm)$ are Liouville domains : since the Liouville vector field is $\dd_\t$, it is outgoing, hence $R_\pm$ have no negative boundary.

Moreover the condition \cref{item.dfn_sut_R=Liouv} imply that the Reeb vector field of $\l$ is positively (resp. negatively) transverve to $R_+$ (resp. $R_-$). Using its flow we obtain neighbourhoods of $R_\pm$ contactomorphic to $((1-\e, 1]_t \times R_+, Cdt + \b_+)$ and $([-1, -1 +\e]_t \times R_-, Cdt + \b_-)$, extending the coordinate $t$ to a neighbourhood of all of $\dd V$. Note that the coordinate $\t$ can be recovered by integrating the Liouville vector field on each $t$-level, since on $\N_0$ it is $\dd_\t$.

\begin{rmk}
If there exists an exact symplectomorphism $(R_+, \b_+) \simeq (-R_-, \b_-)$, the sutured contact manifold will be called {\it balanced}. It could also be described as a manifold contactomorphic to the symplectisation of $(R_+, \b_+)$ near the boundary (or away from a compact, if working with completed manifolds, see \cref{ssec.Completions}).
\end{rmk}

\subsection{Sutured Legendrians}

\begin{dfn}
A {\it sutured Legendrian} is an embedded Legendrian submanifold $\La \subset (V, \l, \Ga)$ which is :
\begin{itemize}
    \item centered, ie its boundary $\dd \La$ is included in the suture $\{\t = t=  0\}$ ;
    \item cylindrical, ie on $\N_0$ we have $\La = (-\e, 0] \times \{0\} \times \dd \La$ for $\e$ small enough.
\end{itemize}
\end{dfn}

Note that the second condition implies that the boundary $\dd \La$ is a Legendrian submanifold of $(\Ga, \l_\Ga)$.
We refer to \cref{fig.Sutured_legendr} for a depiction of the situation.

\begin{figure}[h!]
\center
\def\svgwidth{8cm}
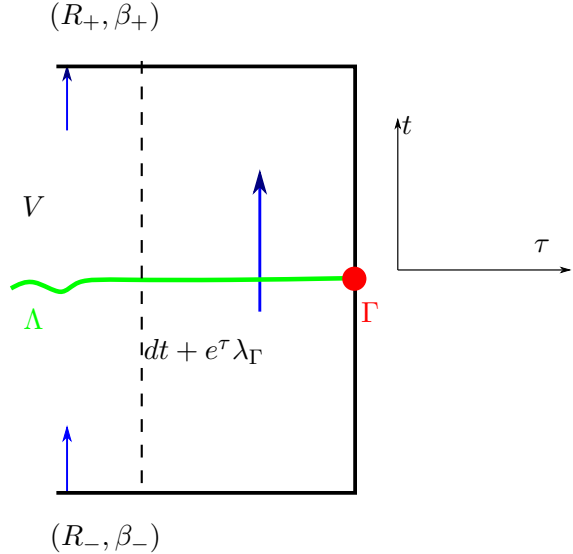
\caption{A sutured Legendrian, in green, inside a sutured contact manifold.}
\label{fig.Sutured_legendr}
\end{figure}

\begin{rmk} $\La$ could also be called a {\it Legendrian filling} of $(\dd \La, \Ga, \l_\Ga)$. We can also relax this definition to allow for several $t$-levels near the boundary, but forbidding Reeb chords in the boundary (ie the projection of $\dd \La \subset [-1,1] \times \Ga$ to $\Ga$ is embedded).
\end{rmk}

Those definitions extend Lagrangian fillings in the following way : for $C$ big enough, an exact, immersed, cylindrical Lagrangian filling $(L, W, \b)$ lifts to a sutured Legendrian 
\[ \La = \{(-f(x), x)\} \subset ([-1, 1] \times W, C dt + \b),\]
where $f$ is a primitive of $\b_{\rest L}$ vanishing on the boundary, and auto-intersection points lift to Reeb chords.

\subsection{Convex hypersurfaces} \label{ssec.cvx_hsurf}

\begin{dfn} Consider $(V, \xi)$ a contact manifold. A vector fields $X$ is {\it contact} if it satisfies one of the following equivalent conditions : 
\begin{itemize}
\item the flow  of $X$ preserves $\xi$, ie $(\phi_X^t)^*\xi = \xi$ ;
\item for any 1-form $\lambda$ of kernel $\xi$, we have $\L_X \lambda = f\lambda$, where $f \in C^\infty(V,\RR)$ ;
\item there exists a 1-form $\lambda$ of kernel $\xi$ such that $\L_X \lambda = f\lambda$, where $f \in C^\infty(V,\RR)$.
\end{itemize}
\end{dfn}

\begin{lem} \label{lem.ContactVf} Let $(V,\xi)$ be a contact manifold. There is a bijection the contact vector fields and the functions $\C^\infty (V,\RR)$. If we choose an adapted contact form, a bijection is given by  
\[ \vphi \in \C^\infty (\RR) \mapsto  X_\vphi = \vphi R_{\lambda} + v,  \]
where $v\in \xi$ is such that $(\i_v d\l)_{\rest \xi} = d\vphi_{\rest \xi}$.
\end{lem}

\begin{rmks} \begin{itemize}
\item If $\tilde{\vphi}$ is a $C^1$-perturbation of $\vphi$, then $X_{\tilde{\vphi}}$ is $C^0$-close from $X_\vphi$ ;
\item $X_\vphi$ is a Reeb vector field iff $\vphi >0$ ;
\item We can interpolate "in time" (if $X$ and $Y$ are contact vector fields, $(1-t)X + tY$ is contact), and "in space" on disjoints closed subsets (if $X$ and $Y$ are two contact vector fields on those subsets, we can interpolate between the corresponding  functions).
\end{itemize}
\end{rmks}

An hypersurface $\Si \subset (V, \xi)$ is {\it convex} if there exist a contact vector field $X$ transverse to $\Si$. A direct consequence is that the set 
\[ \Ga_X = \{ x \in \Si, X_x \in \xi_x \}  \]
is a contact submanifold of codimension $2$, dividing $\Si$ into two Liouville domains $R_\pm$. This configuration is actually a characterisation of a convex hypersurface :

\begin{lem} \cite[Lemme 2.2]{CGHH10_Sutures} \label{lem.cvx_adapt}
Let $\Sigma^{2n} \subset (V^{2n+1}, \xi)$ be a closed hypersurface, without boundary. $\Sigma$ is convex if and only if there exists an orientation of $\Sigma$, a submanifold $\Gamma^{2n-1} \subset \Sigma$ and a contact form $\lambda$  of kernel $\xi$ such that
\begin{itemize}
\item $(\Gamma, \xi \cap T\Gamma)$ is a contact manifold, oriented such that the contact form is positive.
\item $\Gamma$ splits $\Sigma$ alternating parts $R_\pm$, such $R_+$ induces on $\Gamma$ the previous orientation (ie $\Gamma = \dd R_+ = - \dd R_-$).
\item The Reeb vector field associated to $\l$ is positively (resp. negatively) transverse to $R_\pm$.
\end{itemize}
A contact form satisfying those conditions will be called {\it adapted} to the convex boundary. 
\end{lem}

Note that the last condition is equivalent to requiring that $(R_+, \lambda)$ and $(-R_-, \lambda)$, where $R_+$ is oriented as $\Sigma$ and $R_-$ is endowed with the opposite orientation, are Liouville domains.

As a corollary of the proof, we get a normalisation of the contact structure on a neighbourhood of the convex hypersurface.
\begin{cor}\cite[Corollary 2.5]{CGHH10_Sutures}\label{lem.cvx_st_2} If $X$ is a contact vector field transverse to $\Sigma$, there exists a contact form $\lambda$ of kernel $\xi$ and coordinate functions $t : \N(\Sigma) \ra \RR$, $\tau: \N(\Gamma_X) \ra \RR$ such that :
\begin{itemize}
\item $X=\dd_t$, $\Sigma=\{t=0\}$ and $\Gamma=\{t=\tau=0\}$ ;
\item on $\N(\Sigma) \setminus \N(\Gamma)$, $\lambda = \pm dt + \beta_\pm$ where $\beta$ doesn't depend on $t$ (and has no term $dt$) ;
\item on $\N(\Gamma)$, $\lambda = f(\tau) dt + g(\tau) \l_\Ga$ with $\{f=0\}=\{0\}$, $g>0$, $\l_\Ga$ a contact form $\Gamma$ and $f'g - fg' > 0$. 
\end{itemize}
\end{cor}

We also define some additional notion, which restricts which kind of Liouville domain can appear as part of the boundary :
\begin{dfn}
    Consider a convex hypersurface $\Si \subset (V, \xi)$, and an adapted contact $\l$. Then $\Si$ will be called $W$-convex (resp. $S$-convex) for $\l$ if the Liouville domains $(\pm R_\pm, \l)$ are Weinstein (resp. Stein).
\end{dfn}

\subsection{Changing boundary condition : from convex to sutured boundary} \label{ssec.cvx_2_sut}

Let $(V, \xi)$ be a contact manifold containing a convex hypersurface $\Si$, and $\La \subset (V, \xi)$ a Legendrian intersecting $\Si$ transversely.

\begin{lem}\label{bord_leg_reliure} If $\Sigma$ is W-convex for a contact form $\l$, then for a generic perturbation of $\Lambda$ there exists a contact vector field $X$, transverse to $\Sigma$, such that $\Lambda \pitchfork \Sigma \subset \Gamma_X$
\end{lem} 

\begin{proof}
First note that generically, the intersection $\Lambda_\Sigma := \Lambda \cap \Sigma$ is transverse. We then bring $\Lambda_\Sigma$ in a neighbourhood of the suture using the Liouville vector field $Y$ :  $Z = t\dd_t + Y$ is a contact vector field on $\N(\Sigma) \sms \N(\Gamma_X)$, 
which is is extended by zero away from $\Sigma$. Generically $\Lambda_\Sigma$ is a $(n-1)$-submanifold which doesn't intersect the skeleta of $R_\pm$, since they are the union of isotropic submanifolds, of dimension bounded by $n$.

Hence the flow $\phi_Z$ brings $\Lambda_\Sigma$ in a neighbourhood of $\Gamma_X$, and we now need to perturb $X$ such that $\phi_Z \Lambda \cap \Sigma \subset \Gamma_{\tilde{X}}$. We will then have $\Lambda_\Sigma \subset \Gamma_{\phi_Z^*\tilde{X}}$.

For this purpose we use the coordinates obtained in the previous paragraph : on $\N(\Gamma)$, the contact form can be written $\lambda = \tau dt + g(\tau)\l_\Gamma$, the contact vector field is $X=\dd_t$ and the suture is given by $\{t=\t=0\}$. 

We choose a contact vector field $\tilde{X}$, such that $\phi_Z \Lambda \cap \Sigma \subset \Gamma_{\tilde{X}}$, in the following way : by \cref{lem.ContactVf}, $\tilde{X}$ is determined by the function $\l(\tilde{X})$. Note that for $X$ we get $\psi_0 := \l(X) = 
f(\t)$. Pick $\tilde{\psi} : \N(\Gamma) \ra \RR$ such that $\tilde{\psi}(\phi_Z \Lambda \cap \Sigma) = 0$, so we obtain $\tilde{X}= \psi \R + v$, where $v\ in \xi$ is such that $\i_v\ d\l_{\rest \xi} = d\tilde{\psi}_{\rest \xi}$.

Moreover $\phi_Z \Lambda \cap \Sigma$ can be taken arbitrarily close to $\Gamma_X$, and after a generic perturbation of $\La$ the projection of $\phi_Z \Lambda \cap \Sigma$ to $\Gamma$ (parallely to $\dd_\t$) is embedded (it is achievable by dimensional reasons, since $\dim (\La \cap \Sigma) = n-1$ and $\dim(\Ga) = 2n -1$), hence $\tilde{\psi}$ can be chosen arbitrarily $\C^1$-close to $\psi$. Finaly $\tilde{X}$ is transverse to $\Sigma$ when the difference between $\psi_0$ and $\tilde{\psi}$ is small enough, and the suture is then $\{\psi = t = 0\}$ because $\l(\tilde{X}) = \tilde{\psi}$.
\end{proof}

\begin{rmk} We could also conserve the Legendrian and perturb $\lambda$ (but not the contact structure) such that $\dd \La$ doesn't intersect the skeleta of $(R_\pm, \l_{\rest R_\pm})$.
\end{rmk}

Here we normalise the contact structure on a neighbourhood of the suture. Loosely speaking we will show that it looks like the neighbourhood of the binding of an open book decomposition.

We start by constructing a contact vector field tangent to $\Lambda$ :
\begin{lem} Let $\Sigma \subset (V,\xi)$ be a convex hypersurface and $\Lambda \subset V$ a Legendrian intersecting $\dd V$ transversely. If $X$ is a contact vector field transverse to $\Sigma$ such that $\Lambda_\Sigma = \Lambda \cap \Sigma \subset \Gamma_X$, then there exists a contact vector field $\tilde{X}$ such that :
\begin{itemize}
\item $\tilde{X}$ is transverse à $\Sigma$.
\item $\La_\Sigma = \Sigma \cap \Lambda \subset \Gamma_{\tilde{X}}$
\item $\tilde{X}$ is tangent to $\Lambda$ in a neighbourhood of $\Sigma \cap \Lambda$.
\end{itemize}
\end{lem}
The last condition implies that 
$\Lambda_\Sigma$ is a Legendrian submanifold of $(\Gamma_X, \xi \cap T\Gamma)$. Moreover, $\Lambda$ is invariant under the flow of $\tilde X$ and $T\Lambda = <T\Lambda_\Sigma,X>$ on $\N (\Lambda_\Sigma)$, so we can choose coordinates such that $\Lambda = \{t, \tau = 0, \dd \Lambda \subset \Gamma\}$ in a neighbourhood of $\Sigma$.

\begin{proof}
We use the coordinates $(t, \t)$ given by \cref{lem.cvx_st_2}, where $X = \dd_t$, and such that $\xi = \ker( f(\t) dt + g(t) \l_\Ga)$.  Let $\La_0$ be the Legendrian obtained by pushing $\Lambda_\Sigma$ by the flow of $X$ : $\Lambda_0 = \under{\bigcup}{-\e < s < \e} \phi_X^s(\Lambda_\Sigma) = [-\e, \e]_t \times \{0\} \times \La_\Ga$. It is still a Legendrian since $\lambda_{\rest T\Lambda_\Sigma} =0$ and $\l(\dd_t)=0$ on $\{\t =0\}$. 

By \cref{Vois_st_leg}, there exists a neighbourhood $\N(\Lambda)$ contactomorphic to a neighbourhood of $0_\Lambda \subset (J^1(\Lambda), dz - p.dq)$,
and for $\e > 0$ small enough, $\Lambda_0$ stays in $\N(\Lambda_\Sigma) \subset \N(\Lambda)$. It is represented by the graph of a function $\phi : \Lambda \ra \RR$, defined on a neighbourhood of $\Lambda_\Sigma \subset \Lambda$. We extend this function to $\Lambda$ by zero, and its graph will still be denoted $\Lambda_0$. The two Legendrians coïncide on $\Sigma$ (since $\Lambda_0 \cap \Sigma=\Lambda_\Sigma$), hence $\phi_{\rest \Lambda_\Sigma}=0=d_{\Lambda_\Sigma}\phi$. 

In the coordinates $(z;q,p) \in \RR \times T^*\Lambda = J^1(\Lambda)$, we define a vector field 
\[Z(z;q,p) = \phi(q) \dd_z + \dd_{q_i}f. \dd_{p_i}, \]
which is contact :
\[ \L_Z(dz-p.dq) = d(f(q)) - \i_Z(dp \wedge dq) = d_qf - \dd_{q_i}f .dq_i = 0,   \] 
extended by zero far from $\N(\Lambda)$. 
Moreover its flow (at times $1$) maps $\Lambda \simeq 0_\Lambda$ to $\Lambda_0$, hence it induces a contactomorphism $(V,\xi,\Lambda) \ra (V, \xi, \Lambda')$ preserving $\Lambda_\Sigma$ (but not $\Sigma$). The pullback of $X$ is contact, tangent to $\Lambda$ on a neighbourhood of $\Sigma$, and stays transverse for $\e$ small enough : indeed $X$ is given by a function $V\ra \RR$ (we need to choose a contact form), and we can make the contactomorphism $\C^1$-close to the identity (on $\Lambda_\Sigma$ it is $C^1$-equal to the identity). The vector field $\tilde{X}$ being given by the pullback of this function, it is $C^0$-close of $X$.

\end{proof}

Adapting \cite[Lemme 4.1]{CGHH10_Sutures} to the Legendrian setting, we obtain the following result.

\begin{lem} \label{lem.cvx->sut} Let $(V, \xi)$ be a contact manifold with convex boundary $\Sigma$, $X$ a contact vector field transverse to $\Sigma$, and $\Lambda \subset (V, \xi)$ a Legendrian such that $\Lambda \pitchfork \Sigma \subset \Gamma_X$. Then there exists a contact form $\l$, as well as a neighbourhood of the suture $\N(\Gamma)$, such that $(M \setminus N(\Gamma), \lambda)$ is a sutured contact manifold and $\Lambda$ is a cylindrical contact manifold.
\end{lem}

\begin{proof}
We only slightly modify the proof of \cite{CGHH10_Sutures} to follow the Legendrian : we will show that there exists a neighbourhood of the suture $\N(\Gamma_{\tilde{X}})$, a $\xi$-form $\lambda$, as well as coordinates $(r, \theta)$ defined on $\N(\Gamma)$, such that \begin{itemize}
\item $\Sigma = \{\theta= 0, \pi\}$, $\Lambda_\Sigma \subset \Gamma_{\tilde{X}}$, $\Lambda \subset \{\theta = \pi/2\}$, and $\dd_\theta = \pm \tilde{X}$ for $\theta \simeq 0, \pi$ and $r$ big enough.
\item $\lambda = h_0(r, \theta , x) .(\lambda_0 + r^2 d\theta)$.
\end{itemize}

\begin{enumerate}
\item[(i)]
By \cref{lem.cvx_st_2} we have on a neighbourhood of the suture
 \[ \lambda =  g(\tau)(\tilde{f}(\t)dt + \lambda_0) \]
where  $\tilde{f}(\t) = \t$ if $|\t| \leq 1/4$, $g>0$  
 and $g(\t) = g(-\t)$.

Taking polar coordinates $(r, \theta)$ such that 
\[ (\t, t) = (r \cos \theta, r \sin \theta) \text{  and  } \N(\Gamma)=\{ \pi \leq \theta \leq 2\pi, 0\leq r \leq \d \}  \]
we define
\[ \lambda_s = \l_\Gamma + (1-s) \tau dt + s r^2 d\theta\]
Those forms and their differentals being constants on $\Sigma$ (near the suture), they are contact on a neighbourhood of the suture. 

By Gray's theorem (see \cref{prop.Moser}) there exists a family of (local) diffeomorphisms $\phi_s$  such that $\phi_0 = id$ and $\phi_s^* (\ker \lambda_s) = \ker \xi_0$. In other words, there exists coordinates $(\tilde{r}, \tilde{\theta})$  and a function $h_0 : \N(\Gamma) \ra \RR$  such that 
\[\lambda = h_0 (\l_\Gamma + \tilde{r}^2 d\tilde{\theta})\]

\item[(ii)]
On $\Sigma$ we have $(\tilde{r}, \tilde{\theta}) = (r, \theta)$, since $\ker \lambda_s$ is constant here. Moreover, after this change of coordinates, 
$\dd_{\tilde{r}}$ stays tangent to the Legendrian.

Indeed by detailling Moser's trick, $\phi_s$ is the flow of a vector field $X_s$ such that 
  \[ X_s \in \xi_s \text{  and  } \dd_s \l_s + \i_{X_s} d\l_s = \nu_s \lambda_s \]
Evaluating at $R_{\l_s} = R_\Gamma$ (the Reeb vector field of $\l_\Gamma$) we get
  \[\nu_s = \dd_s \lambda_s(R_\Gamma) = (-\t dt + r^2 d\theta)(R_\Gamma) = 0 \]
hence $X_s$ is determined by
  \[ \i_{X_s} d\l_s = -\dd_s \l_s = -\t dt + r^2 d\theta\]
On $\{\t= 0\}$, we have $\xi_s = <\ker \b_0, \dd_t, s\t R_0 - \dd_\t>$ and $X_s$ is given by 
\[i_{X_s}( d\b_0 + (1+s) d\t \wedge dt ) = t^2 d\t. \]
Hence $X_s = \frac{-t^2}{1+s} \dd_t$.

\item[(iii)]
We now choose $h: U \ra \RR$  such that \begin{itemize}
\item $h=h_0$ on $\dd U \cap \{\tilde{r}=\d\}$
\item $\dd_{\tilde{r}} h <0$
\item $h = C_0 /\tilde{r}^2$ for $\e /2 \leq \tilde{r} \leq \e \leq \d$
\end{itemize}

Finally, we define a contact form $\xi$-adapted by 
\begin{align*}
\tilde{\l} & = \l \text{  on  } V \setminus U \\
 & = h (\b_0 + \tilde{r}^2 d\tilde{\theta})  \text{  on } U.
\end{align*}
Then $V \setminus \{ \tilde{r} < \e/2 \} $ is a sutured contact manifold,  with suture $\Gamma = \{\tilde{r}=\e/2, \tilde{\theta} = 3\pi/2\}$ and $\N_0(\Gamma) = V \cap \{ \e/2 \leq \tilde{r} \leq \e\}$, where the Reeb vector field is $\frac{1}{C_0} \dd_{\tilde{\theta}}$), and $\Lambda$ is a cylindrical Legendrian. 
\end{enumerate}
\end{proof}

More generally, we proved that if $\Sigma$ is a convex hypersurface splitting $V$ into $V_1$ and $V_2$, there exists a contact form and a neighbourhood $N$ of $\Gamma$  such that :  \begin{itemize}
\item $V_i\setminus N, i \in \{1, 2\}$, are sutured contact manifolds ;
\item $V \setminus N$ is {\it circularly sutured}, ie it is the $S^1$-contactisation of $(-\e, 0]\times \Ga$ near the boundary ;
\item $V_i \cup N, i \in \{1, 2\}$, are {\it negatively sutured} contact manifolds (called concave in \cite{CGHH10_Sutures}).
\end{itemize}

\subsection{Construction of sutured manifolds} \label{ssec.ExVarSut}

We now construct the principal examples of sutured contact manifolds and sutured Legendrians that will be used in the rest of this paper.

\paragraph{Contactisation.}
The simplest example of a sutured contact manifold is the contactisation of a Liouville domain $(W, \b)$ : 
\[(V, \l) = (I \times W, C.dt + \b),\]
which trivially satisfies the requirements, with suture $\{0\} \times \dd W$. 

By \cref{dfn.LagrExact}, sutured Legendrians in $V$ correspond to exact immersed Lagrangians in $W$ (one need to take $C$ big enough for the lift to exist in $V$).

\paragraph{Complement of a Legendrian.} Given a Legendrian $\La_0 \subset (V, \xi)$, it was a standard neighbourhood, whose boundary is a convex hypersurface. Applying the convex-sutured procedure of \cref{ssec.cvx_2_sut}, we get a sutured manifold $\check V$, of horizontal boundaries $R_\pm \simeq D^*\La$, which yields invariants of $\La$ (this was already observed in \cite{CGHH10_Sutures}). This construction can also be seen as a stop, see \cite[§2.4]{Dat22_Conormal_stops_Hyperbolic_knots} for a more detailed discussion. 
\begin{rmk}Given a sutured Legendrian $\check \La \subset \check V$, a Lagrangian filling $L \subset D^*\La$ of boundary $\La_\Ga$ induces a Legendrian $\La \subset V$, potentially intersecting $\La_0$.
\end{rmk}

\paragraph{Neighbourhood of a convex hypersurface.} The following construction appeared in \cite{Vau12_Bypass} for the $3$-dimensional case. Consider two Liouville domains $(W_\pm, \b_\pm)$, of same boundary $(\Gamma, \l_\Ga)$. We can construct a manifold with convex boundary 
\[([-1,1]_u \times (W_+ \cup_\Gamma W_-), f du + \b)\] 
where $f : W_+ \cup_\Gamma W_- \ra \RR$ vanishes on $\Ga$, and is strictly negative (resp. positive) on $W_-$ (resp. $W_+$), and $\b=\b_\pm$ on $W_\pm$ (it needs to be smoothed). The contact vector field $\dd_u$ is transverse to the boundary, of dividing set $\{\pm 1\} \times \Ga$.

To make it sutured, we can apply \cref{lem.cvx->sut} explicitly : denote $r$ the coordinate on $V$ given by the Liouville vector field on a neighbourhood of $\dd W_\pm$, such that $W_\pm = \{\pm r >0\}$. We define $\l=f du + \b$ in the following way :
\begin{itemize}
\item away from $I \times \N(\Gamma)$, take $f=\pm 1$ and $\b = \b_\pm$  on $W_\pm$ ; 
\item on $I \times \N(\Gamma)$, take $f= f(r,u)$ and $\b = g(r,u) \l_\Ga$.
\end{itemize}
The contact condition is $g^{n-1}(g f_r - f_r g) >0$, and the Reeb vector field is 
\begin{align*}
    R & = \pm \dd_u & \text{ away from } I \times \N(\Gamma)  \\
      & \sim X_g + \dd_r f R_\Gamma  & \text{ on } I \times \N(\Ga), \qquad\quad 
\end{align*}
where $X_g$ is defined by $\i_{X_g} (dr \wedge du) = dg$

Choose a function $g$ with maximums at $(\pm 1, 0)$ (of same value), and a saddle point at $(0,0)$. Finally remove a neighbourhood of the dividing set : $V  = \{g \leq g(\pm 1, 0) - \e\}$ is a sutured manifold, and $\l$ is now adapted, with its Reeb vector field pictured on the right side of \cref{fig.cyl_cvx} (projected on the plane $(r,u)$).

\begin{figure}[h!]
  \center
  \def\svgwidth{14cm} 
  \import{./Dessins/}{/Var_cvx_cyl.pdf_tex}
  \caption{The Reeb vector field on the neighbourhood of a convex hypersurface. On the left the boundary is convex (the functions $f$ and $g$ only depend on $r$), and on the right the contact form is adapted to the sutured manifold.}
  \label{fig.cyl_cvx}
\end{figure}
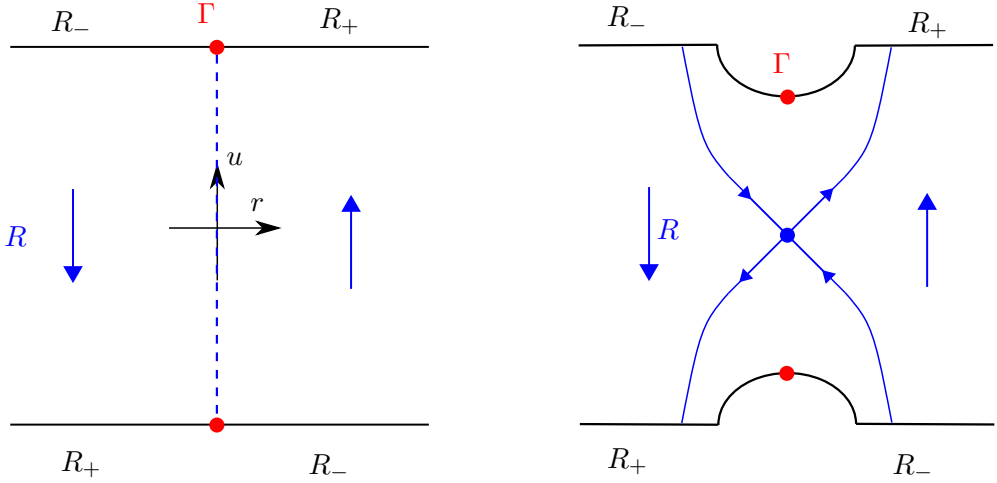

Given a Legendrian $\La_0 \subset \Ga$, we can thicken it to obtain a sutured Legendrian $\La = [-1,1] \times \La_0$. More generally, any exact Lagrangian cobordism in the symplectisation of $\Ga$ induces a sutured Legendrian in $V$.

\paragraph{Conormal of a manifold with boundary.}  We recall a construction from \cite{Dat22_Conormal_stops_Hyperbolic_knots}.
Consider a smooth manifold $M$, with boundary, and choose coordinates on a neighbourhood of $\dd M$ :
$\N(\dd M)  \simeq (-\e, 0]_u \times \dd M$
as well as a metric $g_\dd$ on $\dd M$, and set $g = du^2 + g_\dd$, extended arbitrarily to $M$.
Then the unit tangent bundle 
\[ V = S_g TM = \{(x, v) \in TM, g_x(v) = 1\}\]
comes endowed with a contact form $(\l_g)_{(x, v)} = g_x(v, dv)$, and the vector field $\dd_u$ (on $M$) preserve the metric, hence it lifts to a contact vector field on $S_g TM$, which is transverse to the boundary. Consequently $V$ has convex boundary, and we have coordinates 
 \[ U_g \N(\dd M) = \{ (u, y; \nu, w) \in (-\e, 0]_u \times \dd M \times \RR \times T_y \dd M |\ \nu^2 + g_\dd(w, w) = 1     \}.\]
in which the boundary $\dd V$ splits into 
$R_\pm  = \{u = 0, \pm \nu >0 \} \simeq D^* (\dd M)$ (as Liouville domains), glued along the dividing set $\Ga = \{u = 0, \nu = 0\} \simeq S_{g_\dd} T(\dd M)$ (as contact manifolds).

\begin{rmk}
This construction has a counterpart in the unit cotangent bundle $ST^* M$, whose contact structure is canonically defined. Indeed, any vector field on $M$ lifts to a contact vector field on $ST^* M$, and if it is transverse to $\dd M$ its lift is transverse to $\dd V$. 
\end{rmk}

To get a contact form adapted to the sutured contact manifold, we can take the form induced by the metric $f(u).(du^2 + g_\dd)$, where $g_\dd$ is a metric on $\dd M$ and $f$ is a strictly increasing function (note that the metric $du^2+f(u).g_\dd$ also works, where $f$ is still increasing). Indeed, one can check that the Reeb vector has the behaviour depicted on \cref{fig.cyl_cvx}.

Moreover, for any compact submanifold $N \subset M$ with boundary in $\dd M$, intersecting transversely the boundary, we can choose the coordinates on $\N (\dd M)$ so that $N = (\e, 0] \times \dd N$ on that neighbourhood. Hence its normal bundle
\[U_N M = \{(x, v), x \in N, v \perp  T_x N\} \subset  (ST^* M, \xi_\st)  \]
is Legendrian, of boundary $U_{\dd N} (\dd M)$.

On the cotangent side, it means that we can find a contact vector field on $ST^* M$, such that the boundary of the unit conormal 
\[ST^*_N M = \{(x, \a) \in ST^*M, x\in N, \a=0 \text{ on } T_x N\}\]
is included in the dividing set. To summarise the relevant (sub)manifolds, we have:
\begin{align*}
    & V = ST^*M \qquad \dd V = D^* (\dd M) \under\cup\Ga D^* (\dd M) \qquad \Ga = ST^*(\dd M) \\
  & \qquad \qquad \qquad \La = ST^*_N M \qquad \dd \La = ST^*_{\dd N} \dd M.
\end{align*}

\section{Completions and maximum principles} \label{sec.Max_princ}

\subsection{Completions} \label{ssec.Completions}

We first present the construction of \cite[§2.4]{CGHH10_Sutures}, which we extend to the Legendrian situation using ideas coming from Floer theory.

Let $(V, \Gamma, \N_0(\Gamma), \lambda)$ be a sutured contact manifold. We complete it into a non-compact  contact manifold $(V^*, \lambda^*, \Lambda^*)$ in the following way : 
\begin{itemize}
\item On $\N(R_\pm)$, we have $\lambda = C dt + \beta_\pm$ with $t\in [-1, -1+\e) \cup (1-\e, 1]$. We extend the manifold "vertically" by gluing $([1, \infty)_t \times R_+, C dt + \b_+)$ and $((-\infty, -1]_t \times R_-, C dt + \b_-)$. Thus we get a contact manifold with boundary $\RR \times \Gamma$.
\item On a neighbourhood $(-1, 0] \times \RR \times \Gamma$ of this new boundary, with the coordinate $\tau$ extended by translation, we have $\lambda = C dt + e^\tau \lambda_0$. We now complete "horizontally" by gluing $([0, \infty)_\t \times \RR_t \times \Gamma, C dt + e^\tau \lambda_0$).
\end{itemize}
To summarize, we have 
\[ V^* = V \cup \Cont_\pm(R_\pm) \cup \Cont(\Symp_+( \Ga)).  \]

Given a sutured Legendrian $\La \subset (V, \l)$, there exists several of extending it, inspired by the symplectic setting. 
Remember that given an exact Lagrangian $L$ in a Liouville domain $(W, \b)$, the cylindrical completion is
\[ \hat L = L \cup \RR^+ \times \dd L \subset (\hat W, \hat \b).  \]
Moreover for an Hamiltonian $H_u : \RR^+ \times \dd W \ra \RR, u \in [0, 1]$, vanishing at $\{0\} \times \dd W$, we can define the completion relative to $H$ : 
\[L^H  = \phi^1_{X_H} (\hat L),\]
where $X_H$ is determined by $\i_{X_H} d\hat \b = - dH$.

\paragraph{Cylindrical Legendrian completion.} 

On $\N_0(\Gamma)$, we have $\Lambda = (-1, 0] \times \{0\} \times \dd \Lambda \subset (-1, 0] \times \RR \times \Gamma$.
We extend the Legendrian cylindrically by 
\[\Lambda^* = \Lambda \cup [0, \infty) \times \{0\} \times \dd\Lambda, \]
see \cref{fig.Cylindrical_completion} for a picture.

\begin{figure}[h!]
\center
\def\svgwidth{8cm}
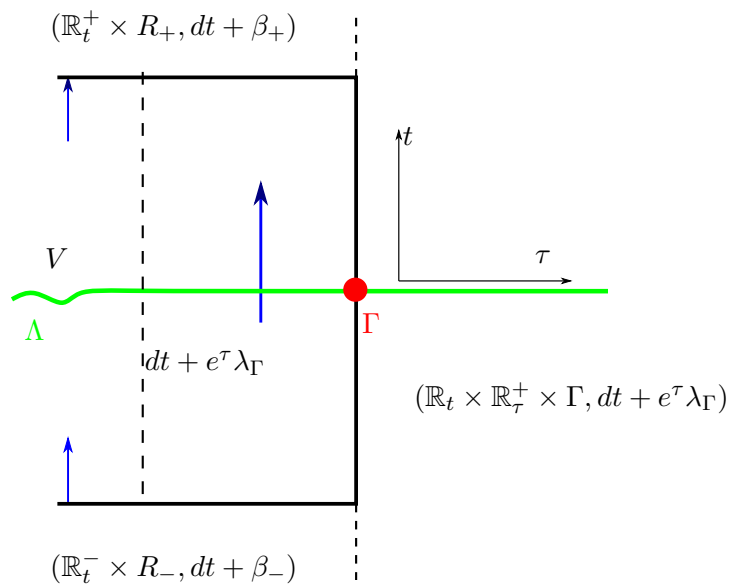
  \caption{Cylindrical completion of a Legendrian.}
  \label{fig.Cylindrical_completion}
\end{figure}

\begin{ex}
If $(W, \beta)$ is a Liouville domain, the completion of $(W \times I, dt + \b)$ is the contactisation of its completion $\hat{W}$ : 
\[(W^*, \lambda^*) = (\RR \times \hat{W}, dt + \hat{\b}).\]
Moreover the cylindrical completion of a sutured Legendrian projects to the cylindrical completion of an exact immersed Lagrangian.
\end{ex}

\paragraph{Wrapped Legendrian completion.} 

Consider $H : \RR_+ \ra \RR$ a function vanishing near $0$. We define the completion wrapped by $H$ in the following way : 
\[\Lambda^H  =\La \cup  \{(\tau,- f_H(\tau), \phi_{R_0}^{g_H(\tau)}(x)), \tau \geq 0, x \in \dd \Lambda\} \]  where  $f_H = \frac{1}{C} \int_0^\tau e^\tau H'' .d\tau$    and   $g_H = H'$, see \cref{fig.completion}.

Note that the projection of $\Lambda_H$ to $\RR_+ \times \Gamma$ is the cylindrical Lagrangian $\RR_+ \times \Lambda_\Sigma$ (in the symplectisation of $(\Ga, \l_\Ga)$) 
wrapped by $H$ : 
\[\pi (\Lambda^H) = \phi_{X_H}^1 (\RR_+ \times \dd \Lambda),\]   
where $X_H$ is determined by $\i_{X_H} d\hat \b = - dH$, because $\phi_{X_H}^1 = \phi_{R_0}^{H'(\tau)}$. 
Alternatively, we can see this Legendrian as a lift of the Lagrangian wrapped by $H$, and from that point of view $f_H$ is the primitive of $(e^\t \l_\Ga)_{\rest L^H}$.

\begin{dfn} 
The completion will be called {\it positive} if $H'$ is increasing (ie if $H'' \geq 0$), and {\it total} if $H' \ra \infty$. 
\end{dfn}

\begin{figure}[h!]
\center
\def\svgwidth{12cm}
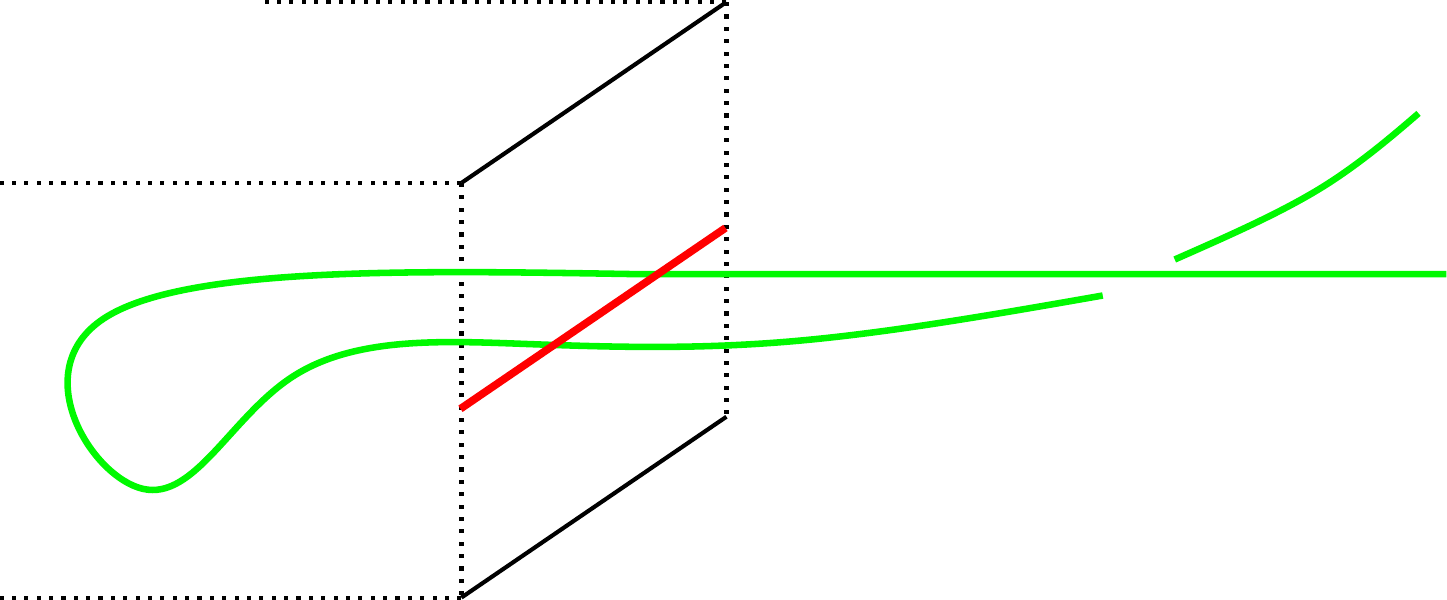
  \caption{Completion of a Legendrian : $\dd_1 \La$ is extended cylindrically, and $\dd_0 \La$ is extended by a positive wrapping.}
  \label{fig.completion}
\end{figure}

\begin{rmk}\label{rmk.Completion_isotopy}
More generally, any isotopy $\mu$ of $\dd \La \subset (\Ga, \l_\Ga)$ induces an exact Lagrangian cobordism  $L \subset ([0, C]_\t \times \Ga, e^\t \l_\Ga)$, as described \cref{ssec.ConstrCobordLagr}. This Lagrangian lift to a Legendrian $\hat L$ in the contactisation, and we can extend $\La$ into $\La^\mu := \La \cup \hat L$. 
\end{rmk}

\subsection{Adapted almost complex structures} \label{ssec.Struct_p-complexe}

Let $W$ be a manifold of even dimension. An {\it  almost complex structure} is a smooth linear bundle map $J:TW \ra TW$, such that $J^2 = -Id$ (on each fiber). If $(W, \w)$ is symplectic, $J$ will be called {\it $\w$-admissible} if $\w(v,Jv) > 0$ and $\w(., .) = \w(J., J.)$ (in other words $\w(., J.)$ define a metric).

The space of admissible almost complex structure on $(W, \w)$ is non-empty and contractible, because they are in bijection with the metrics. 

\paragraph{In a symplectisation.}
\begin{dfn} \label{JadaptSymp} Let $(Y, \xi)$ be contact, $W=\RR_s\times Y$ its symplectisation, and $J$ 
an almost complex structure on $W$. It will be called {\it adapted} to the symplectisation if
\begin{itemize}
\item $J$ is invariant by $\RR$-translation ;
\item on each level $\{s\} \times Y$, $J$ preserves $\xi$ ;
\item There exists a form $\lambda$, $\xi$-adapted, such $J$ maps $\dd_s$ to $R_\lambda$ ;
\item $J_{\rest \xi}$ is $d\lambda$-admissible.
\end{itemize}
\end{dfn}
Note that the last condition doesn't depend on the choice of $\lambda$. If we want to precise the used contact form we will say that $J$ is $\lambda$-adapted. Moreover, $J$ is fully determined by $\l$ and its projection $\bar{J} = J_{\rest\{s\} \times \xi}$.

\paragraph{In a cobordism.}
Let $(W, \b)$ be a Liouville cobordism from $(V_+, \l_+)$ to $(V_-, \l_-)$, and $J$ an admissible almost complex structure. We will say that $J$ is adapted if there exists a compact $K$ such that $W \setminus K = \Symp_+(V_+, \l_+) \sqcup \Symp_-(V_-, \l_-)$ and $J$ is adapted to the symplectisations in the sense of \cref{JadaptSymp}.

\paragraph{In a sutured manifold.}

We take the definitions of \cite{CGHH10_Sutures}, which are unchanged in the Legendrian case. 

\begin{dfn} \label{JajustSut} Let $(M, \Gamma, \N_0(\Gamma), \xi)$ be a sutured contact manifold, and $\lambda$ an adaped form. An almost complex structure on the symplectisation (of the completion) $W= \RR_s \times M^*$ will be called {\it tailored} if 
\begin{itemize}
\item $J$ is $\l^*$-adapted
\item $J$ is $\dd_t$-invariant on a neighbourhood of $M^* \setminus \mathring M$ 
\item The projection of $J$ to $T\hat R_\pm$ (parallel to $t$ and $s$), denoted $J_\pm$, is adapted to the Liouville form $\hat{\beta}_\pm$ : it is is $d\beta_\pm$-positive on $R_\pm$, and $\l_\Ga$-adapted on $\{\tau \geq 0\} = \RR_+ \times \Gamma$. 
\end{itemize}
\end{dfn}
In particular the restriction to $\ker \l_\Ga$, which we denote $J_Ga$, is well-defined and 
uniquely determines $J$ on $\{\t \geq 0\}$. \\

The main feature of those almost-complex structures is the nice behaviour of the projection to $\RR^+_\t \times \Ga$ : 

\begin{lem} 
Let $J$ be an almost complex structure tailored to the symplectisation of a sutured manifold. Then
\begin{itemize}
\item $d\pi \circ J = J_\Ga \circ d\pi$, where $\pi$ is the projection parallel to $s$ and $t$ ;
\item $\l_\Ga \circ J = d\t$.
\end{itemize}
\end{lem}

\begin{proof}
The first point is by definition of a tailored almost complex structure. For the second part, the two 1-forms vanish on $\dd_s, \dd_t$ and $\xi_\Ga$, evaluate to $1$ in $\dd_\t$, and vanish on $R_\Ga$ 
{\sm since
\[ R_\Ga = \left(R_\Ga - \frac{e^\t}{C}\dd_t \right) + \frac{e^\t}{C}\dd_t.\]}
\end{proof}

\subsection{Holomorphic curves and energies}
We now define the moduli spaces which will be used to construct our invariants. Let us first recall some definitions. 

\paragraph{Chords and orbits.}

Given a contact manifold $(V,\l)$, a {\it Reeb orbit} is a closed trajectory of the Reeb vector field, and their set will be denoted $\P(V, \l)$.  The {\it action} of an orbit $\g$ is $\A(\g) = \int_\g \l$, which is also the length of the trajectory (parameterised using the Reeb vector field). A Reeb orbit $\g$ is called {\it non-degenerate} if the linearized return map of the Reeb flow, restricted to $\xi_\g$, has no fixed vector.

Given a  Legendrian $\La \subset (V, \l)$, a {\it Reeb chord} is a trajectories of the Reeb vector field going from the Legendrian to itself, and their set will be denoted $\C(\La; V, \l)$. Given a second Legendrian $\La' \subset (V, \l)$, the set a chords going from $\La$ to $\La'$ will be denoted $\C(\La, \La' ; V,\l)$.
The {\it action} of a chord $c$ is once again $\A(c)  = \int_c \l $, and a chord $c$ of action $T$ will be called non-degenerate if
\[(d\phi_R^T)(T_{c(0)} \La) \pitchfork T_{c(T)} \La.\]

\begin{dfn}
A sutured contact manifold will be called non degenerate if all Reeb orbits are non-degenerate. \\
A sutured Legendrian $\La \subset (V, \l, \Ga)$ will be called 
\begin{itemize}
    \item non-degenerate if all Reeb orbits and chords are non-degenerate (in $(V, \l)$) ;
    \item totally non-degenerate if it is non-degenerate, and all Reeb chords of $\C(\dd \La; \Ga, \l_\Ga)$ are non-degenerate.
\end{itemize}
A sutured contact manifold will be called hypertight if all Reeb orbits are non-contractibles. \\
A sutured Legendrian $\La \subset (V, \l, \Ga)$ will be called relatively hypertight if all Reeb chords are non-contractible.
\end{dfn}

Given a sutured contact manifold, the set of adapted contact forms which are non-degenerate is a countable intersection of dense open sets.
Similarly, given a sutured Legendrian $\La \subset (V, \l, \Ga)$, the set of adapted contact forms so that $\La$ is totally non-degenerate is also a countable intersection of dense open sets.

\begin{rmk}
It is also possible to fix the contact form, and instead to perturb the Legendrian. Once again the set of perturbations so that the Legendrian is totally non-degenerate is a countable intersection of dense open sets.
\end{rmk}

\paragraph{Holomorphic curves.} Consider a Liouville cobordism $(W, \b)$ (which could simply be a symplectisation), with an adapted complex structure $J$. Assume that we have a (embedded) Lagrangian cobordism $L \subset (W, \b)$, of boundary $\La_\pm \subset \dd_\pm W$.

An {\it holomorphic curve} is a  map $F : (S, j) \ra (W, J)$ defined on a Riemann surface, such that $dF \circ j = J \circ dF$ and $F(\dd S) \subset L$.

Let $g$ be the metric induced by $(\w, J)$. We then have the identity
\[ A (F) =  \text{Area}_g (F(S)) = \int_S F^*\w  \geq 0\]

We also define several notions of energy for homolorphic curves :  
\begin{itemize}
    \item  If $F : (S, j) \ra (W, d\beta, J)$ is an holomorphic curve in a symplectic manifold $(W, \w)$ endowed with an adapted complex structure $J$, we define for any function $\phi : W \ra \RR_+$
\[E_\phi(F) = \int_S F^*(\phi \w) \geq  0\]

\item If $F = (a, f) : (S, j) \ra  = (\RR_s \times V, \lambda, J)$ is an holomorphic curve in the symplectisation of $(V, \lambda)$, endowed with an adapted complex structure $J$, we set
\begin{align*}
   E_{d\lambda} (F) &  = \int_S f^* d\lambda \geq 0 \\ 
   E_\lambda(F) & = \underset{\sm \vphi:\RR \ra \RR_+, \int \vphi = 1}{\sup} \int_S (\vphi \circ a) ds \wedge f^*\lambda \geq 0 \\
   E(F) & = E_{d\lambda} (F)  + E_\lambda(F) \\
    & \geq \underset{\sm \psi : \RR \ra [0,1], \psi'\geq 0}{\sup} \int_S F^*d( \psi \lambda) \geq 0
\end{align*}
\end{itemize}
Note that those definitions also make sense for a curve in a Liouville cobordism, using the $\RR$-coordinate at infinity.

\paragraph{Monotonicity.}

Consider a Lagrangian $L$ in a Liouville domain $(W, \b)$, which we complete cylindrically into $\hat L \subset (\hat W, \b)$. This non-compact manifold has a so-called {\it bounded geometry}, see \cite[Lem 2.43]{GPS17_Fuk_sectors}. In particular, the following monotonicity result holds :

\begin{prop}\cite[Propositions 4.7.1 and 4.7.2(iv)]{AL94_Sik_Ch5:Prop_hol} \label{prop.Monotonicite_global}
If $(W, \w, J, L)$ has bounded geometry, and $F :(S,j) \ra (W,J)$ is a compact holomorphic curve such that 
\[ f(S) \cap K \neq \emptyset \text{  and  } F(\dd S) \subset K \cup L, \]
then there exists a constant $C$ such that $F(S) \subset \N_g(K, C \Area(F(S)) )$
\end{prop}

\subsection{Maximum principles in a Liouville cobordism}

We prove several maximum principles allowing the use of the results of \cite{BEHWZ03_Compactness_SFT}. We start with some standard situations before discussing the sutured case.

\begin{lem}[Folklore, \cite{GPS17_Fuk_sectors}] \label{PrincMaxSympl}
Let $L \subset (\RR^+_s \times V, e^s \lambda)$ be a cylindrical  Lagrangian in the positive symplectisation of a contact manifold, $J$ an adapted almost complex structure, and $F=(a, f)  : S \ra \RR_+ \times V$ an holomorphic curve, without punctures, such that 
\[ \dd S = \dd_l S \cup \dd_n S,\ F(\dd_l S) \subset L,\ a(\dd_n S) = 0 \text{ and } F^*\lambda_{\rest \dd_l S} \leq 0\]
Then $F$ is a constant map.
\end{lem}

\begin{proof}
We first use the proof of \cite{Bou03_Intro_contact} : denote $s$ the $\RR$-coordinate. Setting $d^\CC \a = d\a \circ J$ we have $\Delta = -d d^\CC$, hence
\[  \Delta e^a = \Delta (F^* e^s) = F^* (\Delta e^s) \]

However $\Delta e^s = -d (d(e^s) \circ J)=  -d (e^s ds \circ J) =d (e^s \lambda) = \w$, hence $\Delta e^a = F^* \w >0$ on $S$.

Consequently the function $e^a$ has no local maximum on $\mathring S$, so it is also the case of $a$.

For deal with the boundary, we reproduce the proof of \cite[Lemme 2.45]{GPS17_Fuk_sectors}, see also \cite[Lemme 7.2]{AbSe09_Open_string_analogue} for the case of Floer curves.
 
Let $\vphi : \RR_+ \ra \RR_+$ be a function such that $\vphi(0)=0$ and $\vphi' \geq 0$. By Stokes formula we have : \begin{align*}
0 \leq & \int_S F^*(\vphi(s) d\lambda) = \int_{\dd S} F^*(\vphi(s) \lambda) - \int_S F^*(\vphi'(s) ds \wedge \lambda) \\
& = \int_{\dd_l S} \vphi\circ a. F^*\lambda - \int_S \vphi' \circ a. F^*(ds \wedge \lambda) \leq 0  
 \end{align*}
Thus $F^* d\l= 0$, hence $\im (dF) \subset \Vect(R, \dd_s)$. If the curve is not constant, it should be a cylinder on a Reeb trajectory, which is impossible since we assumed that the curve was without puncture.
\end{proof}

Thus an holomorphic curve in a symplectisation $F: (S, \dd S) \ra (\RR \times V, \RR \times \Lambda)$ has no local maximum, even in the boundary.

Let $(\xi_s)_{s\in [0,1]}$ be a path of contact structures on $V$, and fix adapted contact forms $\l_0$ and $\l_1$, as well as admissible almost-complex structures $\bar{J}_0 : \xi_0 \ra \xi_0$ and $\bar{J}_1 : \xi_1 \ra \xi_1$. As seen in \cref{ssec.ConstrCobordLagr}, we can construct (using Gray's theorem) a Liouville cobordism $[-S, S] \times V$ going from $(V, C\l_1)$ to $(V, \l_0)$. The space of admissible almost-complex structures being contractible, there exists $J$ admissible interpolating between the chosen almost-complex structures. The following lemma shows that we can pick $J$ such that the $s$-coordinate is pluri-sub-harmonic (after reparametrisation).

\begin{lem}  [In an isotopy cobordism] 
\cite[Lemme 3.3]{CGHH10_Sutures}
Let $\RR_s \times Y$ be a cobordism from $(Y, C \lambda_1)$ to $(Y, \lambda_0)$ induced by an isotopy, as described previously. Then there exists a complex structure $J$ on $\RR \times Y$, $s_\pm \in \RR_\pm$ and $\phi : \RR \ra \RR$ increasing such that 
\begin{itemize}
\item for $s > s_+$, $J$ is adapted to $C \lambda_1$ and $J_{\rest \xi_+} = J_1$.  
\item for $s < s_-$, $J$ is adapted à $\lambda_0$ and $J_{\rest \xi_-} = J_0$.
\item $\phi (s) : \RR \times Y \ra \RR$ is $J$-pluri-sous-harmonic.
\end{itemize}
In particular if $u : (S, j) \ra (\RR \times Y, J)$ is a $J$-holomorphic curve, $s \circ u$ has no local maximum local on $\mathring{S}$. \end{lem}

\begin{proof}
The proof is as in \cite{CGHH10_Sutures} : we start by showing that there exists $C > 0$ and $J$ adapted to $C.\lambda_+$ (resp. $\lambda_-$) for $s > 1$ (resp. $s<0$)  such that $s : \RR\times Y \ra \RR$ is $J$-pluri-sous-harmonic, then we will set $\tilde{s} = s/C$ for $s$ big enough to obtain the lemma.

According to Gray's theorem, we can assume that $\l_+ = f \l_-$, by composing with a diffeomorphism preserving the $s$-levels. We now consider the contact forms $\l = g(s, x) \l_-$ where $g$ is such that \begin{itemize}
\item $g(s, *) = 1$ for $s \leq 0$ ;
\item $g(s, *) = C f$, where $C$ is a constant greater than $\max(1/f)$ ;
\item $g_s = \dd_s g \geq 0$.
\end{itemize}

Take $J$ an almost complex structure preserving $\xi = \ker(\l_-)$, mapping $\dd_s$ to $R_s$ the Reeb vector field of $g(s, *)\l_-$ and such that $d_Y \l \circ (id \otimes J)$ is a metric on $\xi$ (where $d_Y$ is the differential on $Y$). We compute
\[ \Delta s = -d(ds \circ J) = d\l = g_s ds \wedge \l_- + d_Y \l.   \]
Splitting a vector into $v = a \dd_s + b R_s + X$, where $X \in \xi$, we have  $Jv = a R_s - b \dd_s + JX$ and we get
\[ \Delta s (v, Jv) = g^{-1} g_s  (a^2 + b^2) + d_Y\l(X, JX) \geq 0.  \]
\end{proof}

We now prove a similar result for curves with boundary. Let $\Lambda_t \subset (V, \xi_t)$ be a Legendrian isotopy, which induces an exact Lagrangian cobordism  $L \subset (\RR \times V, e^s \l_t) $ going from $\Lambda_+$ to $\Lambda_-$, cylindrical if $s< s_-$ or $s> s_+$. 

\begin{lem} \label{PrincDuMaxCobord}
There exists an almost-complex structures adapted to the cobordism, and such that there exists no holomorphic curve without puncture
\[ u : (S, \dd S,j) \ra ([s_-, \ity) \times V, \{s_-\} \times V \cup L, J)  \]
Moreover we can choose $J_{\rest \xi}$ for $s > s_+$ and $s < s_-$.
\end{lem}

\begin{proof}
Once again the pluri-sous-harmonicity only forbid maximums in the interior of the surface. To generalise to the case of curves with boundary we use the Stokes formula.
We just showed that $u^*\Delta s$ is non-negative. If such a curve exists, we split its boundary into :
\[ \dd S = \dd_l S \cup \dd_n S, \quad u(\dd_l S) \subset L, \quad u(\dd_n S) \subset \{s_-\} \times Y.   \]   
Taking $\phi : \RR \ra \RR^+$ increasing and vanishing at $s_-$, we compute

\[ 0 \leq \int_S  u^*(\phi(s) \Delta s) = \int_S u^*(\phi(s) d(g\l_-)) = \int_{\dd S} u^*(\phi(s) g\l_-) - \int_S u^*(\phi'(s)  ds \wedge \l)   \] 
using $\Delta s = - d(ds \circ J) = d(g\l_-)$. The second term is negative since $ds \wedge \l$ is positive : if $v =  a \dd_s + b R_s + X$, then $ds \wedge \l (v, Jv) = a^2 + b^2$. Morevoer $\phi = 0$ on $\dd_n S$ and $(g \l_-)_{\rest L} = df$, where $f$ is zero at $s_-$. Thus we get
\[ 0 \leq \int_S u^*(\phi(s) \Delta s) \leq \int_{\dd_l S} u^*(\phi(s) df)  \]
Taking a sequence of function $\phi$ converging to $1$ (pointwise), this inequality becomes
\[ 0 \leq \int_S u^*(\Delta s) \leq \int_{\dd_l S} d(f\circ u) = 0  \]
since $f$ is constant at $s_-$.

Thus $u^* d\l = 0$ on $\mathring{S}$, and $du=0$ on an open set. By unicity of the holomorphic extension, the curve is constant.
\end{proof}

\subsection{Maximum principles in a sutured manifold}
\paragraph{Cylindrical completion.}

We extend the result of \cite{CGHH10_Sutures}, bounding curves in the completion of a sutured contact manifold, to the case of curves with boundary.

\begin{lem} \label{PrincMaxSut} Let $(V, \Gamma, \N_0(\Gamma), \lambda)$ be a sutured contact manifold endowed with an adapted contact form, $\Lambda \subset V$ a cylindrical Legendrian, $J$ an ajusted complex structure on $V^*$ and $F  \in \M(c_0; c_i, \Lambda^*)$  
a $J$-holomorphic curve. Then $\tau \circ F <0$ and $|t \circ F|$ is bounded by a constant (depending only on $c_+$, $J$, and $\lambda$).
\end{lem}

\begin{proof}
We start by showing that the curve is horizontally bounded (ie in the $\t$-coordinate).The complex structure being tailored to the completion, the projection to $\RR^+_\t \times \Gamma$ is $J_\Ga$-holomorphic. Indeed by setting 
\[S_0 = (\tau \circ F)^{-1}(\RR_+) \ \ \text{ and } \ \ G = \pi \circ F : (S_0,j) \ra  \RR^+_\tau \times \Gamma  \]
where $\pi$ is the projection parallel to $(s,t)$, the condition $dF \circ j = J \circ dF$ implies
\[ dG \circ j = d\pi \circ dF  \circ j =  d\pi \circ J \circ dF = J_\Ga \circ d\pi \circ dF = J_\Ga \circ dG  \]
because $J$ is adapted.

We obtain a $J_\Ga$-holomorphic curve in the symplectisation of $(\Gamma, \l_\Ga)$, satisfying the hypothesis of \cref{PrincMaxSympl}. In particular $\tau \circ F$ has no local maximum, even in the boundary. Since the curve $F$ is asymptotic to chords contained in the interior of $V$, it stays in that interior.

For the vertical direction (the $t$-coordinate), the proof is as in \cite{CGHH10_Sutures}. We quickly present the arguments, starting with a case of a S-sutured contact manifold. \\

{\it The Stein case} : assume that the manifold is $S$-sutured, endowed with an tailored almost complex structure such that its projection to $R_\pm$ is Stein. 

Let $J_\pm$ be the structures induced on $\hat R_\pm$ by the projections parallel to $\dd_s$ and $\dd_t$. Since $(\hat R_\pm, \hat{\beta}_\pm, J_\pm)$ are Stein, there exists functions $\phi_\pm : \hat R_\pm \ra \RR$ strictly plurisubharmonics such that $\hat{\beta}_\pm \circ J_\pm = d\phi_\pm$.
We can now compute
\begin{align*}
\Delta (t \circ F) & = - d d^\CC(t \circ F) = - d(d(F^* t) \circ j) =  -d(F^*(dt \circ J)) \\
 & = -F^*d((\frac{\lambda - \hat{\beta}_\pm}{C}) \circ J) = \frac1C F^*d(\hat \b_\pm \circ J)
\end{align*}
because $\lambda \circ J = ds$. Moreover $\hat{\beta}_\pm \circ J = d\phi_\pm$ : those two 1-forms vanish on $\dd_s$ and $\dd_t$, and coincide on $T\hat R_\pm$ : if $v$ is tangent to $\hat R_\pm$, then $\hat{\b} \circ \bar{J}(v) = \hat{\b} (J_\pm v + * \dd_t) = \hat{\b}(J_\pm v)$.

Hence $\Delta(t\circ F) =0$, in other words $t\circ F$ is harmonic on $F^{-1}(\{|t| > 1 \})$. The maximum principle then forbids the existence of a local maximum, which would exist if the curve was exiting $\{|t| < 1\}$ (remember that the curve is asymptotic to Reeb chords which are in $\mathring V$, hence there is no boundary issue).

\begin{rmk}
A priori a Stein structure on $\hat{R}_\pm$ doesn't lift to a tailored structure on the sutured manifold, because the levels of the plurisubharmonic function $\phi$ can be distinct from the $\t$-level, see \cite[§3.2]{CGHH10_Sutures} for an example. \\
Also note that if the Liouville domain $(W, \b)$ admits an almost complex structure $J_\Ga$ which make it Stein, then $\b \circ J_\Ga$ is exact. However for this proof it is enough to control the sign of this form. 
Hence instead of using the technical result from \cite{CGHH10_Sutures} (which we will outline shortly), it would be enough to find $J_+$ (resp. $J_-$)  such that $d(\b \circ J_+)$ is positive (resp. negative). 
\end{rmk}
~\\

{\it General case} : If the  manifold is not S-sutured, it is shown in \cite{CGHH10_Sutures} that the bound in the $t$-coordinate $t$ depends on the complexity of the curve (ie genus and number of punctures). 
The Legendrian case is handled in the same way, because the Legendrian does not cross $R_\pm$. In what follows we denote by $c_{top}$ the topological complexity of a marked surface  $(S, \pp)$, defined by
\[c_\top (S) = g(S) + |p|\]
 and we set $\bar{t} = t \circ F$, $\bar{\t} = \t \circ F$.

\begin{lem}\cite[Lemme 5.18]{CGHH10_Sutures}  \label{PrincDuT-max} Let $J$ be a tailored almost complex structure, and $F_n : (\Sigma_n, j_n, p_n) \ra (\RR \times V^*, J)$ a sequence of holomorphics curves such that $E(F_n), c_{top}(\Sigma_n)$ and $\bar \tau_n$ are uniformly bounded. Then $\bar t_n$ is uniformly bounded. 
\end{lem}

\end{proof}

\paragraph{Wrapped completion.}
We now consider a Legendrian $\La = \sqcup \La_i \subset (V, \l)$, with a family of Hamiltonians $H_i : \RR^+ \ra \RR$ which we use to construct the wrapped Lagrangian 
\[L = \cup L_i \subset (\RR^+_s \times V, \b = e^s \l)\]\
where 
\[ L_i = \La_i^{H_i} = \{ (s, \phi^{H_i'(s)}_{R}(x)), x \in \La_i  \}  \]
Note that this correspond to wrapping by the Hamiltonians vector fields induce by $H_i$, because $\phi^{H'_i(s)}_R = \phi^1_{X_{H_i}}$.  We also fix $J$ an almost complex structure adapted to the symplectisation. We start by showing that the holomorphic curves satisfy a maximum principle.

\begin{lem}\label{PrincMaxEnrou} 
With the previous notations, let $F = (a, f) : (S,j) \ra (\RR^+ \times V,J)$ be an holomorphic curve without puncture such that 
\[ \dd S = \dd_l S \cup \dd_n S,\ \ F(\dd_l S) \subset L \text{  and  } a(\dd_n S) = 0. \]
Then $F$ is constant.
\end{lem}

\begin{proof}
A priori we can't apply \cref{PrincMaxSympl} because we can't control the sign of $F^* \lambda$ on the boundary. However we can still use the Stokes formula to conclude. We choose $\phi : \RR^+ \ra \RR^+$   increasing, vanishing at $0$, and we compute : 
\begin{align*}
 0  \leq & \int_S F^*(\phi(s) d\lambda) = \int_{\dd S} F^* (\phi(s) \lambda) - \int_S F^*(\phi'(s) ds \wedge \l)  \\
  & =  \int_{\dd_l S} F^*(\phi(s) \l) - \int_S \phi'\circ a . F^*(ds \wedge \l) 
\end{align*}

Since $F$ is $J$-holomorphic and $J$ is tailored, the factor $F^*(ds \wedge \l)$ is positive. By splitting the boundary into $\dd_l S = \sqcup \dd_i S$ such that $F(\dd_i S) \subset L_i$, we get $\l_{\rest T L_i} = df_i(s)$ where $f_i = f_{H_i}$ is positive and  vanish at $0$, as in  \cref{ssec.Completions} (see also \cref{dfn.LagrExact}). We then  compute :

\[ \int_{\dd_i S} G^*(\phi\l) = \int_{\dd_i S} G^*(\phi f_i' ds) = \int_{\dd_i S}  d(\psi_i \circ G) \]
where $\psi_i(s)$ is a primitive of $\phi f'_i$ vanishing at $0$. 
Since the curve has no puncture, $\dd(\dd_i S) \subset \{s=0\}$ and so the first integral vanishes.
 \end{proof}

\begin{rmk} The crucial point in the previous computation is that the wrapping only depends on the $\t$-coordinate. More precisely, the projection of the Legendrian to $\RR_\t \times \Gamma$ is an exact Lagrangian, and the primitive of restriction of the Liouville form is a function of $\t$. Also note that we did {\it not} use the positivity of the wrapping. 
\end{rmk}

\begin{cor}
Let $\Lambda_0, \Lambda_1$ be two sutured Legendrians in a sutured manifold $V$, $H : \RR^+ \ra \RR^+$ a smooth increasing function vanishing on a neighbourhood of $0$, and $F \in \M(c_0; c_i; \Lambda^H_0 \cup \Lambda_1^*)$ an holomorphic curve. Then $\sup \t \circ F \leq \max(0, \t(c_0))$, and $|t\circ F| \leq C(c_0, c_\top)$.
\end{cor}

\begin{proof}
For the $\t$-direction, the idea is as previously : after restriction and projection we obtain a curve $G : S \ra \RR^+ \times \Gamma$, where the boundary of the surface splits into $\dd S = \dd_n S \cup \dd_0 S \cup \dd_1 S$ such that 
\[ \t(\dd_n S)=0,\ \ G(\dd_0 S) \subset \Lambda_0^H, \ \ G(\dd_1 S) \subset \Lambda_1^*  \]
The previous maximum principle then implies that $\sup \bar{\t} \leq \max(0, \t(c_0), \t(c_i))$.

We now show that for any negative puncture $c_i$, we have 
$\t(c_i) \leq \t(c_0)$. 
Assume it is not the case : we would have, 
after restriction and projection, a $J_\Ga$-holomorphic curve 
\[U_0 = (a, u) : (S, j) \ra (\RR^+_\t \times \Gamma, J_\Ga)\]
with a negative puncture, and  such that $\dd S = \dd_0 S \cup \dd_n S \cup \dd_1 S$ and $U_0 (\dd_0 S) \subset \pi_\Gamma (\La_0^H)$, $a(\dd_n S)=0$, $U_0 (\dd_1 S) \subset \La^*_1$.

We now compute, while being careful about the orientation of the boundary (see 
\cref{fig.orient_bord}) :

\begin{figure}[h!]
\center
\def\svgwidth{11cm} 
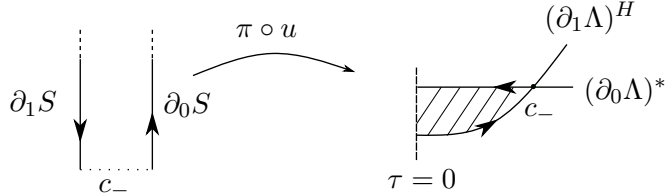
\caption{Orientation of the boundary near a negative puncture}
\label{fig.orient_bord}
\end{figure}

\begin{align*}
0 \leq & \int_S U^* \phi(\t)d\l_\Ga = \int_{\dd S} U^* \phi(\t) \l_\Ga - \int_S U^* (\phi' d\t \wedge \l_\Ga ) \\
 & = \int_{\dd_0 S} U^*(\phi \l_\Ga)  - \int_S U^* (\phi' d\t \wedge \l_\Ga )  \leq 0
\end{align*}

Indeed 
\[\int_{\dd_0 S} \phi. \l_{\Ga \rest T\La_1^H} = \int_{\dd_0 S} U^*( \phi d\psi (\t)) = \int_{U(\dd_0 S)} \phi \psi' d\t \leq 0\]
because $\psi$, primitive of $ (\l_\Ga)_{\rest T\La_1^H}$ vanishing at $0$, is a positive and increasing function of $\t$, see \cref{ssec.Completions} (this time the positivity of the wrapping matters).

For the $t$-direction, the proof is unchanged.
\end{proof}

\section{Sutured Legendrian homologies}\label{sec.Sutured_homologies}

\subsection{Definition of the invariants} \label{ssec.DfnInvariants}

For simplification, we assume that there is an hypertight\footnote{If it is not the case, one need to add the contractible Reeb orbits to the algebra (note that it is also possible to consider all of them).} contact form $\l$, such that $\La$ is non-degenerate. We also pick $J$ a tailored almost complex structure. 

To further simplify the construction, we also assume 
\begin{itemize}
    \item $H_1(V)$ is free ;
    \item the first Chern class $c_1(\xi)$ vanishes.
\end{itemize}
More general constructions can be considered, see \cite{EES05_LCH_PxR} \cite[§2.3]{EENS11_Knot_hom} as well as \cite{Bou03_Intro_contact} and \cite{Par15_CH_&_virtual}.

\subsubsection{Cylindrical sutured homology} \label{sssec.Cylindrical_sutured_homology}

Consider a sutured Legendrian $\La \subset (V, \l, \Ga)$, and assume that it is relatively hypertight (in the general case we get a dga $\LC(\La)$, see \cref{ssec.dga_wrapped}).

\paragraph{Generators.}  We define the {\it sutured cylindrical complex} $LC(\La; V, \l, \Ga, J)$ as the $\ZZ[H_1(\La)$-bimodule generated by Reeb chords $\C(\La; V, \l)$.

An homological element is graded by the opposite of its Maslov class.
To define the degree of the Reeb chords, choose : a base point $p_i$ for any connected component of $\La$ ; for all chords, a path in $\La$ from its endpoints to the relevant base point $p_i$ ; a path between $p_i$ and $p_j$ for any $i \neq j$, as well as a path of Lagrangians in $\xi$ from $T_{p_i} \La$ to $T_{p_j} \La$.
With this data, we can construct for any chord $c$ a loop $\g_c$, on which we trivialise $\xi$ (here we need to pick generators of $H_1(V)$, as well as a trivialisation of $\xi$ over each of them). Thus we get a path a Lagrangians, whose extremities are transverse, that we close by a "positive rotation" to get a loop of Lagrangians $\hat \g_c$ (see \cite{EES02_LCH_R2n+1}). The degree of $c$ is then defined by 
\[ |c| = \mu(\hat \g_c) -1,  \]
where $\mu$ is the Maslov class of the loop. Note that a different choice of data may change the degrees, however the $\ZZ_2$-grading is well-defined.

\paragraph{Differential.}
The differential of a chord will count holomorphic strips in the symplectisation of $(V^*,\l^*)$, with Lagrangian boundary condition.
Nevertheless we define moduli spaces of curves negatively asymptotic to several chords, as it will be used later.
\begin{dfn}[Moduli spaces in a cobordism]
Let $L \subset (W, \b)$ be an exact Lagrangian cobordism from $\La^+ \subset (V^+, \l^+)$ to $\La^- \subset (V^-, \La^-)$, and $J$ an adapted complex structure. Then for $c \in \C(\La^+)$ and $w=q_0 c_1 q_1 ... c_r q_r$ a word of chords $c_i \in \La^-$ and homological elements $q_i \in H_1(L)$, we define $\M (c, w; W, L, J)$ the space of $J$-holomorphic curves $F= (a, f): (S, j) \ra (W, J)$, up to reparametrisation of the domain, such that 
\begin{itemize}
\item $S = D\setminus \{z_0, ..., z_r \}$ where the $z_i \in \dd D$ are ordered in the direct order ;
\item $f(\dd D \setminus \{z_i\}) \subset L$ ;
\item $f((z_i, z_{i+1})) =  q_i \in H_1(L)$ (here we need to use the capping paths) ;
\item $F$ is positively asymptotic to $c_0$ at $z_0$ and negatively asymptotic to $c_i$ at $z_i$ (for $i \geq 1$). In other words, $a \under{\ra}{z_0} +\infty$, $a \under{\ra}{z_i} - \infty$ for $i \geq 1$, and there exists local holomorphic coordinates $(x, *)$ near those punctures such that $f(x, .) \under{\ra}{z_i} c_i(.)$ in $\C^\infty ([0, 1], V)$.
\end{itemize}
\end{dfn}

For a symplectisation, we also have to quotient the curves by $\RR$-translation on the target :
\begin{dfn}[Moduli spaces in a symplectisation]  Let $(V, \lambda)$ be a contact manifold, $J$ an adapted complex structure on its symplectisation, and $\Lambda \subset V$ a Legendrian. Then for $c \in \C(\La)$ and $w=q_0 c_1 q_1 ... c_r q_r$ a word of chords $c_i \in \La$ and homological elements $q_i \in H_1(\La)$, we define $\M (c, w; V, \La, J)$ the space of $J$-holomorphic curves $F= (a, f): (S, j) \ra (\RR_s \times V, J)$, up to reparametrisation of the domain {\it and $\RR_s$-translation}, such that 
\begin{itemize}
\item $S = D\setminus \{z_0, ..., z_k \}$ where the $z_i \in \dd D$ are ordered in the direct order ;
\item $f(\dd D \setminus \{z_i\}) \subset \Lambda$ ;
\item $f((z_i, z_{i+1})) =  q_i \in H_1(L)$ (once again we need to choose capping paths) ;
\item $F$ is positively asymptotic to $c_0$ at $z_0$ and negatively asymptotic to $c_i$ at $z_i$ (for $i \geq 1$). \end{itemize}
\end{dfn}

By \cite{BEHWZ03_Compactness_SFT}, the energies of a holomorphic curve (in either a Liouville cobordism or a symplectisation) positively asymptotic to a chord $c_0$, negatively asymptotic to chords $c_i$, and with no other puncture, are given by
\begin{align*}
    & E_{d\lambda}(F) = \A(c_0) - \sum_{i\geq 1} \A(c_i) 
    & E_\lambda(F) = \A(c_0).
\end{align*}

Building upon Gromov's result \cite{Gro85_Hol_curves}, those moduli space can be compactified by adding {\it buildings} of holomorphic curves, see \cite[§7.2]{BEHWZ03_Compactness_SFT}. More precisely, one need to add broken curves modelled on nodal surfaces obtained by degeneration from $S$. For our construction it is enough to consider one-dimensional moduli spaces, which only break into buildings with two levels. 

Thus we define, for $c_0 \in \C_*(\La)$ and $|w| = * -2$, the {\it compactified moduli space} 
\[ \Mbar(c_0, w; V,\La,J) = \M(c_0; w) \cup \bigcup_{\substack{c \in \C_{* - 1}(\La) \\ w_1w_-w_2 = w }} \M(c_0;  w_1 c w_2) \times \M(c; w_-).    \]

For more general asymptotics (without those grading restriction), one need to consider buildings with several levels. \\
In the case of a cobordism $W$, the moduli space must be compactified by adding holomorphic buildings $k_+|1|k_-$, modeled on nodal curves with $k_\pm$ levels in $\RR \times \dd_\pm W$, and one level in $W$ itself. 

\begin{thm}[\cite{Gro85_Hol_curves} \cite{BEHWZ03_Compactness_SFT}] If the holomorphic curves of a moduli space $\M(c, w; V,J)$ are contained in a compact of $V$, then the compactified moduli space $\Mbar(c, w; V,J)$ is indeed compact. \\
Moreover the same statement holds for holomorphic curves in a cobordism.
\end{thm}

The virtual dimension of a moduli space $\Mbar(c, w; V,J)$ is given by 
\[\dim \M(c_0; w; \La) = |c_0| - \under{\sum}{i\geq 1} |c_i| -1,\]
and for curves in a cobordism the $-1$ disappears. This number can also be seen as the Fredholm index of some operator, thus it does not depends on the choice made to define the grading (when the Chern class vanishes). Finally,  for a generic $J$ those moduli spaces are manifolds, and their dimension is equal to the virtual one, see \cite{Diri16_Leg_surg}.

We can now define the differential of a chord 
\[ \dd c^+ = \under{\sum}{|c^+| = |c^-|-1} \# \M(c^+,q_0 c^- q_1; V^*, \l^*, \La_0^*, \La_1^*, J)\ q_0 c^- q_1.  \]
where $q_0,q_1 \in H_1(\La)$.
By \cref{PrincMaxSut} the curves stay in a compact (for fixed asymptotics), so Gromov's compactness result holds and this is indeed a differential.

\begin{rmk}
By the same argument the differential of the Chekanov-Eliashberg dga is well defined, see \cref{ssec.dga_wrapped} for more details.
\end{rmk}

\paragraph{Relative homology.} We can also consider a pair of sutured Legendrians $\La_0, \La_1$, which we again assume hypertight\footnote{If it is not the case, we get a $\LC(\La_0)-\LC(\La_1)$-bimodule.}. Then the {\it relative sutured complex} is generated by Reeb chords going from $\La_0$ to $\La_1$, and the differential counts holomorphic strips in the symplectisation of $V^*$, with boundary on $\RR\times \La_i^*$. The resulting complex is a $\ZZ[H_1(\La_0)]-\ZZ[H_1(\La_1)]$-bimodule, denoted $LC(\La_0, \La_1; V, \l, \Ga, J)$.

\subsubsection{Wrapped sutured homology}

Consider two disjoint sutured Legendrians $\La_0, \La_1 \subset (V, \l, \Ga)$, as well as a tailored complex structure $J$. For simplicity, we once again assume that $\l$ is hypertight, and that both Legendrians are relatively hypertights.

Now choose an Hamiltonian $H : \RR^+ \ra \RR$, inducing a positive and total wrapping. 
As described in \cref{ssec.Completions}, we obtain a non-compact Legendrian completion $\La_0^H$, and by construction
\[\C(\La_0^H, \La_1^* ; V^*, \l^*) = \C(\La_0, \La_1; V, \l) \cup \C(\dd\La_0, \dd \La_1; \Ga, \l_\Ga), \]
and we will call {\it interior} (resp. {\it exterior}) {\it chords} the elements of the first (resp. second) subset. Note that the interior chord are all in $\mr V$, while the exterior chords are contained in $\{\t > 0\}$.

We define $\W LC (\La_0, \La_1; V, \l; J, H)$ as the $\ZZ[H_1(\La_0)]-\ZZ[H_1(\La_1)]$-bimodule generated by Reeb chords $\C(\La_0, \La_1; V, l)$, graded as previously and whose differential counts holomorphic curves 
\[ \dd c^+ = \under{\sum}{|d| = |c|-1} \# \M(c^+,xc^-y; V^*, \l^*, \La_0^H, \La_1^*, J)\ xc^-y.  \]

By \cref{PrincMaxEnrou} the curves again stay in a compact (for fixed asymptotics), so Gromov's compactness result holds and this is indeed a differential.

\begin{rmk}
We could also define a Rabinowitz complex by adding chords going from $\La_1$ to $\La_0$. The differential of such a chord $c$ would count chords {\it negatively} asymptotic to $c$, as well as bananas, see \cite{Alb12_Rabinowitz-Floer}.
\end{rmk}

\begin{rmk} We could also wrap $\La_1$ negatively, to get a quasi-isomorphic complex. \\
We can also work with uncentered sutured Legendrians, in which case one need to ensure that the completion make $\dd \La_0$ go below $\dd \La_1$, which is always possible if we start the wrapping far enough (in the $\t$-direction).
\end{rmk}

\subsection{Other points of view}
We present some others constructions geometrically similar, but whose implementation is more difficult. \\

{\it Wrapping the contact form.} We can perturb $\dd_2 \Lambda$ to make it $-\e$-centered (ie a neighbourhood is $(-1,0]_\t \times \{t=\e\} \times \dd_2 \La$), and complete it cylindrically. We now extend the contact form to the completion $V^*$ by $f dt + e^\t \l_\Ga $ where $f$ is an increasing function such that $e^{-\t}f' \ra \infty$. The Reeb vector field is now directed by $\dd_t - e^{-\t}f'R_\Ga$, so Reeb chords appears in $V^* \sms V$, in bijection with $\C(\dd_1 \La \ra \dd_2 \La; \l_\Gamma)$. One must now prove that the differentials coincident, however the projection of an holomorphic curve to $\RR^+\t \times \Ga$ is {\it not} $J_\Ga$-holomorphic anymore. \\ 

{\it Reduced energy}.
We could also work with cylindrically completed Legendrians, using the usual contact form. The differential should now count more general curves, whose energy might be infinite, but with finite "reduced energy", defined by integrating $\phi(\t)d\t \wedge \l_\Ga$ on $V^* \sms \mr V$, where $\phi$ is a function on $\RR^+$ with finite integral. The chords of the suture now appear at infinity (in the $\t$-direction), however the proper way of compactifying the moduli spaces is not straightforward. Note that with this construction, the Reeb orbits of $(\Gamma, \l_\Ga)$ also appear, which was not the case previously. \\

{\it Counting Floer curves.} Alternatively we could also count curves satisfying a Floer equation. Indeed Floer strips between two Lagrangians, asymptotic to 1-periodic trajectories of an Hamiltonian vector field $X_H$, correspond to pseudo-holomorphic curves between $L$ and $\vphi_H^{-1}(L')$, asymptotic to intersection points. In that Hamiltonian framework, the maximum principle is proved by \cite{AbSe09_Open_string_analogue} (see also \cite{Abo10_Generating_Fuk}), and this theory is equivalent to the one we defined by \cite[Remarque 1.10]{Aur13_Intro_Fuk}. \\

{\it Cancelling Reeb orbits.} Finally, one could make appear the Reeb chords (and orbits) of the dividing set, by creating two families of Reeb orbits in cancelling position, as in \cite{CGHH10_Sutures}. This perspective is more adapted to the definition of a wrapped Chekanov-Eliashberg dga, and will be made more precise in \cref{ssec.dga_wrapped}.

\subsection{Proof of the invariance}\label{ssec.Proof_invariance}

In this section we show that the sutured Legendrian homologies don't depend on some choices we made along the way. 
\subsubsection{Cylindrical sutured homology }

\begin{thm} \label{InvrtCyl}  Let $(V, \N_0(\Gamma), \xi)$ be a sutured contact manifold and $\Lambda$ cylindrical Legendrian totally  non-degenerate. Then $LH(\Lambda; V, \N_0(\Gamma), \l, J)$ does not depend on the choice of adapted contact form, nor on the tailored almost complex structure. \\
Let $(V, \N_0^u(\Gamma), \xi^u)$ be a path of sutured contact manifolds, and  $\Lambda^u$ a path of cylindrical Legendrians $\dd$-non-degenerates such that $\Lambda^0$ and $\Lambda^1$ are totally non-degenerate. Then there is a quasi-isomorphism $L C(\Lambda^0, \l^0) \ra LC(\Lambda^1, \l^1)$, which only depends (up to homotopy) on the homotopy class of the path $(\xi^u, \Lambda^u)$.
\end{thm}
The geometric content of the proof still holds for the Chekanov-Eliashberg dga, however the algebraic statement of the invariance is more involved, see eg \cite[Lemma 5.6]{EkOa15_Sympl_Cont_dga}.

\begin{rmk}
In particular, the choice of coordinates on $\N_0(\Gamma)$ does not matter. Indeed the space of choices is contractible, and any path of coordinates functions induces a isotopy of adapted contact forms.
\end{rmk}

\begin{proof}
We adapt the proof of \cite{CGHH10_Sutures}. It is enough to prove the second part of the theorem, which implies the first point. 
We choose a family of almost complex structures $J^u$ tailored to $(V^*, \l^u)$ (note that the space of tailored structure is contractible).\\

{\it First step} :  We start by assuming that $\l^u$, $J^u$ and $\Lambda^u$ are independents of $u$ on the area $\{\t \geq 0\}$, in which case the proof goes as in the usual compact situation. Consider the cobordism  
\[ (\RR_s \times V^*, \w=d(e^s \l^{\phi(s)} ), \tilde{J} = J^{\phi(s)} ) \]
where $\phi : \RR \ra [0,1]$ is a smooth decreasing function such that $\phi(s) = 1$ for $s \leq -N$ and $\phi(s) = 0$ for $s \geq N$, with $N$ big enough. If $|\phi'|$ is small enough, $\w$ is symplectic and $\tilde{J}$ is adapted. 
Moreover there exists an exact Lagrangian cobordism $L \subset (\RR \times V^*)$ interpolating between the two Legendrians, and coinciding with $\RR \times \{0\} \times \RR^+ \times \dd \La$ on $\{\t > 0\}$.

We then define a map
\[ \varPhi : L C(\Lambda^0, \l^0, J^0) \ra  L C(\Lambda^1, \l^1, J^1) \]
by counting rigid holomorphic curves, with one positive asymptotic and with boundary in $L$, whose projections to $M^*$ stay in a compact.

Indeed \cref{PrincMaxSut} bound the curves horizontally : since $\tilde{J}$ is $s$-invariant near the boundary, the curves stay in $\{\t \leq 0\}$. Moreover if we have a sequence of curves $F_n$ (with fixed asymptotics and topology) as well as $z_n \in \dot{\Sigma}$  such that $t \circ F_n(z_n) \ra \infty$, then similarly to \cref{PrincMaxSut} there would exist a holomorphic curve of finite energy in $\RR \times \RR \times R_\pm$, which is impossible because there is no Reeb orbit in that area. 

The usual argument, which relies on the compactification of those curves, implies that $\Phi$ is a chain map, and that there exists an homotopical inverse.
Indeed the curves in the cobordism break into holomorphic buildings, and by looking at 1-dimensionnal moduli spaces we get the equality $\Phi \dd_0 - \dd_1 \Phi =0$. Similarly, the cobordism from $\l^1$ to $\l^0$ gives a morphism in the opposite direction $\Psi : LC(\Lambda^1, \l^1, J^1) \ra  LC(\Lambda^0, \l^0, J^0)$.
 
Moreover the composition $\Psi \circ \Phi$ can be recovered by counting curves in the cobordism obtained by gluing the two previous cobordisms. Since this new cobordism is homotopic to the trivial one, the study of the 1-dimensionnal moduli spaces of curves show that there exists a map $H$  such that $\Psi \circ \Phi - Id = H \dd_0 - \dd_1 H$.\\

{\it Second step} : We now study the case where $J^u$, $\l^u$ and $\Lambda^u$ change on $\{\t \geq 0\}$. 
We will construct an intermediary sutured contact manifold $(V, \N_0(\Gamma), \tilde{\l}, \tilde{J})$ so we can compare the two complexes.
According to \cref{PrincDuMaxCobord}, there exists a Liouville form $\tilde{\b}_\Ga$, $\tilde{J}_\Ga$ an almost complex structure on $\RR_\t \times \Gamma$, and $L \subset \RR^+ \times \Ga$ an exact Lagrangian interpolating between $(\dd \La^0, \l_\Gamma^0,J_\Gamma^0)$ and $(\dd \La^1, C \l_\Gamma^1,J_\Gamma^1)$, and  such that the $\tilde{J}_\Ga$-holomorphic curves with boundary in $L$ do not present a $\t$-maximum.

We then define $(\tilde{\l}, \tilde{J})$ by
\begin{itemize}
\item $\tilde{\l} = \l^0$ on $\{\t <0\}$
\item $\tilde{\l} = dt + \tilde{\b}_\Ga$ for $\t > 0$ 
\item $\tilde{J} = \tilde{J}_\Ga$ on $\ker \tilde{\b}_\Ga$
\item $\tilde{J} (\dd_a) = \dd_t$
\item $\tilde{J} (\dd_\t) = \tilde{J}_\Ga\dd_\t + \mu \dd t \in \ker\tilde{\l}$, where  $\mu = - \tilde{\l}_\Ga (\tilde{J}_\Ga \dd_\t)$.
\end{itemize}  
The manifold $(\tilde{V}=\{\t < S, |t| \leq 1\}, \tilde{\l})$ is sutured, $\tilde{J}$ is a tailored structure and $\tilde{\Lambda}$, Legendrian lift of $L$, is a cylindrical Legendrian totally non-degenerate. 

Moreover $\tilde{J}$ lifts the structure $\tilde{J}_\Ga$ : \[d\Pi \circ \tilde{J} = \tilde{J}_\Ga \circ d\Pi,\]
where $\Pi$ is the projection on $\RR^+_\t \times \Ga$, hence any  $\tilde{J}$-holomorphic  curve projects to a  $\tilde{J}_\Ga$-holomorphic curve. In particular, by \cref{PrincDuMaxCobord} they stay in the zone $\{\t < 0\}$.
We then get
\[ L C(\Lambda^0; V, \N_0(\Gamma), \l^0, J^0) = L C(\tilde{\Lambda}; \tilde{V}, \tilde{\l}, \tilde{J})  \]
as complexes, because generators and differentials correspond.

Finally there exists a path between $\tilde{\Lambda} \subset (\tilde{V}, \tilde{\l}, \tilde{J})$ and $\Lambda^1 \subset (\tilde{V}, \l^1, J^1)$ which is fixed at the boundary, and the previous step yields a quasi-isomorphism beteween the associated dgas.\\

{\it Third step} : We now show that two homotopic paths induce the same isomorphism (in homology). 
We start by recalling the usual case,  where $V$ is compact without boundary : we can glue the induced cobordisms to obtain a cobordism from $\Lambda^0 \subset (V, \l^0, J^0)$ to himself. Since both paths are homotopic, this cobordism is homotopic to the trivial one. By counting holomorphic curves in this cobordism we get a morphism $F : LC(\Lambda^0; \l^0, J^0) \ra LC(\Lambda^0; \l^0, J^0)$ homotopic to the identity, meaning that there exists $H : LC_*(\Lambda^0; \l^0, J^0) \ra LC_{* + 1}(\Lambda^0; \l^0, J^0)$ such that $F - Id = H \dd + \dd H$, hence $H_*(F)= Id$. Moreover $F$ can also be obtained as the composition of $F_0$ and $G_1$, where $F_0 : LC(\Lambda^0; \l^0, J^0) \ra LC(\Lambda^1; \l^1, J^1)$ is induced by the first isotopy and $G_1 : LC(\Lambda^1; \l^1, J^1) \ra LC(\Lambda^0; \l^0, J^0)$ by the reversed second isotopy. Hence we get $Id = H_*(F) = H_*(G_1) \circ H_*(F_0)$ and by composing with $H_*(F_1)$ we indeed have $H_*(F_1) = H_*(F_1) H_*(G_1) H_*(F_0) = H_*(F_0)$ (the composition $F_1 \circ G_1$ can also be obtained by gluing a cobordism to its opposite, so this map is homotopic to the identity).

For sutured manifolds, the algebraic argument is the same, although we now need to check that the holomorphic curves stay in a compact. We start by choosing contact forms adapted to $\xi^0$ and $\xi^1$, as well as tailored almost complex structures. As previously, we can construct a sutured contact manifold $(\tilde{V}, \tilde{\l}, \tilde{J})$ with a Legendrian $\tilde{\Lambda}$  such that 
\begin{itemize}
\item $LC(\tilde{\Lambda}; \tilde{V}, \tilde{l}, \tilde{J}) = LC(\Lambda^0; V, \N_0(\Gamma), \l^0, J^0)$ as dga.
\item there exists an isotopy from $\tilde{\Lambda} \subset (\tilde{V}, \tilde{l}, \tilde{J})$ to $\Lambda^1 \subset (V, \l^1, J^1)$ fixed at the boundary.
\end{itemize}
Hence we can apply the arguments of the compact case to this setting : the cobordisms used to define the morphisms are constants at the boundary, and the homotopy to the trivial  cobordism can be realised without modification near the boundary. Hence all the curves counted satisfy a maximum principle in the $\t$-coordinate.
\end{proof}

\begin{rmk}
We emphasise that the cobordism  maps are not  defined for any path of tailored almost complex structure, but only for those with a nice behaviour near the boundary.
\end{rmk}

Using the construction of \cref{lem.cvx->sut}, this homology  becomes an invariant of manifolds with (smooth) convex boundary.

\begin{cor}\label{Cor.InvrtBordLisse} Let $(V, \xi)$ be a contact manifold with convex boundary, $X$ a contact vector field transverse to $\dd V$, and $\La$ a Legendrian transverse to $\dd V$ such that $\dd \La \subset \Ga_{X}$. Then the sutured homology  of the Legendrian obtained by the convex-sutured operation only depends on $(V, \xi, \Ga_X, \La)$.

 Let $\xi^u$ be a path of contact structures on $V$  such that $\dd V$ is convex for all $u$ ;  $X^u$ a path of contact vector fields transverses au boundary, and $\Lambda^u$ a path of Legendrians transverses to the boundary and  such that $\dd \La^u \subset \Ga_{X^u}$.
Then the homologies $LH(V, \xi^0, \Ga_{X^0}, \La^0)$ and $LH(V, \xi^1, \Ga_{X^1}, \La^1)$ are quasi-isomorphic, and the map only depends the homotopy class of the path $(\xi^u, \Ga_{X^u}, \La^u)$.
\end{cor}
In particular if $\La \subset (V, \xi)$ is a Legendrian without boundary, then the homology does not depend on the suture.

Moreover, and as noticed in \cite{CGHH10_Sutures}, for disjoint Legendrians $\La, \La_0 \subset V$ the sutured homology  $LH(\La, V \sms \N(\La_0))$ is an invariant of $(\La, \La_0, V)$, where $\N(\La_0)$ is a standard neighbourhood with convex boundary, disjoint from $\La$.

\begin{proof}
The set of sutures is contractible, because the set of contact vector fields transverse to $\dd V$ also is : if $X_0, X_1$ are two contact vector field transverse to the boundary, the vector fields
\[ X_t = t X_1 + (1-t) X_0, t \in [0,1]  \]
stays contact and transverse to the boundary. 
Moreover the convex-sutured operation is continuous relatively to a path of contact structures adapted to the boundary. In other word a path of such structures yields a path of sutured contact manifolds.
\end{proof}

\subsubsection{Wrapped sutured homology }

Similarly to \cite{AbSe09_Open_string_analogue} (see also \cite{Aur13_Intro_Fuk}), the wrapped complex can also be constructed as a colimit on partial wrappings.  
For a positive and total Hamiltonian $H_\ity$, we choose a sequence of Hamiltonians $(H_k)_{k\in\NN^*}$  such that 
\begin{itemize}
\item $H_k$ induces a positive wrapping (ie $H_k''\geq 0$) ;
\item $H_k = H_\ity$ on $[0, k]$ ;
\item $H_k'' = 0$ for $\t \geq k+ 1$.
\end{itemize}

We obtain that way a sequence of cylindrical Legendrians $\La^k$, and by \cref{PrincMaxEnrou} we have inclusions of complexes $LC(\La^k) \subset LC(\La^{k+1})$.
Hence we get 
\[\W LC (\La^\ity, V, \l; J, \sp, \T) \simeq \under{\colim}{k \ra \ity} LC(\La^k; V, \l; J, \sp, \T).\]

\begin{rmk}
This is still well-defined if we work with dgas, because the morphisms are only inclusions of dga.
\end{rmk}

Thus the wrapped homology does not depends on the choice of $H$. Furthermore, it is invariant along a path of sutured Legendrians

\begin{thm}\label{InvrtEnroul} Let $(V, \N_0^u(\Gamma), \xi^u)$ be a path of sutured manifold, and 
\begin{itemize}
    \item $\La_0^u \subset (V,\xi^u)$ a path of sutured Legendrians
    \item $\La_1^u \subset (V,\xi^u)$ a path of \emph{uncentered} Legendrians, so that 
    \[\pi_\Ga(\dd \La_0) \cap \pi_\Ga(\dd \La_1) \neq \emptyset \Ra \t(\dd\La_0) < \t(\La_1). \]
\end{itemize}
Pick $\l^0, \l^1$ contact forms adapted to $(V,\xi^0)$ and $(V, \xi^1)$, and $J^0,J^1$ tailored almost complex structures. 
Then, for properly chosen Hamiltonian $H^0,H^1$, there is a quasi-isomorphism 
\[ \W LC(\La_0^0, \La_1^0; V, \l^0, J^0, H^0) \lra \W LC(\La_0^1, \La_1^1; V, \l^1, J^1, H^1) \]
which only depends (up to homotopy) on the homotopy class of the  path $(\xi^u, \La^u)$. \\
In particular, the wrapped homology  $\W LH(\La_0, \La_1; \l, J, H)$ is independent of the choices of $\l$, $J$ and $H$.
\end{thm} 

Let us emphasise that we don't require of $\La^u$ to be sutured during the isotopy : the projections of $\dd_i \La^u$ to $\Gamma$ {\it can intersect}, as long as $\dd_0 \La$ pass below $\dd_1 \La$.

\begin{proof}
If the boundary stays invariant during the isotopy, it induces an exact Lagrangian cobordism in $\RR \times V^*$ (\cref{ssec.ConstrCobordLagr}), which defines a quasi-isomorphism between the complexes. 

Now if the boundary is affected, we get an isotopy of the boundary $\dd_0 \La^s \subset (\Ga, \l_\Ga)$, so we can construct an exact Lagrangian  cobordism $L_\Ga \subset [0,1] \times \Ga$ interpolating between $\dd_0 \La^0$ and $\dd_0 \La^1$. After liffing to the contactisation, we get a Legendrian $\hat \La_0 \subset \{\t \leq 1\}$, of boundary $\dd \La^1$, that we wrap by an Hamiltonian $H$. Hence by the previous point 
\[ \W LC(\La^1) \simeq \W LC(\hat \La_0).  \]

To show that this complex is quasi-isomorphic to $LC(\La^0)$, we construct a sequence of Legendrians $\hat \La_T$, fixed at infinity (in the direction $\t$), and  such that 
\begin{itemize}
    \item $\hat \La_T$ coincides with $(\La^0)^H$ on $\{\t \leq T\}$,
    \item $\hat \La_T$ coincides with $(\La^1)^H$ on $\{\t \geq T+1\}$.
\end{itemize}
Loosely speaking, the isotopy of the boundary, which is finite, is "absorbed" by the wrapping.

To construct this sequence, we define
\begin{itemize}
    \item on $[0, T] \times \Ga$,  $L_T = \under{\cup}{\t \in [0, T]} \{\t\} \times \phi_{R_\Ga}^{H'(\t)}(\dd_0 \La^0)$
    \item on $[T, T+1] \times \Ga$, $L_T = \phi^{H'(T)}(L_\Ga)$
    \item on $[T+1, \ity)$, we wrap by Hamiltonian of $\t$.
\end{itemize}
We obtain an exact Lagrangian, which lifts to a Legendrian $\La_T \subset V^*$. Moreover, $(\l_\Ga)_{\rest L} = df$, where $f$ is a function of $\t$ on $\RR^+ \sms [T, T+1]$.
Hence for any $A > 0$, there exists $T > 0$ such that all chords of $\La_T$ of action bounded by $A$ are in $\{\t \leq T\}$. According to \cref{PrincDuMaxCobord} (which still holds for an immersed Lagrangian), a curve positively asymptotic to such a chord  stays in that area. Hence we get an equality of complexes
\[ \W LC^A(\La^0) = \W LC^A(\La_T). \]
Moreover, because $\La_T$ coincide with $(\La^1)^H$ at infinity, we also have 
\[\W LC(\La_T) \simeq \W LC(\La^1).\] 
We obtain the desired quasi-isomorphism by taking the colimit of those applications
\[ \W LC^A(\La^0) = \W LC^A(\L_T) \subset \W LC(\La_T) \over{\sim}{\lra} \W LC(\La^1)  \]
for $A \ra \ity$.
\end{proof}

\subsection{Exact sequence} \label{ssec.Exact_seq}

Similarly to \cite{Ekh09_Q-SFT_lin-CH_lagr-FH}, the cylindrical and wrapped homologies fit into an exact sequence which involves the homology of the boundary.
By \cref{PrincMaxSut}, interior Reeb chords form a subcomplex of the wrapped sutured complex, thus we define the {\it exterior complex} as
\[LC^\ext (\La_0, \La_1; \l, J, H) = \W LC (\La_0, \La_1; \l, J, H) / LC(\La_0, \La_1; \l, J). \]
We get an exact triangle  
 \[    \lra  L H(\La_0,\La_1; V, \xi) \lra \W L H(\Lambda_0 , \La_1 \La; V, \xi) \lra L H^\ext (\La_0, \La_1 \Lambda; \Gamma, \xi_\Ga) \over{[-1]}{\lra} 
\]
where the last map is a part of the differential of $ \W L H(\La_0, \La_1; V, \xi)$.

We now formulate some expected results about the exterior homology, which should extend the theorems of \cite{Ekh09_Q-SFT_lin-CH_lagr-FH}. In \cref{sec.Invariants_2-braids}, it will trivially true for the particular case at hand. 

\begin{conj}
\begin{enumerate}
    \item The previous quotient is the bilinearized homology of boundary :
    \[ LC^\ext (\La, \dd_0 \La, \dd_1 \La; \l, J, H) \simeq  LC_\e(\dd_0 \La, \dd_1 \La; \Ga, \l_\Ga, J_\Ga)[1] \]
    where this complex has coefficients in $H_1(\La)$, induced by the inclusion $\dd \La \hra \La$, and $\e$ are augmentations that will be defined in the next section. 
    \item The triangle is invariant in the following sens : let $\La^u$ be a path of sutured Legendrians such that 
    \begin{itemize}
         \item the boundary determines a Legendrian loop : $\dd_0 \La^0 = \dd_0 \La^1$ and $\dd_1 \La^0 = \dd_1 \La^1$ ;
        \item for $i \in \{0,1\}$ (ie at the ends of the path) $\La_i^0$ and $\La_i^1$ are hypertights.
    \end{itemize}
Then the following diagram commutes :
\[ \begin{xymatrix} {
     LH(\La_0^0, \La_1^0; V) \ar[r] \ar[d]_{F_*}^\wr & \W LH(\La^0_0, \La_1^0; V) \ar[r]^-{[-1]} \ar[d]_{F_*^W}^\wr & LH_\e(\dd \La_0, \dd \La_1; \Ga) \ar[r]^-{\d^0} \ar[d]_{F_*^\dd}^\wr & \\
      LH(\La_0^1, \La_0^1; V) \ar[r] & \W LH(\La_0^1, \La_1^1; V) \ar[r]^-{[-1]} & LH_\e(\dd \La_0^1, \dd \La_1^1; \Ga) \ar[r]^-{\d^1}  & }
    \end{xymatrix} \]    
    where the maps $F$ and $F^W$ come from \cref{InvrtCyl} and \cref{InvrtEnroul}, and $F^\dd$ is  induced by a  Lagrangian cobordism in $(\RR \times \Ga, e^s \l_\Ga)$, determined by the path $\dd \La^u \subset (\Ga, \l_\Ga)$. In particular if the boundary fixed along the path, $F^\dd_* = \Id$.
\end{enumerate}
\end{conj}

In the next section we will show that in our particular case,
the triangle is indeed invariant along a path of sutured Legendrians fixed at the boundary (the situation will be considerably simplified due to homotopical restrictions). \\

Before moving on to this example, let us note that, when $V$ is the contactisation of Liouville domain, the first and second points are consequences of Seidel's isomorphism \cite{Ekh09_Q-SFT_lin-CH_lagr-FH}, combined with \cite{Diri16_Lift} (see also \cite{BoCh12_Bilin_&_augm}). The situation is also reminiscent of \cite{PaRu20_Aug_&_Immers_lagr_fill} where, given a sutured Legendrian $\dd \La$ in $J^1(S^1)$, a map from the dga of $\dd \La$ to the dga of $\La$ is defined, see  \cref{ssec.dga_wrapped} for a generalisation of this construction.

The main issue toward a proof of those conjectures is that, while the contact homology of $(\dd \La, \Ga)$ involves curves in the symplectisation of $\Ga$, we did our best to avoid this situation, and only worked with curves bounded in the $\t$-direction. Luckily, there exists an alternate definition of the Chekanov-Eliashberg dga \cite{EkOa15_Sympl_Cont_dga}, which only involves (families of) Floer curves bounded in the $\t$-coordinate. However the moduli spaces used are significantly more complicated, and although no fundamental issue should arises, it does not fit the scope of our result.

\subsection{"Wrapping" the Legendrian dga} \label{ssec.dga_wrapped}
We now briefly explain how to adapt our constructions to the Chekanov-Eliashberg dga. Let us first recall the classical definition, which still holds in the sutured setting.
Consider a sutured Legendrian $\La \subset (V, \l, \Ga)$, and assume that $\l$ is hypertight to avoid dealing with Reeb orbits. Following \cite{Che02_Dga_leg_links}, its Chekanov-Eliashberg dga is defined by picking a tailored almost complex structure $J$, and counting holomorphic curves (with several negative punctures) in $\RR \times V^*$, with boundary in $\RR \times \La^*$.

More precisely, denote $\LC (\La, V, \l, J)$ the unital differential graded $\ZZ_2$-algebra, generated by the Reeb chords $\C(\La, V, \l)$ and the homological variables $\ZZ[H_1(\La,\ZZ)]$ (the chords do not commute with the homological elements), graded as in \cref{ssec.DfnInvariants}.
The differential is defined on Reeb chords by counting holomorphic curves
\[\dd c  = \under{\sum}{|w| = |c|-1} \# \M(c; w; V^*, \Lambda^*; J)\ w, \]
and extended by the Leibniz rule  $\dd (ab) = \dd a. b + (-1)^{|a|}a. \dd b $.

As in \cref{sssec.Cylindrical_sutured_homology}, this is indeed a differential, and the homology is independent of the choices made along the way. More generally, a path of sutured Legendrians $\La^u \subset (V, \xi^u)$ induces a quasi-isomorphism between the dgas (in the sense of \cite[Lemma 5.6]{EkOa15_Sympl_Cont_dga}), which only depends on the homotopy class of $(\xi^u, \La^u)$.

We now explain how to get a wrapped version of this dga. We first modify the contact form near the suture, as in \cite[§4.2]{CGHH10_Sutures} : on the neighbourhood
\[\N \simeq (-1,0]_\t \times [-1,1]_t \times \Ga, \]
take $\tilde \l = f(\t, t) dt + g(\t, t) \l_\Ga $, such that
\begin{itemize}
    \item the contact condition is satisfied : $g^{n-1}(g\dd_\t f - f \dd_\t g) > 0$ ;
    \item $\dd_\t f \geq 0$  ;
    \item near the boundary $f= 1$ and $g = e^ \t$ ;
    \item $g$ presents two cancelable critical points : a maximum at $(-2\e, 0)$ and a saddle point at $(-\e,0)$.
\end{itemize}

The Reeb vector field is directed by $X_g + \dd_\t f R_\Ga$, where $X_g$ is the Hamiltonian vector field associated to $g$, defined by $\i_{X_g} d\t \wedge dt = dg$, see \cref{fig.foliation_wrapped_dga}.

With this modification we have created two families of Reeb orbits, one for each critical point of $g$, which correspond to orbits of $(\Ga, \l_\Ga)$ (we also create other orbits around the maximum, but their action can be taken as big as needed). However there is also a grading shift happening, because for orbits the definition involves the dimension of the ambient manifold. For orbits over the saddle point, the Conley-Zehnder index is unchanged and the dimension increased, so the resulting degree increases by one. Similarly, the degree of an orbit above the maximum increases by two. To avoid dealing with this phenomena, we now assume that $(\Ga, \l_\Ga)$ is hypertight (for this construction this hypothesis is actually important, and not only a simplifying assumption). \\

The sutured Legendrian $\La \subset (V, \l, \Ga)$ stays Legendrian during this process,
and we have again created two families of Reeb chords, corresponding to chords of $\C(\La_\Ga, \Ga, \l_\Ga)$ (as well as other long chords which can again be ignored by an action-filtration argument). For a chord $c$ in $\Ga$, we will denote $\check c$ (resp. $\hat c$) the induced chord above the saddle point (resp. maximum). The degrees of the chords are given by :
\[ |\check c|_{V} = |c|_\Ga \qquad \qquad  |\hat{c}|_V +1 = |c|_\Ga.\]
Thus, and contrary the case of Reeb orbits, the degree is unchanged for the family of chords corresponding to the saddle point (compare also to \cite{CoEk17_Lagr_fillings_Complicated_unknots}, where the degree of the chords also increases).

\begin{figure}[h!]
\center
\def\svgwidth{15cm}
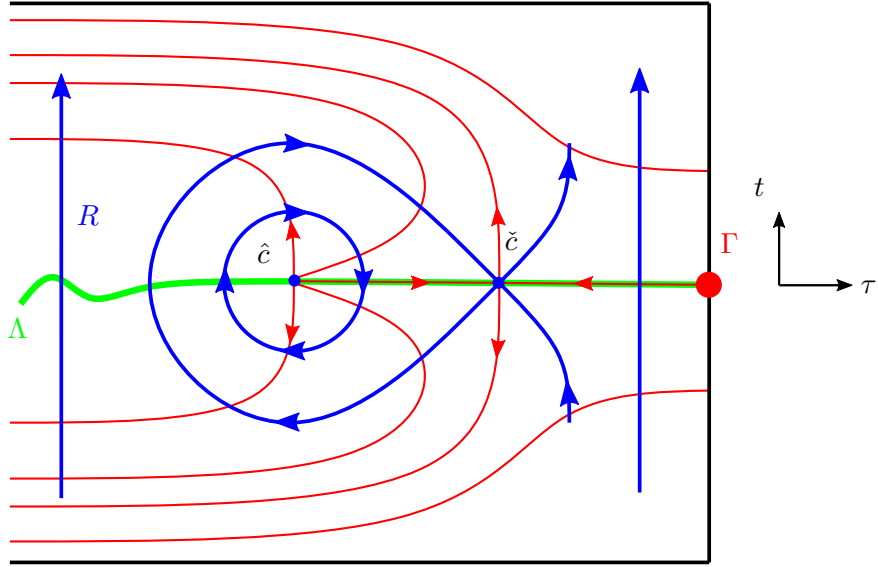
\caption{Perturbation of the Reeb vector field, creating two families of chords. In green, the sutured Legendrian ; in blue, the projection $X_g$ of the Reeb vector field ; in red, the gradient lines of $g$, which lift to holomorphic hypersurfaces in the symplectisation. }
\label{fig.foliation_wrapped_dga}
\end{figure}

\paragraph{Holomorphic foliation.}
We now show that if the manifold is S-sutured, the generators $\check c$ form a sub-dga, by constructing an holomorphic foliation near the boundary. See also 
\cite{AsEk21_Chekanov-Eliashberg_dga_singular} for a stopped point of view (in our language, it means that their manifolds are balanced).

We first choose two almost complex structures $J_\pm$ on $(R_\pm, \b_\pm)$, which coincide on their common boundary $(\Ga, \l_\Ga)$, and such that $\b_\pm \circ J_\pm$ are exact 1-forms. Denoting $J_\Ga$ the restriction to $\xi_\Ga$, we construct a tailored almost complex structure on the symplectisation of $(V^*, \l^*)$, given on the perturbed area by
\begin{align*}
    J :\ & \dd_s \mt R = \mu(X_g + f_\t R_\Ga) & \qquad \text{where } \mu = 1/(g \dd_\t f - f \dd_\t g) > 0\\
    & \dd_\t \mt \dd_t - \frac fg R_\Ga & \\
    & \xi_\Ga \over{J_\Ga}{\lra} \xi_\Ga. 
\end{align*}
Away from the perturbation, ie on $\{\t < 3\e\}$, $J$ is defined by lifting the Stein almost complex structures. Thus the hypersurfaces $\{t = t_0, \t\leq 3\e\}$, where $t_0$ is close of $\pm1$, lift to holomorphic submanifolds of codimension 2 in $(\RR_s \times V^*,J)$. Indeed, the map 
\[ x \in R_\pm \mt (\bar s(x),t_0,x) \in (\RR_s \times \RR_t\times R_\pm,J)   \]
has an image whose tangent space is preserved by $J$ iff $d\bar s = \b \circ J_\pm$, which is exactly the Stein condition.

We now extend this foliation to the perturbed area. We compute 
\begin{align*}
    J R_\Ga & = J\big(g R + g\mu (g_\t (\dd_t - \frac fg R_\Ga) - g_t \dd_\t) \big) \\
    & = - g \dd_s + g\mu \big(- g_\t \dd_\t - g_t (\dd_t - \frac fg R_\Ga)\big)
\end{align*}

Thus the image of
\[ (r, y) \in (-3\e, 0) \times \Ga  \mt (\bar s(r), \t = h(r), t =t_0+r,  y) \in \RR \times V^* \]
is $J$-holomorphic iff $\dd_s + \mu (g_\t \dd_t + g_t \dd_t)$ is tangent to it. We take
\begin{itemize}
    \item $h'(r) = g_\t/g_t$, in other words the map projects to a gradient trajectories of $g$ in the $(\t, t)$ coordinates ;
    \item $\bar s'(r) = \frac1{\mu g_t}$, so $s$ goes to infinity as $r$ increases to zero.
\end{itemize}
Those maps extend the previous lifts of ${\t_0} \times R_\pm$, and by translating in the symplectisation direction, we get a $\RR$-family of codimension 2 holomorphic submanifolds in $(\RR_s \times V^*,J)$.

Moreover, $\RR_s \times \{\t=-2\e, \t=0\} \times \Ga$ and $\RR_s \times \{\t=-\e, t=0\} \times \Ga$ are also holomorphic, and project to the critical points of $g$. Finally the image of the map
\[(s,y) \in \RR \times \Ga \mt (s, \bar \t(s), t=0, y) \in \RR \times V^*\]
is holomorphic iff $\bar \t' = - \mu g_\t$, thus by translating in the symplectisation direction we get two $\RR$-families of codimension 2 holomorphic submanifolds : one projecting to the gradient line $\{-2\e \leq \t \leq u, t=0\}$, such that $s$ goes to $+\ity$ (resp. $-\ity$) as $\t$ goes to $-2\e$ (resp. $-\e$) ; another projecting to $\{\t > -\e, t=0\}$, such that $s$ goes to $-\ity$ (resp. $+\ity$) as $\t$ goes to $-\e$ (resp. $+\ity$).

Hence we get a codimension 2 holomorphic foliation of $(\RR \times V^*,J)$ on a neighbourhood of the boundary (and which extend to the completion), see \cref{fig.foliation_wrapped_dga}. By positivity of the intersections, and using the results of \cite{Sie11_Intersect_th_curves} (see also \cite{Roux17_Weinst_conj}), an holomorphic curve positively asymptotic to a chord $\check c$ must be entirely contained in $\{\t=0, t = -\e\}$. 

Thus the generators $\check c$ form a sub-dga of $\LC(\La, V, \tilde \l)$, isomorphic to $\LC(\dd \La, \Ga, \l_\Ga)$, which yields the following result

\begin{thm}
For a contact manifold S-sutured, endowed with an hypertight contact form, we get an inclusion of dgas
\[ \LC(\dd \La, \Ga, \l_\Ga) \lra \LC(\La, V, \tilde \l).\]
In particular, an augmentation on $\LC(\La, V, \tilde \l)$ induces an augmentation on the dga of the boundary $\LC(\dd \La, \Ga, \l_\Ga)$. 
Moreover an isotopy $\La^u$ of sutured Legendrians fixed along the boundary yields a commutative diagram
\[\begin{xymatrix}{
    \L C(\dd \La, \Ga, \l_\Ga) \ar[r] \ar[d]_{\Id} & \L C(\La^0, V, \tilde \l) \ar[d]^\sim\\
    \L C(\dd \La, \Ga, \l_\Ga)\ar[r] & \L C(\La^1, V, \tilde \l). }
    \end{xymatrix}\]
\end{thm}
Note that the first point already appeared in \cite{PaRu20_Aug_&_Immers_lagr_fill} for $1$-jet spaces, where it was proven using Morse flow trees, and in
 \cite{AsEk21_Chekanov-Eliashberg_dga_singular} for balanced manifolds, where the proof involved using a similar foliation for stopped Weinstein manifolds.

If the boundary $(\dd \La, \Ga, \l_\Ga)$ is relatively hypertight, we can define the {\it wrapped Chekanov-Eliasberg dga} of $\La$ as the quotient by the ideal generated by the boundary chords $\check c$ :
\[ \W \LC(\La, V, \tilde \l) = \LC(\La, V, \tilde \l) /  \<\check c, c \in \C(\dd \La, \Ga, \l_\Ga)\>. \]
Although no actual wrapping appears in that construction, we still think about it as a wrapped dga. Indeed, adding the Reeb chords of the boundary at infinity would create a third family of chords, which would cancel out with the chords of the quotient. In other words, and loosely speaking, wrapping the boundary is equivalent to creating two families of chords in cancelling position, and taking the quotient by the induced sub-dga.

\section{An invariant of local 2-braids} \label{sec.Invariants_2-braids}

\subsection{Braid groups and conormal construction} \label{ssec.braids_gp}

For completeness sake, we first give a presentation of the group of 2-braids, in a surface $S$ of genre $g$, due to \cite{Bel01_Surf_braid}, see also \cite{Sco70_Braid}. 
\begin{dfn} Fix points $x, y \in S$. 
 A 2-braid is a pair of maps $f_1, f_2 : [-1, 1] \ra S$  such that 
\begin{itemize}
    \item $f_1(-1) = x, f_2(-1)=y$ and $\{f_1(1), f_2(1)\} = \{x, y\}$ 
    \item for any $t \in [-1, 1]$, $f_1(t) \neq f_2(t)$.
\end{itemize}
Its realisation is the submanifold 
\[\{(t, f_i(t), t \in [-1, 1], i\in \{1, 2\}\} \subset [1, 1] \times S.\]
A braid will be called {\it pure} if $f_1(1)= x$ and $f_2(1) = y$, ie if the induced permutation is trivial. 
\end{dfn}

We will denote $\B_2(S)$ the group of 2-braids up to homotopy, with the group law induced by the concatenation.

\begin{thm}\cite[Thm 5.4]{Bel01_Surf_braid}  The group of 2-braids is generated by $\s, a_i,$ where $i \in \{1, ..., 2g(S)\}$, with the relations 
\begin{itemize}
    \item $\s^{-1} a_i \s^{-1} a_i = a_i \s^{-1} a_i \s^{-1}$ 
    \item $a_i \s^{-1} a_j = \s a_j \s^{-1} a_i \s $ where $i < j$
    \item $\prod a_{2i-1} a_{2i}^{-1} \prod a^{-1}_{2i-1} a_{2i} = \s^2$
\end{itemize}
\end{thm}

For any $w \in \B_2(S)$, we denote $B_w$ geometric realisation, and a braid will be called {\it local} if it is a power of $\sigma$. Note that the subgroup of pure braids is the kernel of  $\d : \B_2(S) \ra \ZZ_2$, induced  by $\d(\s) = -1, \d(a_i) = \d(b_i) = 1$.
Moreover this subgroup is generated by (see again \cite{Bel01_Surf_braid}) :
\[t = \s^2, a_i, A_i = \s^{-1} a_i \s^{-1}\] 
with the relations \begin{itemize}
    \item $a_i A_i = A_i a_i$
    \item $a_i A_j = t A_j a_i$
    \item $\prod a_{2i-1} a_{2i}^{-1} \prod a^{-1}_{2i-1} a_{2i} = t$.
\end{itemize}

\subsection{Adapted contact form and 1-jet space} \label{ssec.Braid_1-jet}

Consider a surface $S$ which is {\it not} a sphere. We start by computing the linearised Legendrian homology of two fibers $\La_x, \La_y \subset US$, endowed with the contact form  $\l_S$ induced by a metric of constant courvature. 
Choose two generators $\mu_x \in H_1 (\La_x)$ and $\mu_y \in H_1(\La_y)$.
Since Reeb chords correspond to geodesics, there is one by element of $\pi_1(S)$. There is no contractible chord, hence there exists neither holomorphic disk or strip, because the (linearised) differential preserve the homotopy class. Hence
\[LH(\La_x , \La_y ; U_g S) = \under{\oplus}{\g \in \pi_1 (S)} \ZZ[\mu_x].c_\g.\ZZ[\mu_y] \]

\begin{rmk}
The non-linearised differential also vanishes : a geodesic minimises the length in its homotopy class, the length of a geodesic is the action of the corresponding Reeb chord, and the differential strictly decreases the action. 
\end{rmk}

\subsubsection{Sutured manifold}

We now study a braid $B \subset M=[-1,1] \times S$, such that the functions $f_i$ are constants on a neighbourhood of the boundary.  

As presented previously in \cref{ssec.ExVarSut}, the unit bundle of a manifold (with boundary) has convex boundary : the metric $g= dz^2 + g_S$ induces a contact form on $V =U_{g}M$, which is adapted to the suture $\dd I \times U_{g_S} S$. Moreover the unit conormal $\La_B \subset U(I \times S)$ is a Legendrian with boundary included in the dividing set.

However, we prefer to see this manifold of a the thickening of some convex hypersurface (which was also described in \cref{ssec.ExVarSut}).
Presenting the manifold as $V \simeq I \times (DS \cup DS)$, we fix coordinates
\[V =  \{(u, \nu; z, \eta) \in [-U-\e, U+\e] \times  \RR  \times S \times T_z S\ |\  \nu^2 + g(\eta)^2 = 1 \} \]
in which the contact form is $\l = -\nu du - g_S(\eta, dz)$. 
In other words 
\[ \l = - \nu du - \sqrt{1-\nu^2} \l_\Ga \]
where $\l_\Ga$ is the contact form on $U S$ associated to the metric $g_S$.
With those coordinates, the conormal of a strand $f : [-U, U] \ra S$ is a cylinder
\[ \La_f= \{(u, z= f(u), \nu, \eta), \nu+\dot f(u)\cdot \eta = 0\}  \]
In what follows, we will assume that the braid is constant outside of the area $\{\frac 35 U \leq u \leq \frac 45 U\}$.

\begin{rmk}
We can also think of $V$ as a compact contact manifold $\tilde V$, containing two Legendrian stops (we recover $V$ from $\tilde V$ by removing a standard neighbourhood). Then $\tilde V$ is contactomorphic to the open book of page $D^*\Si$ with trivial monodromy, and the stops are the skeletta of two pages. Moreover the conormal of each strand can be extended into a Legendrian sphere in $\tilde V$, intersecting each stop in exactly one point. 
\end{rmk}

We now construct a  contact form adapted to the sutured manifold, and such that the Legendrian is preserved. We define
\[ \tilde \l = -F(u, \nu)\nu du - G(u, \nu) \l_\Ga,  \]
where $F, G$ are smooth functions  such that
\begin{itemize}
    \item $G$ is  positive function ;
    \item the contact condition is satisfied ;
    \item if $|u| \leq U-\e$, then $F= G$ ;
    \item for $|\nu| \leq \e_\nu $ and $|u| \leq  U-\e$, $G = G(u)$ and presents a minimum at $0$ ;
    \item for $|\nu| \geq 3 \e_\nu$, we have $F=1$ and $G = \sqrt{1 - \nu^2}$ ;
    \item $G$ presents two maximums at $(\pm U\pm \e, 1)$ ;
    \item for $U-\e \leq |u| \leq $ and $|\nu| \leq \e_\nu $, we have $F = 1$ and $G = e^{\e_0 |u-U|}$.
\end{itemize}

The Hamiltonian vector field $X_G$ associated to this function is depicted on \cref{fig.Reeb_2-tresse}. The resulting sutured manifold is then 
\[\tilde V = \{G \leq G(U) \}. \]

The associated Reeb vector field is given by (see \cref{ssec.ExVarSut})
\begin{align*}
  R & = X_G+ \dd_\nu (\nu F) R_{US} \\
   & = X_G + (F + \nu \dd_\nu F) R_{US}, 
\end{align*}
so $\tilde{\l}$ is a contact form adapted to the sutured manifold.

The conormal of a strand $f : [-U-\e, U+e] \ra S$, constant outside of $\{\frac35 \leq u \leq \frac45\}$, is given by  
\[\La_f = \{(u, z=f(u), \nu, \eta) |\ \nu + \dot f(u) \cdot \eta = 0.   \}\]
Moreover, after increasing $U$, we can assume that this  Legendrian is included in the region $\{ |\nu|\leq \e_\nu\}$, hence it is unchanged along an isotopy between $\ker \l$ and $\ker \tilde \l$.

\begin{figure}[h!]
\center
\def\svgwidth{10cm}
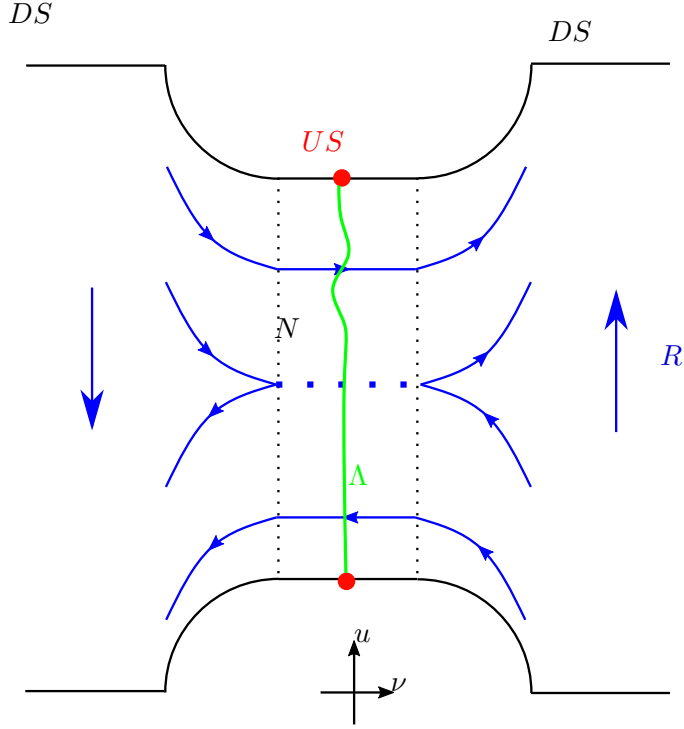
\caption{The Reeb vector field induced by the contact form $\tilde \l$.}
\label{fig.Reeb_2-tresse}
\end{figure}

For any contact form $\tilde \l$, Gray's theorem yields a contactomorphism with $(J^1(\RR \times S^1), \xi_\st)$. However with the present construction, we can even construct an exact contactomorphism, so that the situation boils down to the 1-jets space $(J^1(\RR \times S^1), \l_\st)$. 
\begin{rmk}
The contact form $\tilde \l$ presents degenerates orbits, however Reeb chords are indeed non-degenerates.
\end{rmk}

Define a neighbourhood 
\begin{align*}
    N & = \{|\nu| \leq \e\} \\
    & \simeq [-U, U]_u \times [-\e, \e]_\nu  \times US ,
\end{align*}
on which $\tilde \l = - \nu F(u)du + G(u) \l_\Ga$. For $U$ big enough, $N$ contains the Legendrian ; we will show that we can restrict to this neighbourhood to compute the Legendrian homology.

\subsubsection{Completion}

Here we define coordinates on a neighbourhood of $\dd N$,  such that the contact form can be written as the contactisation of a symplectisation. 
We set 
\begin{align*}
   &  (-\e, 0]_\t \times [-\e, \e]_t \times US & \lra &\  N &   \\
 \psi_1 :\ & \qquad \qquad (t, \t, y, \zeta) & \mt\  &\  (u = U + \t, \nu = - \e_0 t, (z, \eta) = \phi^{\mu(\t, t)}_{US} (y, \zeta)), &   
\end{align*}
where $\phi^t_{US}$ is the flow of $R_{US}$. Then by taking $\mu(\t, t) = te^{-\e_0  \t}$, we get $\psi_1^*\tilde \l = dt + e^{\e_0 \t}\l_\Ga$. 

Indeed on a neighbourhood of $\dd N$ we have $\tilde \l =-\nu du + e^{\e_0 (u-U)} \l_\Ga$, and so we compute :
\begin{align*}
    \psi_1^*\tilde \l :\ & \dd_t \over{d\psi_1}{\mt} - \e_0 \dd_\nu + e^{-\e_0 \t} R_{US} \over{\tilde\l}{\mt} 1\\
    & \dd_\t \mt \dd_u -\e_0 te^{-\e_0 \t} R_{US} \mt \e_0t - \e_0 t = 0\\
     & \dd_{y} \mt (\phi_{US}^\mu)_* \dd_{y} \mt e^{\e_0 \t} (\phi_{US}^\mu)_* (\zeta \cdot \dd_y)  = e^{\e_0 \t} \zeta \cdot \dd_y\\
     & \dd_{\zeta} \mt (\phi_{US}^\mu)_* \dd_{\zeta} \mt 0
\end{align*}

There exists a similar map on the other end of $N$, consequently the completion of $N$ is 
\begin{align*}
    &  N^* = \RR \times (-\e_\nu, \e_\nu) \times US \\
    &  \l^* = - \nu F(u)du + G(u) \l_\Ga
\end{align*}
where $G = e^{\e_0 |u-U|}$ and $F= 1$ for $|u| \geq U-\e$. 

Moreover, the boundary ot the conormal of a strand ending at $z_0 \in S$ is the unit fiber $\{u=U, \nu=0, z=z_0\}$, whose preimage by $\psi_1$ is $\{t=0, \t=0, y = z_0\}$.
Hence the Legendrian completed by wrapping with an Hamiltonian $H$ is, in $N^*$ :
\[\La^H = \La \cup \big\{ \big(u, \nu = \e_0 f_H(|u-U|), z, \eta\big),\ |u|\geq U, (z, \eta) \in \phi_{US}^{\tilde g_H(|u-U|)} U_{z_0} S \big\},  \]
where $f_H$ and $\tilde g_H$ are given by (see \cref{ssec.Completions}) : 
\begin{align*}
    f_H(x) & = \int_0^x e^{\e_0 x} H''dx \\
    \tilde g_H(x)&  = H'(x) + \mu(x, -f_H(x)) = H' - f_H(x)  e^{-\e_0 x}  \\
    & =  H' - e^{-\e_0 x} \int e^{\e_0 x} H''.
\end{align*}

\subsubsection{Lifting to the plane}

We now lift the strands to paths $\hat f_i : [-U, U] \ra \RR^2$. The projection $\RR^2 \ra S$ induces a contactomorphism $ \Pi :  (U \RR^2, \ker \l_\eu) \ra (U\TT^2, \ker \l_{US})$, 
where $\l_\eu$ is the form induced by the flat metric $\RR^2$ (note that the forms induce by the flat and hyperbolic  metrics  on $\RR^2$ are proportionals).
We now lift the neighbourhood. Define
\begin{align*}
    & \hat N = [-U, U]_u \times [-\e, \e]_\nu \times U\RR^2_{(\hat z, \hat \eta)}  \\
    & \hat \l = - \nu F(u)du + G(u) \Pi^* \l_\Ga.
\end{align*}

\paragraph{Torus case.} 
We now assume that  $S = \TT^2$, hence $\Pi^* \l_{US} = \l_\eu$. A modification of the standard contactomorphism $(U\RR^2, \l_\eu) \simeq (J^1(S^1), \l_\st)$ make the situation boil down to  $J^1(\RR \times S^1)$ : define 
\begin{align*}
& J^1(\RR_a \times S_q^1) \ra \hat N \\
\psi_2 :\ & (s; a, q; \a, p) \mt \big(u=a, \hat z = \frac{p+sq}{G(a)}, \nu(a, \a); \hat \eta = q \big) 
\end{align*}
where $q \in S^1 \subset \RR^2$, $\a \in \RR$, and $p\in \RR^2$ is such that $p \cdot q = 0$.

\begin{lem}
By taking $\nu = \frac{G\a + G'}{GF}$, we get $\psi_2^* \hat \l  = ds - \a da - pdq $. Moreover the conormal of a strand $f : [-U, U] \ra S$ becomes the 1-jet of the function $h_f : \RR \times S^1 \ra \RR $, defined by
\begin{itemize}
    \item if $|u| \leq U, h_f = G(a) \hat f(a) \cdot q$ ;
    \item if $|a|\geq U, h_f  = e^{\e_0 |a - U|}  z_0 \cdot q  +\e_0 \int_0^{|a-U|} e^{\e_0 x} H' dx$.
\end{itemize}
where $\hat f$ is a lift $[-U, U] \ra \RR^2$.
\end{lem}

\begin{proof}
In coordinates we have $\hat \l = - F(u) \nu du + G(u) \l_\eu$, and we compute
\begin{align*}
    \psi_2^* \hat \l : \ & \dd_s \over{d\psi_2}{\mt}\frac{1}{G}q\dd_z \over{\hat \l}{\mt} q.q = 1\\
    & \dd_a \mt \dd_u + * \dd_\nu - \frac{G'}{G^2}(p+sq)\dd_z \mt - \frac{G\a + G'}{GF}F -  \frac{G'}{G^2} G q \cdot(p+sq) \\
     & \quad \quad \quad = \a + \frac{G'}{G} - \frac{G'}{G} = \a \\
    & \dd_{q_i} \mt \frac sG \dd_{z_i} + \dd_{\eta_i} \mt s q \cdot \dd_{z_i} = 0
     \quad \text{ because } \dd_{q_i}\text{ is orthogonal to } q\\
    & \dd_\a \mt * \dd_\nu \mt 0\\
     & \dd_p \mt \frac{1}{G} p\dd_z + * \dd_\eta  \mt \frac 1G p \cdot q = 0.
\end{align*}

Now let $\La^H \subset N^*$ be the conormal of a strand $f : [-U, U] \ra S$, that we lift to $\hat \La^* \subset \hat N$. It is not necessary to explicit all equations to express $\psi_2^*(\hat \La)$, because the coordinate $s$ determine the Legendrian :
\begin{itemize}
    \item if $|a|\leq U$, we have $p+sq = G(a) \hat f(a)$, hence $s = (p+sq)\cdot q = G(a) \hat f(a) \cdot q$ ;
    \item if $|a| \geq U$, we have \[\phi_{U\RR^2}^x (U_{\hat z_0} \RR^2) = \{(\hat z, \hat \eta) |\ |\hat z - \hat z_0| = x, \hat \eta = \frac{\hat z- \hat z_0}{|\hat z - \hat z_0|} \},\]
    hence when  $(z, \eta) = (\frac{p+sq}{G}, q)$, the condition $(\hat z, \hat \eta) \in \phi_{US}^{\tilde g_H} (U_{\hat z_0} \RR^2)$ implies
    \begin{align*}
        \tilde g_H & =  |z - z_0| = (z- z_0) \cdot \frac{z-z_0}{|z - z_0|} \\
         & = (\frac{p+sq}G - z_0)\cdot q = \frac s G - z_0 \cdot q.
    \end{align*}
    On this region $G = e^{\e_0 (|a-U|}$, so 
    \begin{align*}
        s & =    G(\tilde g_H + z_0 \cdot q) \\
        & =  e^{\e_0 |a - U|} \big( H'(|a-U|) + z_0 \cdot q \big) - f_H(|a-U|).
\end{align*}
Moreover $e^{\e_0 x} H'- \int_0^x e^{\e_0 x} H'' = \e_0 \int e^{\e_0 x} H'$.
\end{itemize}
\end{proof}

\paragraph{Higher genus.} When $g(S) > 1$, ie if the surface is not a torus anymore, lifting the contact form (induced by a metric with constant curvature) from $U_g S$ to $U\RR^2 \simeq J^1(S^1)$ yields a contact form which differs of the standard contact form by a conformal factor. 
To use the results of the euclidean case. we interpolate toward the standard form. In each degree, there is a chord by homotopy class, hence the map induced by a cobordism is a trivial isomorphism trivial (at the chain level). Once again we get two Legendrians : the zero section and the 1-jet of a function $h = G(a) \hat f(a) \cdot q + \tilde H (a)$.

\begin{rmk}
If $S$ is not a torus, we can still use this description by taking two strands very close (and by fixing the perturbation making the manifold sutured).
\end{rmk}

\subsection{Proof of the theorem} \label{ssec.proof}
We compute the exact triangle associated to the local braid $B_{t^k}$. Pick a point $z_0 \in S$, and choose the braid given by 
\begin{align*}
    & f_1 : u \mt z_0 + \e_1 e^{i\eta_k(u)}\\
    & f_2 : u \mt z_0 
\end{align*}
where $\eta_k : [-U, U] \ra \RR$ is
\begin{itemize}
    \item increasing (resp.  decreasing) if $k \geq 0$ (resp. $k \leq 0$) ;
    \item vanishing on $[-U, \frac {3U} 5]$ ;
    \item evaluating to $2\pi k$ on $[\frac{4U}5, U]$.
\end{itemize}
Complete the Legendrian $\La_1 \cup \La_2 = U_{B_1} (I \times S) \cup U_{B_2}(I \times S)$, by wrapping $\La_1$  using Hamiltonian $H$, so it is included in $N^*$. Pick a lift to $\RR^2$ such that $0$ projects to $z_0$, and $(\e_1, 0)$ projects to $z_0 + \e_1$, the end of $B_1$. Then $\psi_2^*(\hat \La) \subset J^1(\RR \times S^1)$ is the 1-jet of the functions 
\begin{align*}
     h_1 (a, q) & = G(a) \e_1 \cos (\th + \eta_k(a))  & \text{if } |u| \leq U \\
        & \quad \ \ \e_1 e^{\e_0 |a-U|} \cos \th + \e_0 \int_0^{|a-U|} e^{\e_0 x} H'  &  \text{if } |u| \geq U\\
     h_2 (a, q) & = 0   
\end{align*}
where $q = e^{i\th} \in \RR^2$.

Since the homolomorphic strips in $J^1(S^1 \times \RR)$, with boundary in $\La_1 \cap \La_2$, correspond to Morse trajectories of $h = h_2-h_1$, we define $\tilde H = \e_0 \int_0^x e^{\e_0 x} H'$, and we compute
\begin{align*}
    dh_1 & = \e_1 \big( \cos (\th + \eta_k) G' - \sin(\th + \eta_k)\eta' \big)da - \e_1 \sin(\th + \eta_k) d\th &  \text{ if } |a| \leq U & \\
     & \quad \ \ \e_0\big(\e_1 e^{\e_0|a-U|} \cos \th + \tilde H'(|a-U|)\big)da - \e_1 e^{|a-U|} \sin \th d\th  &  \text{ if } |a| \geq U. &
\end{align*}
This differential vanishes when $\sin (\th + \eta_k)=0$ and $G'=0$, if $|a| \leq U$, or when $\sin \th = 0 $ and $\e_1 e^{\e_0 |a-U|} + \tilde H'(|a-U|)= 0$. The function $h$ presents three positive critical points, which we denote $c_0 \in \{|a| \leq U\}$ and $c_\pm \in \{\pm a \geq U\}$, and there exists a unique rigid Morse trajectory of Morse from $c_\pm $ à $c_0$, in the neighbourhood of the path parametrized by $\th = \pi -\eta_k (a)$, see 
\cref{fig.Maple.Gradient_2tresse_1jet}.

\begin{figure}[h!]
\center
\def\svgwidth{14cm}

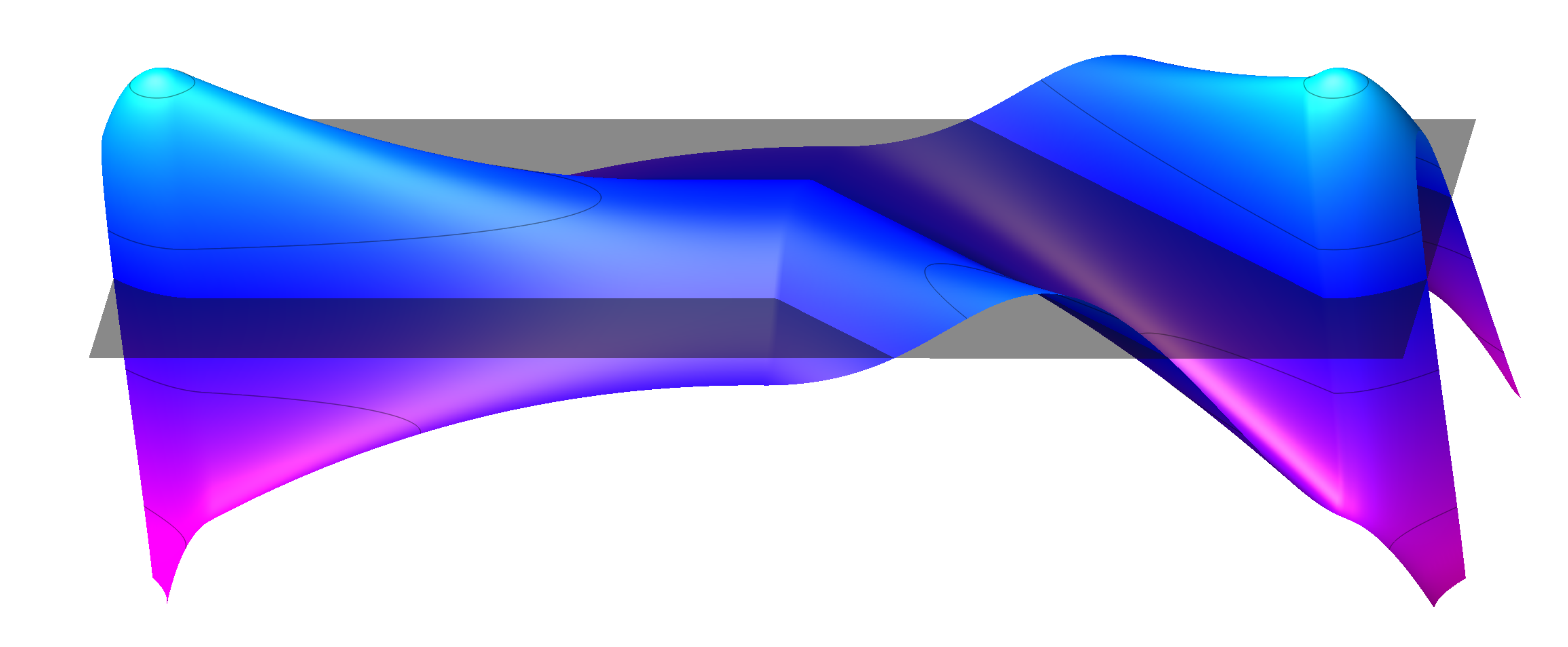
\caption{The graph of the function $h$ (with the zero section in gray) for the braid with one twist. In red, its positive critical points and its gradient trajectories.}
\label{fig.Maple.Gradient_2tresse_1jet}
\end{figure}

\begin{lem}
For $\e_0$ small enough and $G$ close enough of $1$, we can count curves in $(N^*, \l^*)$
\end{lem}

\begin{proof}
Indeed the difference of actions between $c_\pm$ and $c_0$ converge toward $0$ when $\e_0$ goes to $0$ and $G$ goes to $1$. The monotonicity result \cref{prop.Monotonicite_global} implies that a curve a small energy stays in $N^*$.
Note that their might also be other long chords, because of the topology of the surface, however their action can be made as big as necessary so we can ignore them by a standard action-filtration argument.
\end{proof}

Hence the wrapped Legendrian complex is $\W LC(\La_1, \La_2) = C[1] \oplus C  \oplus C[1]$, and the differential is 
    \[\dd c_\g^0 = 0 \qquad \qquad  \dd c_\g^- = c_\g^0 \qquad \qquad    \dd c_\g^+ = \mu_x^{-k} c_\g^0 \mu_y^k \]
We now prove the theorem. If $\La_k$ and $\La_{k'}$ are isotopic (as Legendrians), we can glue them together to obtain an  isotopy between $\La_0$ and $\La_{k-k'}$. Thus we can assume, without loss of generality, that $k'=0$. The invariance of the exact triangle implies the commutativity of the following diagram
\[\begin{xymatrix}{
\ar[r] &  C \ar[r]^0 \ar[d]^{F_*} & C[1] \ar[r]^-{f_0} \ar[d]^{F_*^\W} & C[1] \oplus C[1] \ar[r]^-{\Id \oplus \d_0} \ar[d]^{\Id} &  C \ar[r]^0 \ar[d]^{F_*} &  \\
\ar[r] & C \ar[r]^0  & C[1] \ar[r]^-{f_k} & C[1] \oplus C[1] \ar[r]^-{\Id \oplus \d_k}  & C \ar[r]^0 & }
\end{xymatrix}\]
where the maps $f_0, f_k, \d_0$ and $\d_k$ are morphisms of $\ZZ[\mu_x]-\ZZ[\mu_y]$-bimodules, and
\[     f_0(c_\g) = (c_\g, c_\g), \qquad \d_0 = \Id, \qquad \d_k(c_\g) = \mu_x^{-k} c_\g \mu_y^k.\]
The kernel of $\Id \oplus \d_0$ is the anti-diagonal, generated by the elements $(c_\g , - c_\g)$, which implies that $k=0$. 

\subsection{Sketch in the general case}\label{ssec.Sketch_general_case}

We quickly present the strategy to handle the general case, which will be treated in a forthcoming paper. Given a pure braid $B$ with $n$ strands in $\Si$, our goal is to construct an automorphism of $\pi_1(\Si \sms \{x_1... x_n\})$ (up to inner automorphism). This is a complete invariants by Dehn–Nielsen–Baer theorem, see \cite[Theorem 8.1]{FarbMarg11_Primer_mapping_class_group}.

The main idea was mentioned in \cref{ssec.ExVarSut} : given two Legendrians $\La$ and $\La_0$ in a closed, compact contact manifold $(V, \xi)$, the Legendrian contact homology of $\La_0$ in \[V_
\La = V \sms \N(\La)\] is an invariant of the (ordered) pair of Legendrians. From another point of view, it is the contact homology of $\La_0$ in $(V, \xi)$ {\it stopped} at $\La$.

This construction should extend to the sutured setting : consider a sutured Legendrian $\La \subset (V, \l, \Ga)$. After removing a neighbourhood, we get a "2-sutured" contact manifold $V_\La$, with two kinds of boundary. We can complete it so that, at infinity, it looks like the contactisation of some {\it Liouville  sector} (for the appropriate notion of sector). Moreover, the exact sequence from \cref{ssec.Exact_seq} should be modified as follows. If $\La_0, \La_1$ are two other sutured Legendrians in $(V, \l, \Ga)$, we should get an inclusion as before
\[ LC(\La_0, \La_1; V_\La) \hra WLC(\La_0, \La_1; V_\La), \]
whose quotient is the exterior complex. For simple situations (notably when we don't need to augment the chords), the exterior homology should be isomorphic to the stopped homology of the boundary :
\[ LC^\ext(\La_0, \La_1; V_\La) \simeq LC(\dd \La_0, \dd \La_1; \Ga_{\dd \La}),  \]
and the exact sequence is an invariant (in the sense of \cref{ssec.Exact_seq}) under isotopy of $(\La_0, \La_1, \La)$ fixed at the boundary (note that the construction from \cref{ssec.dga_wrapped} might be more adapted).

We now take for sutured contact manifold the unit bundle 
$V = U(I \times \Si)$,
and for sutured Legendrian stop the conormal $\La = \La_B$, which is a collection of $n$ cylinders. Finally $\La_0$ and $\La_1$ are the conormals of two push-off of the first strand.

The boundary of $\La_B$ are unit fibers in the suture $\Ga = U\Si \sqcup U\Si$, and by definition the complement of the unit fiber over $x \in  \Si$ is $U(\Si \sms \{x\})$. Thus, the exterior homology is the relative Legendrian homology of two unit fibers in $U(\Si \sms \{x_1...x_n \})$, in other words 
\[ LH^\ext (\La_0, \La_1; V_\La) \simeq \ZZ[\pi_1(\Si \sms \{x_1...x_n\})] \oplus \ZZ[\pi_1(\Si \sms \{x_1...x_n\})].  \]
Moreover if we add the product structure obtained by counting pair of pants (notice that we can identify the chords of $\La_i$ with the mixed chords from $\La_0$ to $\La_1$), we can upgrade this to a ring isomorphism.

On the other hand, using techniques from \cite{Dat22_Conormal_stops_Hyperbolic_knots}, we should prove that, to compute Legendrian homology, we can replace $V_\La = U(I \times \Si) \sms \La_B$ by $U(I \times \Si \sms B)$. Hence, using the homotopical restrictions on the differential, we get the exact sequence
\[ \ra  C_0 \ra C \ra C_-[1] \oplus C_+[1] \ra, \]
where $C = C_\pm = C_0 = \ZZ[\pi_1(\Si \sms \{x_1...x_n\})]$, and the boundary map will be induced by the classical map raised by the braid $f_B \in \Aut (\pi_1(\Si \sms \{x_1...x_n\}) )$. As previously, we can recover this map using the invariance of the exact triangle.

\bibliographystyle{alpha}
\bibliography{references}

\end{document}

%% file: packages.tex
\usepackage{import}
\usepackage[utf8]{inputenc}
\usepackage[T1]{fontenc}
\usepackage{amsmath}
\usepackage{amsthm}
\usepackage{amsfonts}
\usepackage{amssymb}
\usepackage[vmargin=3.0cm]{geometry}
\usepackage[all]{xy}
\usepackage{enumitem}
\usepackage{color}
\usepackage{setspace}
\usepackage[hidelinks]{hyperref}
\usepackage{graphicx}
\usepackage[noabbrev]{cleveref}
\usepackage{comment}



%% file: new_commands.tex
\newcommand{\ZZ}{\mathbb{Z}}

\newcommand{\RR}{\mathbb{R}}
\newcommand{\CC}{\mathbb{C}}

\newcommand{\NN}{\mathbb{N}}
\newcommand{\TT}{\mathbb{T}}

\newcommand{\A}{\mathcal{A}}
\newcommand{\B}{\mathcal{B}}
\newcommand{\C}{\mathcal{C}}

\renewcommand{\L}{\mathcal{L}} 
\newcommand{\LC}{\mathcal{L}C}
\newcommand{\LH}{\mathcal{L}H}
\newcommand{\M}{\mathcal{M}}
\newcommand{\Mbar}{\overline{\mathcal{M}}}

\newcommand{\N}{\mathcal{N}}
\renewcommand{\P}{\mathcal{P}} 
\newcommand{\R}{\mathcal{R}} 

\newcommand{\T}{\mathcal{T}} 
\newcommand{\W}{\mathcal{W}}

\newcommand{\pp}{\mathbf p}

\newcommand{\ra}{\rightarrow}
\newcommand{\hra}{\hookrightarrow}

\newcommand{\lra}{\longrightarrow}
\newcommand{\Ra}{\Rightarrow}
\newcommand{\mt}{\mapsto}

\renewcommand{\a}{\alpha}
\renewcommand{\b}{\beta}
\renewcommand{\d}{\delta} 
\newcommand{\e}{\varepsilon} 
\renewcommand{\i}{\iota}

\newcommand{\g}{\gamma}
\newcommand{\Ga}{\Gamma}
\renewcommand{\l}{\lambda}
\newcommand{\La}{\Lambda}

\renewcommand{\sp}{\mathfrak{s}}
\newcommand{\s}{\sigma}
\newcommand{\Si}{\Sigma}
\renewcommand{\t}{\tau}
\renewcommand{\th}{\theta}
\newcommand{\w}{\omega}

\newcommand{\vphi}{\varphi}

\newcommand{\Area}{\text{Area}}
\newcommand{\Aut}{\text{Aut}}

\newcommand{\Core}{\text{Core}}
\newcommand{\Cont}{\text{Cont}}
\newcommand{\Crit}{\text{Crit}}

\newcommand{\eu}{\text{eu}}

\newcommand{\ext}{\text{ext}}

\newcommand{\im}{\text{Im }}
\newcommand{\Id}{\text{Id}}

\newcommand{\st}{\text{st}}
\newcommand{\Symp}{\text{Symp}}

\renewcommand{\top}{\text{top}}

\newcommand{\colim}{\text{colim }}

\newcommand{\Vect}{\text{Vect}}
\newcommand{\dd}{\partial}
\newcommand{\ity}{\infty} 

\newcommand{\rest}{\restriction}

\renewcommand{\over}{\overset}
\newcommand{\under}[2]{\underset{#2}{#1}}
\newcommand{\sms}{\setminus}
\newcommand{\mr}{\mathring}

\newcommand{\sm}{\footnotesize}

\newcommand{\<}{\left<}
\renewcommand{\>}{\right>}

\setlist{noitemsep,topsep=0pt,parsep=0pt,partopsep=0pt}

\newtheorem{thm}{Theorem}[section]
    \newtheorem{lem}[thm]{Lemma}
    \newtheorem{prop}[thm]{Proposition}
    \newtheorem{cor}[thm]{Corollary}
    \newtheorem{conj}[thm]{Conjecture}

\newtheorem{MainThm}{Theorem}

\newtheorem*{thm*}{Theorem}
\newtheorem{MainConj}[MainThm]{Conjecture}

\theoremstyle{remark}
    \newtheorem{rmk}[thm]{Remark}
    \newtheorem{rmks}[thm]{Remark}
    \newtheorem{ex}[thm]{Example}
    \newtheorem{exs}[thm]{Examples}

\theoremstyle{definition}
    \newtheorem{dfn}[thm]{Definition}

%% file: Dessins/Leg_sutured.pdf_tex
\begingroup%
  \makeatletter%
  \providecommand\color[2][]{%
    \errmessage{(Inkscape) Color is used for the text in Inkscape, but the package 'color.sty' is not loaded}%
    \renewcommand\color[2][]{}%
  }%
  \providecommand\transparent[1]{%
    \errmessage{(Inkscape) Transparency is used (non-zero) for the text in Inkscape, but the package 'transparent.sty' is not loaded}%
    \renewcommand\transparent[1]{}%
  }%
  \providecommand\rotatebox[2]{#2}%
  \newcommand*\fsize{\dimexpr\f@size pt\relax}%
  \newcommand*\lineheight[1]{\fontsize{\fsize}{#1\fsize}\selectfont}%
  \ifx\svgwidth\undefined%
    \setlength{\unitlength}{211.5053288bp}%
    \ifx\svgscale\undefined%
      \relax%
    \else%
      \setlength{\unitlength}{\unitlength * \real{\svgscale}}%
    \fi%
  \else%
    \setlength{\unitlength}{\svgwidth}%
  \fi%
  \global\let\svgwidth\undefined%
  \global\let\svgscale\undefined%
  \makeatother%
  \begin{picture}(1,0.93577267)%
    \lineheight{1}%
    \setlength\tabcolsep{0pt}%
    \put(0,0){\includegraphics[width=\unitlength,page=1]{Leg_sutured.pdf}}%
    \put(0.03027706,0.39943615){\color[rgb]{0,1,0}\makebox(0,0)[lt]{\lineheight{1.25}\smash{\begin{tabular}[t]{l}$\La$\end{tabular}}}}%
    \put(0,0){\includegraphics[width=\unitlength,page=2]{Leg_sutured.pdf}}%
    \put(0.87937809,0.52689687){\color[rgb]{0,0,0}\makebox(0,0)[lt]{\lineheight{1.25}\smash{\begin{tabular}[t]{l}$\t$\end{tabular}}}}%
    \put(0.65940399,0.72298673){\color[rgb]{0,0,0}\makebox(0,0)[lt]{\lineheight{1.25}\smash{\begin{tabular}[t]{l}$t$\end{tabular}}}}%
    \put(0.59163259,0.41013207){\color[rgb]{1,0,0}\makebox(0,0)[lt]{\lineheight{1.25}\smash{\begin{tabular}[t]{l}$\Ga$\end{tabular}}}}%
    \put(0.0281709,0.58609376){\color[rgb]{0,0,0}\makebox(0,0)[lt]{\lineheight{1.25}\smash{\begin{tabular}[t]{l}$V$\end{tabular}}}}%
    \put(0,0){\includegraphics[width=\unitlength,page=3]{Leg_sutured.pdf}}%
    \put(0.07334914,0.90645983){\color[rgb]{0,0,0}\makebox(0,0)[lt]{\lineheight{1.25}\smash{\begin{tabular}[t]{l}$(R_+, \b_+)$\end{tabular}}}}%
    \put(0.07443047,0.03932691){\color[rgb]{0,0,0}\makebox(0,0)[lt]{\lineheight{1.25}\smash{\begin{tabular}[t]{l}$(R_-, \b_-)$\end{tabular}}}}%
    \put(0,0){\includegraphics[width=\unitlength,page=4]{Leg_sutured.pdf}}%
    \put(0.23060504,0.34868055){\color[rgb]{0,0,0}\makebox(0,0)[lt]{\lineheight{1.25}\smash{\begin{tabular}[t]{l}$dt + e^\t \l_\Ga$\end{tabular}}}}%
  \end{picture}%
\endgroup%

%% file: Dessins/Var_cvx_cyl.pdf_tex
\begingroup%
  \makeatletter%
  \providecommand\color[2][]{%
    \errmessage{(Inkscape) Color is used for the text in Inkscape, but the package 'color.sty' is not loaded}%
    \renewcommand\color[2][]{}%
  }%
  \providecommand\transparent[1]{%
    \errmessage{(Inkscape) Transparency is used (non-zero) for the text in Inkscape, but the package 'transparent.sty' is not loaded}%
    \renewcommand\transparent[1]{}%
  }%
  \providecommand\rotatebox[2]{#2}%
  \newcommand*\fsize{\dimexpr\f@size pt\relax}%
  \newcommand*\lineheight[1]{\fontsize{\fsize}{#1\fsize}\selectfont}%
  \ifx\svgwidth\undefined%
    \setlength{\unitlength}{396.74390417bp}%
    \ifx\svgscale\undefined%
      \relax%
    \else%
      \setlength{\unitlength}{\unitlength * \real{\svgscale}}%
    \fi%
  \else%
    \setlength{\unitlength}{\svgwidth}%
  \fi%
  \global\let\svgwidth\undefined%
  \global\let\svgscale\undefined%
  \makeatother%
  \begin{picture}(1,0.47406111)%
    \lineheight{1}%
    \setlength\tabcolsep{0pt}%
    \put(0,0){\includegraphics[width=\unitlength,page=1]{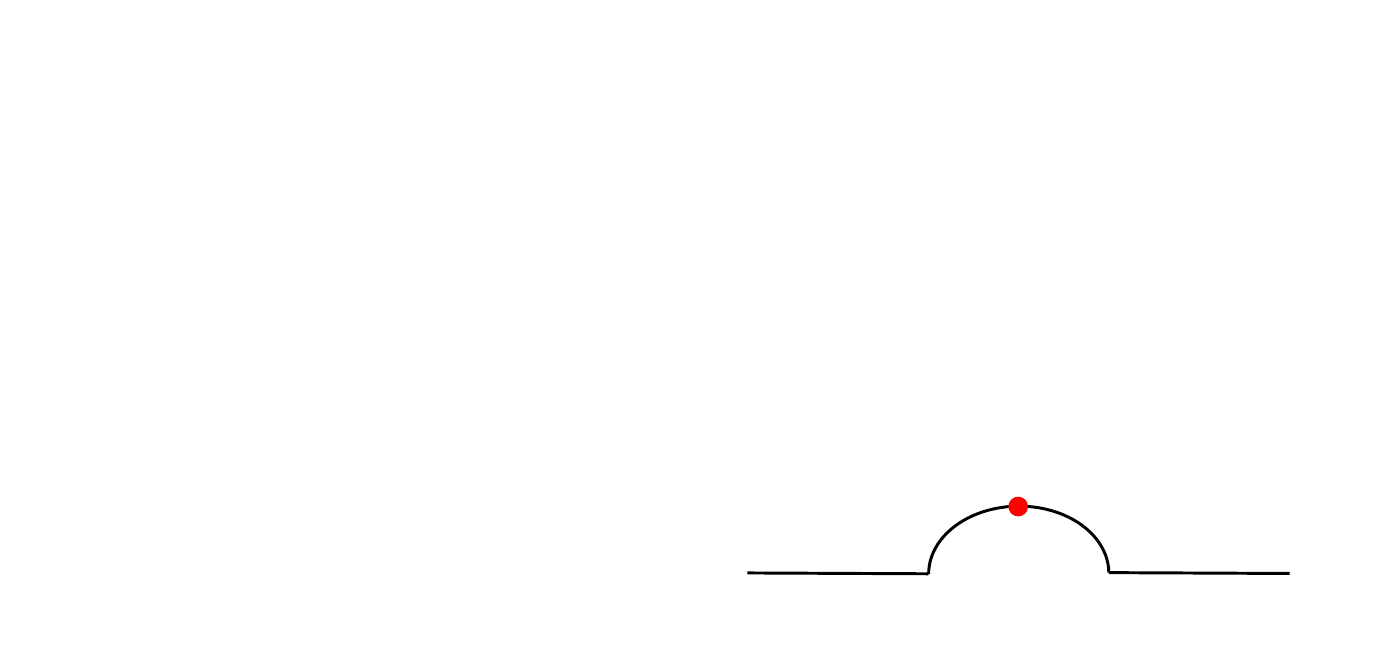}}%
    \put(0.72582715,0.39139749){\color[rgb]{1,0,0}\makebox(0,0)[lt]{\lineheight{1.25}\smash{\begin{tabular}[t]{l}$\Ga$\end{tabular}}}}%
    \put(0.56808502,0.43627733){\color[rgb]{0,0,0}\makebox(0,0)[lt]{\lineheight{1.25}\smash{\begin{tabular}[t]{l}$R_-$\end{tabular}}}}%
    \put(0.85376807,0.43116131){\color[rgb]{0,0,0}\makebox(0,0)[lt]{\lineheight{1.25}\smash{\begin{tabular}[t]{l}$R_+$\end{tabular}}}}%
    \put(0,0){\includegraphics[width=\unitlength,page=2]{Var_cvx_cyl.pdf}}%
    \put(0.56838374,0.01686201){\color[rgb]{0,0,0}\makebox(0,0)[lt]{\lineheight{1.25}\smash{\begin{tabular}[t]{l}$R_+$\end{tabular}}}}%
    \put(0.83954426,0.01069989){\color[rgb]{0,0,0}\makebox(0,0)[lt]{\lineheight{1.25}\smash{\begin{tabular}[t]{l}$R_-$\end{tabular}}}}%
    \put(0.61565311,0.2336379){\color[rgb]{0,0,1}\makebox(0,0)[lt]{\lineheight{1.25}\smash{\begin{tabular}[t]{l}$R$\end{tabular}}}}%
    \put(0.17958432,0.43959106){\color[rgb]{1,0,0}\makebox(0,0)[lt]{\lineheight{1.25}\smash{\begin{tabular}[t]{l}$\Ga$\end{tabular}}}}%
    \put(0.04032848,0.43209281){\color[rgb]{0,0,0}\makebox(0,0)[lt]{\lineheight{1.25}\smash{\begin{tabular}[t]{l}$R_-$\end{tabular}}}}%
    \put(0.29520089,0.43467946){\color[rgb]{0,0,0}\makebox(0,0)[lt]{\lineheight{1.25}\smash{\begin{tabular}[t]{l}$R_+$\end{tabular}}}}%
    \put(0,0){\includegraphics[width=\unitlength,page=3]{Var_cvx_cyl.pdf}}%
    \put(0.22992546,0.26095849){\color[rgb]{0,0,0}\makebox(0,0)[lt]{\lineheight{1.25}\smash{\begin{tabular}[t]{l}$r$\end{tabular}}}}%
    \put(0.20695569,0.3039995){\color[rgb]{0,0,0}\makebox(0,0)[lt]{\lineheight{1.25}\smash{\begin{tabular}[t]{l}$u$\end{tabular}}}}%
    \put(0.04987029,0.01344776){\color[rgb]{0,0,0}\makebox(0,0)[lt]{\lineheight{1.25}\smash{\begin{tabular}[t]{l}$R_+$\end{tabular}}}}%
    \put(0.28482845,0.01113706){\color[rgb]{0,0,0}\makebox(0,0)[lt]{\lineheight{1.25}\smash{\begin{tabular}[t]{l}$R_-$\end{tabular}}}}%
    \put(-0.00376599,0.22714259){\color[rgb]{0,0,1}\makebox(0,0)[lt]{\lineheight{1.25}\smash{\begin{tabular}[t]{l}$R$\end{tabular}}}}%
    \put(0,0){\includegraphics[width=\unitlength,page=4]{Var_cvx_cyl.pdf}}%
  \end{picture}%
\endgroup%

%% file: Dessins/Completion_sutured.pdf_tex
\begingroup%
  \makeatletter%
  \providecommand\color[2][]{%
    \errmessage{(Inkscape) Color is used for the text in Inkscape, but the package 'color.sty' is not loaded}%
    \renewcommand\color[2][]{}%
  }%
  \providecommand\transparent[1]{%
    \errmessage{(Inkscape) Transparency is used (non-zero) for the text in Inkscape, but the package 'transparent.sty' is not loaded}%
    \renewcommand\transparent[1]{}%
  }%
  \providecommand\rotatebox[2]{#2}%
  \newcommand*\fsize{\dimexpr\f@size pt\relax}%
  \newcommand*\lineheight[1]{\fontsize{\fsize}{#1\fsize}\selectfont}%
  \ifx\svgwidth\undefined%
    \setlength{\unitlength}{211.5053288bp}%
    \ifx\svgscale\undefined%
      \relax%
    \else%
      \setlength{\unitlength}{\unitlength * \real{\svgscale}}%
    \fi%
  \else%
    \setlength{\unitlength}{\svgwidth}%
  \fi%
  \global\let\svgwidth\undefined%
  \global\let\svgscale\undefined%
  \makeatother%
  \begin{picture}(1,0.93577267)%
    \lineheight{1}%
    \setlength\tabcolsep{0pt}%
    \put(0,0){\includegraphics[width=\unitlength,page=1]{Completion_sutured.pdf}}%
    \put(0.03027706,0.39943615){\color[rgb]{0,1,0}\makebox(0,0)[lt]{\lineheight{1.25}\smash{\begin{tabular}[t]{l}$\La$\end{tabular}}}}%
    \put(0,0){\includegraphics[width=\unitlength,page=2]{Completion_sutured.pdf}}%
    \put(0.87937809,0.52689687){\color[rgb]{0,0,0}\makebox(0,0)[lt]{\lineheight{1.25}\smash{\begin{tabular}[t]{l}$\t$\end{tabular}}}}%
    \put(0.65940399,0.72298673){\color[rgb]{0,0,0}\makebox(0,0)[lt]{\lineheight{1.25}\smash{\begin{tabular}[t]{l}$t$\end{tabular}}}}%
    \put(0.59163259,0.41013207){\color[rgb]{1,0,0}\makebox(0,0)[lt]{\lineheight{1.25}\smash{\begin{tabular}[t]{l}$\Ga$\end{tabular}}}}%
    \put(0.06493591,0.51905167){\color[rgb]{0,0,0}\makebox(0,0)[lt]{\lineheight{1.25}\smash{\begin{tabular}[t]{l}$V$\end{tabular}}}}%
    \put(0,0){\includegraphics[width=\unitlength,page=3]{Completion_sutured.pdf}}%
    \put(0.07334914,0.90645983){\color[rgb]{0,0,0}\makebox(0,0)[lt]{\lineheight{1.25}\smash{\begin{tabular}[t]{l}$(\RR_t^+ \times R_+, dt + \b_+)$\end{tabular}}}}%
    \put(0.07334914,0.00688717){\color[rgb]{0,0,0}\makebox(0,0)[lt]{\lineheight{1.25}\smash{\begin{tabular}[t]{l}$(\RR_t^- \times R_-, dt + \b_-)$\end{tabular}}}}%
    \put(0.67791786,0.28926055){\color[rgb]{0,0,0}\makebox(0,0)[lt]{\lineheight{1.25}\smash{\begin{tabular}[t]{l}$(\RR_t \times \RR_\t^+ \times \Ga, dt + e^\t \l_\Ga)$\end{tabular}}}}%
    \put(0,0){\includegraphics[width=\unitlength,page=4]{Completion_sutured.pdf}}%
    \put(0.23060504,0.34868055){\color[rgb]{0,0,0}\makebox(0,0)[lt]{\lineheight{1.25}\smash{\begin{tabular}[t]{l}$dt + e^\t \l_\Ga$\end{tabular}}}}%
  \end{picture}%
\endgroup%

%% file: Dessins/Completion_legendr.pdf_tex
\begingroup%
  \makeatletter%
  \providecommand\color[2][]{%
    \errmessage{(Inkscape) Color is used for the text in Inkscape, but the package 'color.sty' is not loaded}%
    \renewcommand\color[2][]{}%
  }%
  \providecommand\transparent[1]{%
    \errmessage{(Inkscape) Transparency is used (non-zero) for the text in Inkscape, but the package 'transparent.sty' is not loaded}%
    \renewcommand\transparent[1]{}%
  }%
  \providecommand\rotatebox[2]{#2}%
  \newcommand*\fsize{\dimexpr\f@size pt\relax}%
  \newcommand*\lineheight[1]{\fontsize{\fsize}{#1\fsize}\selectfont}%
  \ifx\svgwidth\undefined%
    \setlength{\unitlength}{416.5212412bp}%
    \ifx\svgscale\undefined%
      \relax%
    \else%
      \setlength{\unitlength}{\unitlength * \real{\svgscale}}%
    \fi%
  \else%
    \setlength{\unitlength}{\svgwidth}%
  \fi%
  \global\let\svgwidth\undefined%
  \global\let\svgscale\undefined%
  \makeatother%
  \begin{picture}(1,0.41460695)%
    \lineheight{1}%
    \setlength\tabcolsep{0pt}%
    \put(0,0){\includegraphics[width=\unitlength,page=1]{Completion_legendr.pdf}}%
    \put(0.5153628,0.2710864){\color[rgb]{0,0,0}\makebox(0,0)[lt]{\lineheight{1.25}\smash{\begin{tabular}[t]{l}$\Ga$\end{tabular}}}}%
    \put(0.04988124,0.24426852){\color[rgb]{0,0,0}\makebox(0,0)[lt]{\lineheight{1.25}\smash{\begin{tabular}[t]{l}$\La$\end{tabular}}}}%
    \put(0.3755268,0.1312504){\color[rgb]{0,0,0}\makebox(0,0)[lt]{\lineheight{1.25}\smash{\begin{tabular}[t]{l}$\dd_0 \La$\end{tabular}}}}%
    \put(0.38702013,0.24235296){\color[rgb]{0,0,0}\makebox(0,0)[lt]{\lineheight{1.25}\smash{\begin{tabular}[t]{l}$\dd_1 \La$\end{tabular}}}}%
  \end{picture}%
\endgroup%

%% file: Dessins/OrBord.pdf_tex
\begingroup%
  \makeatletter%
  \providecommand\color[2][]{%
    \errmessage{(Inkscape) Color is used for the text in Inkscape, but the package 'color.sty' is not loaded}%
    \renewcommand\color[2][]{}%
  }%
  \providecommand\transparent[1]{%
    \errmessage{(Inkscape) Transparency is used (non-zero) for the text in Inkscape, but the package 'transparent.sty' is not loaded}%
    \renewcommand\transparent[1]{}%
  }%
  \providecommand\rotatebox[2]{#2}%
  \newcommand*\fsize{\dimexpr\f@size pt\relax}%
  \newcommand*\lineheight[1]{\fontsize{\fsize}{#1\fsize}\selectfont}%
  \ifx\svgwidth\undefined%
    \setlength{\unitlength}{458.59152486bp}%
    \ifx\svgscale\undefined%
      \relax%
    \else%
      \setlength{\unitlength}{\unitlength * \real{\svgscale}}%
    \fi%
  \else%
    \setlength{\unitlength}{\svgwidth}%
  \fi%
  \global\let\svgwidth\undefined%
  \global\let\svgscale\undefined%
  \makeatother%
  \begin{picture}(1,0.23920462)%
    \lineheight{1}%
    \setlength\tabcolsep{0pt}%
    \put(0,0){\includegraphics[width=\unitlength,page=1]{OrBord.pdf}}%
    \put(0.18288105,0.10501631){\color[rgb]{0,0,0}\makebox(0,0)[lt]{\lineheight{1.25}\smash{\begin{tabular}[t]{l}$\dd_0 S$\end{tabular}}}}%
    \put(-0.00325811,0.10716826){\color[rgb]{0,0,0}\makebox(0,0)[lt]{\lineheight{1.25}\smash{\begin{tabular}[t]{l}$\dd_1 S$\end{tabular}}}}%
    \put(0,0){\includegraphics[width=\unitlength,page=2]{OrBord.pdf}}%
    \put(0.27003291,0.19647197){\color[rgb]{0,0,0}\makebox(0,0)[lt]{\lineheight{1.25}\smash{\begin{tabular}[t]{l}$\pi \circ u$\end{tabular}}}}%
    \put(0,0){\includegraphics[width=\unitlength,page=3]{OrBord.pdf}}%
    \put(0.69072895,0.12223149){\color[rgb]{0,0,0}\makebox(0,0)[lt]{\lineheight{1.25}\smash{\begin{tabular}[t]{l}$(\dd_0 \La)^*$\end{tabular}}}}%
    \put(0.64338717,0.20938335){\color[rgb]{0,0,0}\makebox(0,0)[lt]{\lineheight{1.25}\smash{\begin{tabular}[t]{l}$(\dd_1 \La)^H$\end{tabular}}}}%
    \put(0.61864035,0.10178847){\color[rgb]{0,0,0}\makebox(0,0)[lt]{\lineheight{1.25}\smash{\begin{tabular}[t]{l}$c_-$\end{tabular}}}}%
    \put(0.1043368,0.00925685){\color[rgb]{0,0,0}\makebox(0,0)[lt]{\lineheight{1.25}\smash{\begin{tabular}[t]{l}$c_-$\end{tabular}}}}%
    \put(0,0){\includegraphics[width=\unitlength,page=4]{OrBord.pdf}}%
    \put(0.45341526,0.01411813){\color[rgb]{0,0,0}\makebox(0,0)[lt]{\lineheight{1.25}\smash{\begin{tabular}[t]{l}$\t = 0$\end{tabular}}}}%
  \end{picture}%
\endgroup%

%% file: Dessins/Wrapped_dga_foliation2.pdf_tex
\begingroup%
  \makeatletter%
  \providecommand\color[2][]{%
    \errmessage{(Inkscape) Color is used for the text in Inkscape, but the package 'color.sty' is not loaded}%
    \renewcommand\color[2][]{}%
  }%
  \providecommand\transparent[1]{%
    \errmessage{(Inkscape) Transparency is used (non-zero) for the text in Inkscape, but the package 'transparent.sty' is not loaded}%
    \renewcommand\transparent[1]{}%
  }%
  \providecommand\rotatebox[2]{#2}%
  \newcommand*\fsize{\dimexpr\f@size pt\relax}%
  \newcommand*\lineheight[1]{\fontsize{\fsize}{#1\fsize}\selectfont}%
  \ifx\svgwidth\undefined%
    \setlength{\unitlength}{228.68954636bp}%
    \ifx\svgscale\undefined%
      \relax%
    \else%
      \setlength{\unitlength}{\unitlength * \real{\svgscale}}%
    \fi%
  \else%
    \setlength{\unitlength}{\svgwidth}%
  \fi%
  \global\let\svgwidth\undefined%
  \global\let\svgscale\undefined%
  \makeatother%
  \begin{picture}(1,0.49908633)%
    \lineheight{1}%
    \setlength\tabcolsep{0pt}%
    \put(0,0){\includegraphics[width=\unitlength,page=1]{Wrapped_dga_foliation2.pdf}}%
    \put(0.13606205,0.20041642){\color[rgb]{0,1,0}\makebox(0,0)[lt]{\lineheight{1.25}\smash{\begin{tabular}[t]{l}$\La$\end{tabular}}}}%
    \put(0,0){\includegraphics[width=\unitlength,page=2]{Wrapped_dga_foliation2.pdf}}%
    \put(0.57710182,0.27910082){\makebox(0,0)[lt]{\lineheight{1.25}\smash{\begin{tabular}[t]{l}$\check c$\end{tabular}}}}%
    \put(0.35749659,0.26697839){\makebox(0,0)[lt]{\lineheight{1.25}\smash{\begin{tabular}[t]{l}$\hat c$\end{tabular}}}}%
    \put(0.197506,0.30104524){\color[rgb]{0,0,1}\makebox(0,0)[lt]{\lineheight{1.25}\smash{\begin{tabular}[t]{l}$R$\end{tabular}}}}%
    \put(0.76798456,0.27636054){\color[rgb]{1,0,0}\makebox(0,0)[lt]{\lineheight{1.25}\smash{\begin{tabular}[t]{l}$\Ga$\end{tabular}}}}%
    \put(0.89241529,0.24093588){\makebox(0,0)[lt]{\lineheight{1.25}\smash{\begin{tabular}[t]{l}$\t$\end{tabular}}}}%
    \put(0.79787531,0.32636365){\makebox(0,0)[lt]{\lineheight{1.25}\smash{\begin{tabular}[t]{l}$t$\end{tabular}}}}%
    \put(0,0){\includegraphics[width=\unitlength,page=3]{Wrapped_dga_foliation2.pdf}}%
  \end{picture}%
\endgroup%

%% file: Dessins/Reeb_2braid.pdf_tex
\begingroup%
  \makeatletter%
  \providecommand\color[2][]{%
    \errmessage{(Inkscape) Color is used for the text in Inkscape, but the package 'color.sty' is not loaded}%
    \renewcommand\color[2][]{}%
  }%
  \providecommand\transparent[1]{%
    \errmessage{(Inkscape) Transparency is used (non-zero) for the text in Inkscape, but the package 'transparent.sty' is not loaded}%
    \renewcommand\transparent[1]{}%
  }%
  \providecommand\rotatebox[2]{#2}%
  \newcommand*\fsize{\dimexpr\f@size pt\relax}%
  \newcommand*\lineheight[1]{\fontsize{\fsize}{#1\fsize}\selectfont}%
  \ifx\svgwidth\undefined%
    \setlength{\unitlength}{243.62461601bp}%
    \ifx\svgscale\undefined%
      \relax%
    \else%
      \setlength{\unitlength}{\unitlength * \real{\svgscale}}%
    \fi%
  \else%
    \setlength{\unitlength}{\svgwidth}%
  \fi%
  \global\let\svgwidth\undefined%
  \global\let\svgscale\undefined%
  \makeatother%
  \begin{picture}(1,0.98981325)%
    \lineheight{1}%
    \setlength\tabcolsep{0pt}%
    \put(0,0){\includegraphics[width=\unitlength,page=1]{Reeb_2braid.pdf}}%
    \put(0.38414202,0.76244495){\color[rgb]{1,0,0}\makebox(0,0)[lt]{\lineheight{1.25}\smash{\begin{tabular}[t]{l}$US$\end{tabular}}}}%
    \put(-0.00613296,0.9336786){\color[rgb]{0,0,0}\makebox(0,0)[lt]{\lineheight{1.25}\smash{\begin{tabular}[t]{l}$DS$\end{tabular}}}}%
    \put(0.71117313,0.90901537){\color[rgb]{0,0,0}\makebox(0,0)[lt]{\lineheight{1.25}\smash{\begin{tabular}[t]{l}$DS$\end{tabular}}}}%
    \put(0.86070963,0.47797858){\color[rgb]{0,0,0.83921569}\makebox(0,0)[lt]{\lineheight{1.25}\smash{\begin{tabular}[t]{l}$R$\end{tabular}}}}%
    \put(0.44551629,0.31797721){\color[rgb]{0,1,0}\makebox(0,0)[lt]{\lineheight{1.25}\smash{\begin{tabular}[t]{l}$\La$\end{tabular}}}}%
    \put(0.34684713,0.51357458){\color[rgb]{0,0,0}\makebox(0,0)[lt]{\lineheight{1.25}\smash{\begin{tabular}[t]{l}$N$\end{tabular}}}}%
    \put(0,0){\includegraphics[width=\unitlength,page=2]{Reeb_2braid.pdf}}%
    \put(0.50257314,0.04504866){\color[rgb]{0,0,0}\makebox(0,0)[lt]{\lineheight{1.25}\smash{\begin{tabular}[t]{l}$\nu$\end{tabular}}}}%
    \put(0.45361765,0.1093679){\color[rgb]{0,0,0}\makebox(0,0)[lt]{\lineheight{1.25}\smash{\begin{tabular}[t]{l}$u$\end{tabular}}}}%
  \end{picture}%
\endgroup%

%% file: Dessins/maple.tresse_1-jets.pdf_tex
\begingroup%
  \makeatletter%
  \providecommand\color[2][]{%
    \errmessage{(Inkscape) Color is used for the text in Inkscape, but the package 'color.sty' is not loaded}%
    \renewcommand\color[2][]{}%
  }%
  \providecommand\transparent[1]{%
    \errmessage{(Inkscape) Transparency is used (non-zero) for the text in Inkscape, but the package 'transparent.sty' is not loaded}%
    \renewcommand\transparent[1]{}%
  }%
  \providecommand\rotatebox[2]{#2}%
  \newcommand*\fsize{\dimexpr\f@size pt\relax}%
  \newcommand*\lineheight[1]{\fontsize{\fsize}{#1\fsize}\selectfont}%
  \ifx\svgwidth\undefined%
    \setlength{\unitlength}{1732.37438965bp}%
    \ifx\svgscale\undefined%
      \relax%
    \else%
      \setlength{\unitlength}{\unitlength * \real{\svgscale}}%
    \fi%
  \else%
    \setlength{\unitlength}{\svgwidth}%
  \fi%
  \global\let\svgwidth\undefined%
  \global\let\svgscale\undefined%
  \makeatother%
  \begin{picture}(1,0.41291023)%
    \lineheight{1}%
    \setlength\tabcolsep{0pt}%
    \put(0,0){\includegraphics[width=\unitlength,page=1]{maple.tresse_1-jets.pdf}}%
    \put(0.98059164,0.14462577){\color[rgb]{0,0,0}\makebox(0,0)[lt]{\lineheight{1.25}\smash{\begin{tabular}[t]{l}$a$\end{tabular}}}}%
    \put(0.07023762,0.39294987){\color[rgb]{0,0,0}\makebox(0,0)[lt]{\lineheight{1.25}\smash{\begin{tabular}[t]{l}$\th$\end{tabular}}}}%
    \put(-0.00086248,0.35709311){\color[rgb]{0,0,0}\makebox(0,0)[lt]{\lineheight{1.25}\smash{\begin{tabular}[t]{l}$h$\end{tabular}}}}%
    \put(0,0){\includegraphics[width=\unitlength,page=2]{maple.tresse_1-jets.pdf}}%
    \put(0.11707309,0.37552907){\color[rgb]{0,0,0}\makebox(0,0)[lt]{\lineheight{1.25}\smash{\begin{tabular}[t]{l}$c^-$\end{tabular}}}}%
    \put(0.49395106,0.35215678){\color[rgb]{0,0,0}\makebox(0,0)[lt]{\lineheight{1.25}\smash{\begin{tabular}[t]{l}$c^0$\end{tabular}}}}%
    \put(0.84843064,0.37845057){\color[rgb]{0,0,0}\makebox(0,0)[lt]{\lineheight{1.25}\smash{\begin{tabular}[t]{l}$c^+$\end{tabular}}}}%
    \put(0,0){\includegraphics[width=\unitlength,page=3]{maple.tresse_1-jets.pdf}}%
  \end{picture}%
\endgroup%